\documentclass[12pt, a4paper]{amsart}
\usepackage[left=3cm, right=3cm]{geometry}
\usepackage{amsmath,amssymb,amsthm,enumerate}
\usepackage{graphicx}
\usepackage{pstricks-add}
\usepackage{mathrsfs} %for script fonts
\usepackage{xspace}
\usepackage[pagebackref=true]{hyperref}

\title[Finite and infinite quotients of discrete and indiscrete groups]{Finite and infinite quotients of \\discrete and indiscrete groups}
\author{Pierre-Emmanuel Caprace}
\thanks{P.-E.C. is a F.R.S.-FNRS senior research associate, supported in part by EPSRC grant no EP/K032208/1.}
\address{Universit\'e catholique de Louvain, IRMP, Chemin du Cyclotron 2, bte L7.01.02, 1348 Louvain-la-Neuve, Belgique}
\email{pe.caprace@uclouvain.be}
\date{May 15, 2018}

\newtheorem{thm}{Theorem}[section]

\newtheorem{prop}[thm]{Proposition}
\newtheorem{lem}[thm]{Lemma}
\newtheorem{cor}[thm]{Corollary}

\newtheorem{prob}[thm]{Problem}

\theoremstyle{definition}

\newtheorem{qu}[thm]{Question}
\newtheorem{rmk}[thm]{Remark}
\newtheorem{example}[thm]{Example}

\newcommand{\Sym}{\mathrm{Sym}}
\newcommand{\Alt}{\mathrm{Alt}}
\newcommand{\Aut}{\mathrm{Aut}}
\newcommand{\Out}{\mathrm{Out}}

\DeclareMathOperator{\Ker}{Ker}

\newcommand{\PGL}{\mathrm{PGL}}
\newcommand{\PSL}{\mathrm{PSL}}
\newcommand{\SL}{\mathrm{SL}}
\newcommand{\FF}{\mathbf{F}}
\newcommand{\QQ}{\mathbf{Q}}
\newcommand{\ZZ}{\mathbf{Z}}
\newcommand{\Isom}{\mathrm{Isom}}

\newcommand{\Comm}{\mathrm{Comm}}
\newcommand{\Cay}{\mathrm{Cay}}
\newcommand{\Nrd}{\mathrm{Nrd}}
\begin{document}

\begin{abstract} 
These notes are devoted to lattices in products of trees and related topics. They provide an introduction to the construction, by M.~Burger and S.~Mozes, of examples of such lattices  that are simple as abstract groups. Two features of that construction are emphasized: the relevance of non-discrete locally compact groups, and the two-step strategy in the proof of simplicity, addressing separately, and with completely different methods, the existence of finite and  infinite quotients. A brief history of the quest for finitely generated and finitely presented infinite simple groups is also sketched. A comparison with Margulis' proof of Kneser's simplicity conjecture is discussed, and  the relevance of the Classification of the Finite Simple Groups is pointed out. A final chapter is devoted to finite and infinite quotients of hyperbolic groups and their relation to the asymptotic properties of the finite simple groups. Numerous open problems are discussed  along the way. 
\end{abstract}

\maketitle

\tableofcontents

%\section{Introduction}

\section{Just-infiniteness versus SQ-universality}
\bigskip

\begin{flushright}
\begin{minipage}{0.8\linewidth}
\itshape
The true infinite is both finite and infinite: it is the overcoming of both infiniteness and finiteness. It is therefore pure indeterminacy and pure freedom. 

\smallskip
\hfill \upshape Carlos Alberto Blanco, \emph{Philosophy and salvation}, 2012
\end{minipage}
\end{flushright}

\bigskip

The general theme of these notes is the normal subgroup structure of infinite groups. Two opposite extreme behaviors  can be observed: groups with `few' normal subgroups, like the \textit{simple groups}, on  one hand, and groups with `many' normal subgroups, like the \textit{free groups}, on the other hand.  The following concepts provide a possible formal way to isolate two precise classes of groups corresponding to those two extremes. A group is called \textbf{just-infinite} if it is infinite and all its proper quotients are finite. In other words, a just-infinite group is an infinite group all of whose non-trivial normal subgroups are of finite index. Infinite simple groups are obvious examples; other basic examples are the infinite cyclic group and the infinite dihedral group. At the other extreme, a group $G$ is called \textbf{SQ-universal} if every countable group embeds as a subgroup in some quotient of $G$. That term was coined by P.~Neumann \cite{Neumann73} following a suggestion of G.~Higman. An emblematic result, due to Higman--Neumann--Neumann \cite{HNN}, is that the free group of rank~$2$ is SQ-universal.  

It is important to observe that these two properties are mutually exclusive for a countable group. Indeed, if a just-infinite countable group were SQ-universal, then it would contain an isomorphic copy of every  finitely generated infinite group. This is impossible, since a countable group contains countably many finite subsets, while the $2$-generated groups fall into uncountably many isomorphism classes by a classical result of B.~Neumann \cite[Theorem~14]{Neumann37}. %These contrasting properties constitute a unifying theme of these notes. as opposite extremes. This There is of course a wide range of other possible kinds of  normal subgroup structure for countable groups. 

Just-infinite groups arise naturally in view of the following basic consequence of Zorn's lemma: \textit{every finitely generated infinite group has a just-infinite quotient}. An important theorem describing the structure of just-infinite groups was established by J.~Wilson in his thesis, see \cite{Wilson71},  \cite{Wilson_JustInfinite} and \cite{Grigorchuk_branch}. His results highlight the importance of the subclass consisting of the so-called \textbf{hereditarily just-infinite groups}, i.e. the infinite groups all of whose finite index subgroups are just-infinite. The wreath product $S \wr C_2$ of an infinite simple group $S$ with a cyclic group of order~$2$ is an example of a just-infinite group which is not hereditarily so. Let us record the following elementary fact. 

\begin{prop}\label{prop:hji}
Let $G$ be a hereditarily just-infinite group. Either $G$ is  residually finite, or the intersection $G^{(\infty)}$ of all its subgroups of finite index is   simple and  of finite index in $G$. 
\end{prop}
\begin{proof}
Suppose that $G$ is not residually finite. Then $G^{(\infty)}$  is a non-trivial normal subgroup of $G$. Hence $[G: G^{(\infty)}]$ is finite, so that $G^{(\infty)}$  is itself just-infinite. Any finite index subgroup of $G^{(\infty)}$ has finite index in $G$, and thus contains $G^{(\infty)}$ by the definition of $G^{(\infty)}$. Therefore, the just-infinite group $G^{(\infty)}$ has no proper subgroup of finite index. Hence, it is  a simple group. 
\end{proof}

The intersection $G^{(\infty)}$ of all finite index subgroups of $G$ is called  the \textbf{finite residual} of $G$.

Proposition~\ref{prop:hji} suggests a two-step approach to show that a group $G$ is virtually simple:   one step is to show  that $G$ is hereditarily just-infinite, the other is that $G$ is not residually finite. As we shall see in Section~\ref{sec:Kneser}, it is precisely this approach that G.~Margulis \cite{Margulis80} used in his proof of Kneser's  conjecture on the simplicity of certain anisotropic simple algebraic groups over number fields. Moreover, non-discrete locally compact groups are used in an essential way to establish the just-infiniteness, via Margulis' Normal Subgroup Theorem (see Section~\ref{sec:Margulis_NST}), while the Classification of the Finite Simple Groups has been used extensively in more recent investigations of the normal subgroup structure of anisotropic simple algebraic groups (see Section~\ref{sec:RPP}).

It is the same two-step strategy that M.~Burger and S.~Mozes  \cite{BuMo2} implemented in their celebrated construction of  simple lattices in the automorphism group of a product of two regular locally finite trees. We  give an overview of their work in Section~\ref{sec:BMW}, and present also the seminal ideas developed independently around the same time by D.~Wise \cite{Wise_phd} concerning lattices in products of trees. In particular, we give an explicit presentation   of a simple group, recently found by N.~Radu \cite{Radu_SimpleLatt}, that splits as an amalgamated free product of two free groups of~rank~$3$ over a common subgroup of index~$5$, see Section~\ref{sec:SimpleBMW}. A key idea due to Burger--Mozes, is to exploit some  geometric aspects of finite group theory, which become relevant through the concept of the \textit{local action} of a group of automorphisms of  a connected locally finite graph, see Section~\ref{sec:LocalAction}. 

The contrast between the study of   finite and  infinite quotients of a finitely presented group is further illustrated in the final section, devoted to Gromov hyperbolic groups. While those are known to be either SQ-universal or virtually cyclic in view of a result independently established  by T.~Deltzant~\cite{Delzant96} and A.~Olshanskii \cite{Olsh_SQ}, it is a major open problem to determine whether they are all residually finite. We discuss that problem, and its relation to the asymptotic properties of the finite simple groups, notably by the consideration of the space of marked groups.

\section{Finitely generated infinite simple groups: historical landmarks}

\subsection{The  first existence proof, after G.~Higman}

The question of existence of a finitely generated infinite simple group was asked by A.~Kurosh \cite{Kurosh}  in 1944. A positive answer was given by G.~Higman \cite{Higman} a few years later. It is based on the following observation. 

\begin{lem}[G.~Higman \cite{Higman}] \label{lem:Higman}
	\begin{enumerate}[(i)]
		\item 	The only positive integer $n$ such that $n$ divides $2^n-1$ is $n=1$. 
		
		\item In a finite group $G$, the only element $g$ that is conjugate to its square by an element of the same order as $g$ is the trivial one. 
	\end{enumerate}	
\end{lem}
\begin{proof}
	(i). Suppose for a contradiction that $n>1$ divides $2^n-1$ and let $p$ be the smallest prime divisor of $n$. Then $p$ divides $2^n-1$. Hence $p$ is odd and the order of~$2$ in the multiplicative group $~\mathbf F_p^*$ is a divisor of $n$, as well as a divisor of~$p-1$. By the choice of $p$, the integer $n$ is relatively prime to $p-1$, and we get a contradiction. 
	
	\medskip \noindent
	(ii). Let $x \in G$ be an element of the same order as $g \in G$, say $n$, such that $xgx^{-1} = g^2$. Then $g = x^n g x^{-n} = g^{2^n}$, so that $g^{2^n-1}=1$. The conclusion follows from (i). 
\end{proof}

\begin{thm}[G.~Higman \cite{Higman}]\label{thm:Higman}
 The group 
 $$H = \langle a_0 , a_1, a_2, a_3 \mid a_i a_{i+1} a_i^{-1} = a_{i+1}^2, \ i \!\!\!\mod 4\rangle$$ 
 is infinite and its only finite quotient is the trivial one. In particular $H$ has a finitely generated infinite simple quotient. 
\end{thm}
\begin{proof}
To see that $H$ is infinite, one observes that it can be decomposed as a non-trivial amalgamated free product whose factors are themselves non-trivial amalgamated free products of proper HNN-extensions. 

In order to elucidate the finite quotients of $H$, let us observe that $H$ has an automorphism of order~$4$ that cyclically permutes the generators. Therefore, if $H$ has a non-trivial finite quotient, then $H$ also has a non-trivial finite quotient  in which  the order of the image of $a_i$ is the same integer $n$ for all $i$. Since $ a_i$   conjugates $ a_{i+1}$ on its square, we deduce from Lemma~\ref{lem:Higman} that the image of $a_i$ in the finite quotient in question is trivial for all $i$, which is absurd. This shows that the only finite quotient of $H$ is the trivial one. 

We conclude the proof by observing that every finitely generated group has a maximal normal subgroup by Zorn's lemma, hence a finitely generated simple  quotient. 
\end{proof}

\begin{rmk}\label{rem:Baumslag}
As observed by G.~Baumslag \cite{Baumslag}, Lemma~\ref{lem:Higman} also implies that the one-relator group
$$B = \langle a, b \mid bab^{-1} a ba^{-1}b^{-1} = a^2\rangle,$$ 
which is infinite and non-cyclic by Magnus' Freiheitssatz, has all its finite quotients cyclic. 
\end{rmk}

The Higman group $H$ itself is far from simple: indeed, it is  {SQ-universal} by \cite[Corollary~3]{Schupp71}. The interest for the Higman group remains vivid in contemporary research: we refer to A.~Martin's recent work \cite{Martin_Higman} for a description of its intriguing geometric features.

\subsection{The first explicit family, after R.~Camm}

Shortly after Higman's theorem~\ref{thm:Higman} was published, R.~Camm obtained the first explicit example of a finitely generated infinite simple group. Her construction actually yields uncountably many pairwise non-isomorphic such groups.

\begin{thm}[R.~Camm \cite{Camm}] \label{thm:Camm}
There is an explicit uncountable set $R$ of triples of permutations of the non-zero integers such that for all $(\rho, \sigma, \tau) \in R$, the group
$$C = \langle a, p, b , q \mid a^i p^{\rho(i)} =  b^{\tau(i)} q^{\sigma\tau(i)},  \ i \in \mathbf Z \setminus\{0\} \rangle$$
is an infinite torsion-free simple group. 
\end{thm}	 

The fact that $C$ is infinite and torsion-free is clear since $C$ is a  free  amalgamated   product of two free groups of rank~$2$ over a common infinitely generated subgroup. In particular $C$ is not finitely presentable (see \cite[Corollary~12]{Neumann37}).

\subsection{The first finitely presented infinite simple group, after R.~Thompson}\label{sec:Thompson}

In an unpublished manuscript that  was circulated in 1965, R. J.~Thompson introduced three finitely generated infinite groups, known as \textbf{Thompson's groups} $F$, $T$ and $V$. Thompson showed  that $T$ and $V$ are finitely presented and simple; they were the first known specimen of that kind, see \cite{CFP}. G.~Higman \cite{Higman_Thompson} showed that Thompson's group $V$  is a member of an infinite family of finitely presented infinite simple groups, known as the \textbf{Higman--Thompson groups}. The following result records K.~Brown's presentation of Thompson's group $V$ as the fundamental group of a triangle of finite groups (we refer to \cite[\S II.12]{BH} for the basic theory of simple complexes of groups). Other short presentations of $V$ were also found more recently by Bleak--Quick \cite{BleakQuick}. 

%\begin{thm}[{K.~Brown \cite[Theorem~3]{Brown}}]\label{thm:Brown}
%The infinite simple Thompson group $V$ is isomorphic to the fundamental group $\Gamma$ of a triangle  $\mathcal T$  of finite groups with vertices $a, b, c$, defined as follows. 
%
%The vertex groups $\Gamma_a, \Gamma_b$ and $\Gamma_c$ are isomorphic to $\Sym(5), \Sym(6)$ and $\Sym(7)$ respectively. The edge groups $\Gamma_{ab}$, $\Gamma_{bc}$ and $\Gamma_{ca}$ are isomorphic to $\Sym(4)$, $\Sym(5)$ and $\Sym(3) \times C_2$ respectively, where $C_2$ is the cyclic group of order~$2$. The face group $\Gamma_{abc}$ is trivial. In order to define the various embeddings, we view $\Sym(n)$ as a permutation group on $\{1, \dots, n\}$. The edge groups $\Gamma_{ab}$ and $\Gamma_{ca}$ are embedded into $\Gamma_a = \Sym(5)$ as the stabilizers of  the point $5$ and the pair $\{4, 5\}$ respectively. The edge groups $\Gamma_{ab}$ and $\Gamma_{bc}$ are embedded into $\Gamma_b = \Sym(6)$ as the pointwise stabilizers of  the pair $\{5, 6\}$ and the point  $4$ respectively. The edge groups $\Gamma_{bc}$ and $\Gamma_{ca}$ are embedded into $\Gamma_c = \Sym(7)$ as the pointwise stabilizers of  the pair $\{4, 5\}$ and the subgroup generated by the pointwise stabilizer of $\{4, 5, 6, 7\}$ together with the bitransposition $(4, 5)(6, 7)$ respectively. 
%\end{thm}

\begin{thm}[{K.~Brown \cite[Theorem~3]{Brown}}]\label{thm:Brown}
The infinite simple Thompson group $V$ is isomorphic to the fundamental group of a triangle  of finite groups, whose vertex groups are isomorphic to $\Sym(5), \Sym(6)$ and $\Sym(7)$ respectively. 
\end{thm}

Triangles of finite groups constitute  the second simplest kind of amalgams of finite groups after graphs of   finite groups. Fundamental groups of the latter are all virtually free groups, hence residually finite. That the fundamental group of a triangle of finite groups may fail to be residually finite is rather striking. As observed by K.~Brown \cite{Brown}, his Theorem~\ref{thm:Brown} provided a negative answer to a question of B.~H.~and H.~Neumann \cite{NN53}. 

It is also natural to ask whether the fundamental group of a triangle of finite groups must be residually finite, or whether it can be simple, under the additional hypothesis that it has \textbf{non-positive curvature} in the sense of \cite[\S II.12]{BH}. We emphasize that Brown's triangle of groups from Theorem~\ref{thm:Brown} fails to satisfy that curvature hypothesis (see the Remark following Theorem~3 in \cite{Brown}); in fact Thompson's group $V$ contains a copy of every finite group, and therefore cannot have a proper cocompact action on a CAT($0$) space. The existence of a non-positively curved triangle of finite groups with a non-residually finite fundamental group has been proved by Hsu--Wise \cite{HsuWise}. However, the following question remains open. 

\begin{prob}\label{prob:triangle}
Can the fundamental group of a non-positively curved triangle of finite groups be simple? 
\end{prob}

Explicit candidates for a positive answer may be found among the $48$ triangles of finite groups discussed by J.~Tits in his contribution to the proceedings of the Groups St Andrews 1985 conference, \cite[\S3.1]{Tits85}. Those can be constructed as follows. Up to isomorphism, there are exactly four triangles of groups whose vertex groups are Frobenius groups of order~$21$, and whose edge groups are cyclic of order~$3$; this was first observed by M.~Ronan  \cite[Theorem~2.5]{Ronan}. Similarly, there are exactly $44$ isomorphism classes of triangles of finite groups with edge groups $C_9$ cyclic  of order~$9$ and vertex groups isomorphic to the Frobenius group $C_{73}\rtimes C_9$, see \cite[\S3.1]{Tits85} and \cite[\S3.2 in Cours 1984--85]{Tits_cours}.   The fundamental group of each of them   acts simply transitively on the  chambers of a Euclidean building of type~$\tilde A_2$. Therefore, all $48$ of them are hereditarily just-infinite by an unpublished theorem of Y.~Shalom and T.~Steger. A couple of them are linear (in characteristic~$0$ and~$2$ respectively) by \cite{KMW84, KMW85}, while most  do not admit any finite-dimensional representation over any commutative unital ring by  \cite[Theorem~1.1 and \S 1.3]{BCL}. It is conjectured that all of those non-linear groups fail to be residually finite, and are thus virtually simple by Proposition~\ref{prop:hji}. However, none of the $4$ smaller triangles have a simple fundamental group by \cite{KMW_triangle}, but  it is possible that the fundamental group of one of the $44$ larger ones is so.  Indeed, among those $44$ groups, $22$ are perfect. Computer experiments led by Stefan Witzel with MAGMA showed that one of   these $22$ groups has $\PSL_2(\FF_{73})$ as a quotient, while none of the other $21$ groups  admits any finite simple quotient of order~$< 2.10^{10}$.

\subsection{Quotients of free amalgamated products of free groups}

The knowledge of   finitely generated simple amalgams of free groups  on one hand (see Theorem~\ref{thm:Camm}), and of finitely presented infinite simple groups on the other hand (see Theorem~\ref{thm:Brown}) led P.~Neumann to ask the following question in 1973. 

\begin{qu}[{P.~Neumann \cite[Problem in \S3]{Neumann73}}]\label{ques:Neumann}
Can an amalgamated free product $G = A *_C B$ with $A, B$ finitely generated free and $C$ of finite index in $A$ and $B$, be a simple group? Or is it always SQ-universal? 
\end{qu}	

Notice that an amalgamated free product of that kind is finitely presented. The necessity of imposing that the index of $C$ be finite in $A$ and $B$ comes from the following (see also \cite{MinasyanOsin} for  more general results on the SQ-universality of groups acting on trees). 

\begin{thm}[{\cite[Corollary~2]{Schupp71}}]\label{thm:Schupp}
Let  $G = A *_C B$ be an amalgamated free product of groups $A, B, C$, where $A$ and $C$ are finitely generated free groups.   If  $[A: C]$ is infinite and $B \neq C$, then $G$ is SQ-universal. 
\end{thm}

P.~Neumann's question generated a substantial amount of research, and was eventually answered   by M.~Burger and S.~Mozes \cite{BuMo2}, see Section~\ref{sec:BMW} below. Before presenting their solution, we mention the following  result  of M.~Bhattacharjee, providing examples of finitely presented amalgamated products of free groups   whose only finite quotient is the trivial one. 

\begin{thm}[{M.~Bhattacharjee \cite[Theorem~8]{Bhatt}}]\label{thm:Bhatt}
There exists  an amalgamated free product $G = A *_C B$, where $A$ and $B$ are free groups of rank~$3$ and $C$ is a free group of rank~$13$ embedded as a subgroup of index~$6$ in $A$ and $B$, whose only finite quotient is the trivial one. 
\end{thm}

M.~Battacharjee provides a very explicit presentation of that amalgam. At the time of this writing, it is still unknown whether Bhattacharjee's group is simple or not, or even if the amalgam is faithful\footnote{We anticipate on the terminology that will be introducted in Section~\ref{sec:BMW} to make the following comment. Among the defining relations of the group considered by Battacharjee, there is a relation specifying that one of the generators is conjugate to its square. Since the group defined by Battacharjee's presentation is torsion-free, that relation prevents the group from acting   properly cocompactly on a CAT($0$), since an element conjugate to its square must have zero translation length, hence be torsion by \cite[Proposition~II.6.10(2)]{BH}. In particular the Battacharjee group is not commensurable to any BMW-group.} (i.e. the $G$-action on the Bass--Serre tree of the amalgam is faithful).  In order to control infinite quotients of certain free amalgamated products of free groups, M.~Burger and S.~Mozes elaborated on ideas originally developed by G.~Margulis in his seminal work on discrete subgroups of semi-simple Lie and algebraic groups. We shall now give a brief introduction to the relevant material, based on the discussion of specific families of examples.

\section{Kneser's simplicity conjecture}\label{sec:Kneser}

\subsection{The multiplicative group of the  Hamiltonian quaternions}

Given a commutative domain $R$ of characteristic~$\neq 2$, we denote by 
$$\mathbf H(R) = \{x_0 + x_1 i +x_2 j + x_3 k \mid x_i \in R\}$$
the ring of Hamiltonian quaternions over $R$. The symbols $i, j, k$ are subjected to the usual multiplication rules: $i^2 = j^2 = k^2 = -1$ and $ij =k = -ji$. The \textbf{conjugate} of a quaternion $x = x_0 + x_1 i + x_2 j + x_3 k $ is defined as $\bar x = x_0 - x_1 i- x_2 j - x_3 k $ and the \textbf{reduced norm} of $x$ is the product $\Nrd(x) =  x \bar x$, which is an element of $R$. 

Over the field of real numbers $\mathbf R$, we get the standard Hamiltonian quaternions, which is a division algebra over $\mathbf R$. Similarly $\mathbf H(K)$ is a division algebra for any subfield $K \subset \mathbf R$. We denote its multiplicative group by  $\mathbf H(K)^*$. It is easy to check that the center of the group $\mathbf H(K)^*$ consists of the non-zero scalars $x_0 \in K^*$. It can also be seen that the quotient group $\mathbf H(K)^*/K^*$ is infinitely generated for any subfield $K \subset \mathbf R$; this will become apparent later on. 

Observe that the reduced norm $\Nrd \colon \mathbf H(K)^* \to   K^*$ is multiplicative. Its kernel is  denoted by $\mathrm{SL}_1 (\mathbf H(K))$.    The center of $\mathrm{SL}_1(\mathbf H(K))$ is the cyclic group $\{\pm 1\}$ of order~$2$, and the embedding of $\mathrm{SL}_1(\mathbf H(K))$ in $\mathbf H(K)^*$ induces an injective homomorphism  
$$ \mathrm{SL}_1(\mathbf H(K)) /\{\pm 1\} \to \mathbf H(K)^*/K^*.$$
To describe its image, observe that the reduced norm descends to a homomorphism
$$\nu \colon \mathbf H(K)^*/K^* \to K^* /(K^*)^2$$
whose target is an abelian group of exponent~$2$. The kernel of $\nu$ coincides with the image of the embedding $ \mathrm{SL}_1(\mathbf H(K)) /\{\pm 1\} \to \mathbf H(K)^*/K^*$. It consists of those cosets $xK^*$ represented by a non-zero quaternion $x$ whose reduced norm $\Nrd(x)$ is a square in $K^*$.

We recall that over the complex numbers, the algebra $\mathbf H(\mathbf C)$ is isomorphic to the matrix algebra $M_2(\mathbf C)$; an isomorphism  is given by the map
$$x_0 + x_1 i +x_2 j + x_3 k  \mapsto 
\left(
\begin{array}{cc}
x_0 + x_1 i & x_2 + x_3 i\\
-x_2 + x_3 i & x_0 - x_1 i
\end{array}\right).$$ 
In particular, for any subfield $K \subset \mathbf R$, the group 
$\mathbf H(K)^*/K^*$ embeds as a subgroup of $\PGL_2(\mathbf C)$. This map identifies $\mathbf H(\mathbf R)^*/\mathbf R^*$ with the compact Lie group $\mathrm{PU(2)}$. Since the reduced norm of a   quaternion $x \in \mathbf H(\mathbf R)^*$ is a positive real number, it is a square in $\mathbf R^*$, hence the embedding $ \mathrm{SL}_1(\mathbf H(\mathbf R)) /\{\pm 1\}  \to \mathbf H(\mathbf R)^*/\mathbf R^*  $ is surjective. This yields  the  isomorphism  $\mathrm{PSU}(2) \cong \mathrm{PU(2)}$.  Recall moreover that the latter group is simple. 

What is the normal subgroup structure of the group $ \mathrm{SL}_1(\mathbf H(K)) /\{\pm 1\}$ for other subfields $K \subset \mathbf R$? 
Our goal in this section is to provide a partial answer to that question by discussing the following facts, which are special cases of a result due to G.~Margulis. 

\begin{thm}[{G.~Margulis \cite{Margulis80}}]\label{thm:Margulis80} Let $K \subset \mathbf R$ be a number field. 
\begin{enumerate}[(i)]
	\item The group   $\mathrm{SL}_1(\mathbf H(K)) /\{\pm 1\}$ is hereditarily just-infinite. 
	
	\item If for every non-archimedean local field $k$ containing $K$ as a subfield, the algebra $\mathbf H(k)$ is isomorphic to   $M_2(k)$, then the group    $\mathrm{SL}_1(\mathbf H(K)) /\{\pm 1\}$ is simple; in particular the derived group of $\mathbf H(K)^*/K^*$ is simple. Otherwise it is residually finite. 

\end{enumerate}
\end{thm}

Theorem~\ref{thm:Margulis80}(ii) is a special case of a theorem from  \cite{Margulis80}, whose scope encompasses the reduced norm~$1$ group $\mathrm{SL}_1(D)/\{\pm 1\}$ for all  quaternion division algebras $D$ over an arbitrary number field. That simplicity statement was known as \textbf{Kneser's conjecture}, in reference to Kneser's remark following Satz~C in \cite{Kneser56}.  We refer the reader to \cite[Chapter~9]{PlatonovRapinchuk} for a detailed account on this fascinating subject; see also Section~\ref{sec:RPP} below for a brief discussion of the tremendous amount of research that it triggered.  

The point we would like to make here is that Margulis' proof of  the simplicity part in Theorem~\ref{thm:Margulis80}(ii) relies on (i) in an essential way. In other words, the proof of simplicity is achieved in two steps: the first is to exclude non-trivial   infinite quotients, and was achieved by Margulis using his celebrated Normal Subgroup Theorem \cite{Margulis79}, to which the next section is devoted. The second step is to analyze   the   finite quotients. As indicated by the statement, the nature of those finite quotients happens to depend on the arithmetic properties of $K$.

Let us illustrate this matter of facts by  concrete examples of number fields. To that end, let us recall that for any local field $k$ of residue characteristic $p \neq 2$, the Hamiltonian quaternion algebra $\mathbf H(k)$ is isomorphic to the matrix algebra $M_2(k)$, see \cite[Proposition~9.14]{JacobsonII}. This is however not true if $k$ has residue characteristic~$2$: indeed,  it can be checked that $\mathbf H(\mathbf Q_2)$ is a division algebra, see \cite[Exercise~5 in Chapter~9]{JacobsonII}.

Now,  if $\mathbf H(k)$ is a division algebra, then the reduced norm~$1$ group $ \mathrm{SL}_1(\mathbf H(k))$ is compact when endowed with the topology induced by the ultrametric topology on $k$ (this   can checked explicitly; for a more general fact see \cite[Theorem~3.1]{PlatonovRapinchuk}). Moreover it is totally disconnected, because the topology of $k$ is so. In particular it is a profinite group. If the number field $K$ embeds in such a local field $k$, then  $\mathrm{SL}_1(\mathbf H(K)) /\{\pm 1\}$ embeds in the profinite group $ \mathrm{SL}_1(\mathbf H(k)) /\{\pm 1\} $, and is thus residually finite. This is in particular the case for $K = \mathbf Q$ since $\mathbf H(\mathbf Q_2)$ is a division algebra.  

On the other hand $\sqrt 2$ is not an element of $\mathbf Q_2$, since otherwise its $2$-adic valuation  would be $1/ 2$, which is impossible because the $2$-adic valuation of any element of  $\mathbf Q_2$ is an integer. Thus $\mathbf Q_2(\sqrt 2)$ is a quadratic extension of $\mathbf Q_2$. Therefore it follows from \cite[Proposition~9.13]{JacobsonII} that $\mathbf H(\mathbf Q_2(\sqrt 2))$ is isomorphic to the matrix algebra $M_2(\mathbf Q_2(\sqrt 2))$. Any local field  $k$ of residue characteristic~$2$ that contains   a copy of $K = \mathbf Q(\sqrt 2)$ also contains $\mathbf Q_2(\sqrt 2)$ (namely $\mathbf Q_2(\sqrt 2)$ is the closure of $K$ in $k$). We deduce that $K = \mathbf Q(\sqrt 2)$ satisfies the condition of Theorem~\ref{thm:Margulis80}(ii), so that $\mathrm{SL}_1(\mathbf H(\mathbf Q(\sqrt 2)) /\{\pm 1\}$ is simple. 

For details on the proof of Theorem~\ref{thm:Margulis80}(ii), we refer to \cite[Chapter~9]{PlatonovRapinchuk}. We shall now emphasize the relevance of locally compact groups in the proof of Theorem~\ref{thm:Margulis80}(i).  

\subsection{The Margulis Normal Subgroup Theorem}\label{sec:Margulis_NST}

A \textbf{lattice}  in a locally compact group $G$ is a discrete subgroup $\Gamma  \leq G$ such that the quotient space $G/\Gamma$ carries a $G$-invariant probability measure. For an excellent treatment of the basic theory of lattices, we refer to \cite{Raghu}. 
A detailed exposition of the following fundamental result, first established in \cite{Margulis79}, may be consulted in \cite[\S IV.4 and \S IX.5]{Margulis_book}. Given a product group $G_1 \times \dots \times G_n$ and a subset $A \subseteq \{1, \dots, n\}$ of the index set, we denote by $G_A$ the subproduct $G_A = \prod_{i \in A} G_i$, that we identify in the natural way with a direct factor of $G_1 \times \dots \times G_n$.

\begin{thm}[{Margulis Normal Subgroup Theorem \cite[Theorem~(4) in the Introduction]{Margulis_book}}]\label{thm:MargulisNST}
Let $n \geq 1$ and for each $i =1 , \dots, n$, let $k_i$ be a non-discrete locally compact field, let $ \mathbf G_i$ be an almost $k_i$-simple algebraic group over $k_i$, let  $r_i$ be the   $k_i$-rank of $\mathbf G_i$, and assume that $r_i >0$.  Let $G_i$ be the quotient of $\mathbf G_i(k_i)$ by its center. 

Let $\Gamma < G_1 \times \dots \times G_n$   be a lattice. Assume that for every partition $\{1, \dots, n\} = A \cup B$ with $A \neq \varnothing \neq B$, the product group  $(\Gamma \cap G_A)(\Gamma\cap G_B)$ is of infinite index in $\Gamma$. If  $\sum_{i=1}^n r_i \geq 2$, then $\Gamma$ is hereditarily just-infinite. 
\end{thm}

The hypothesis that  $(\Gamma \cap G_A)(\Gamma\cap G_B)$ is of infinite index in $\Gamma$ expresses the fact that $\Gamma$ is an \textit{irreducible} lattice. That condition is obviously necessary for $\Gamma$ to be hereditarily just-infinite, since the intersections $\Gamma \cap G_A$ and $\Gamma \cap G_B$ are both normal subgroups of $\Gamma$. 
The hypothesis on the rank of $ G_1 \times \dots \times G_n$ is also necessary: indeed lattices in rank~$1$ simple Lie groups are never just-infinite by \cite{Gromov81}. Actually  they are all  SQ-universal, see Theorem~\ref{thm:LatticeInHyp} below. 

Let us illustrate Theorem~\ref{thm:MargulisNST} with a specific example related to the previous section. We retain the notation introduced there. 

Given a finite set of primes $S = \{\ell_1, \dots, \ell_r\}$, we consider the ring $\ZZ_S = \ZZ[\frac 1 {\ell_1\dots \ell_r}]$ consisting of those rationals $x \in \QQ$ whose $p$-adic valuation is non-negative for all primes $p \not \in S$. We also let 
$$L_S = \mathbf H(\ZZ_S)^*$$
be the multiplicative  group of units of the ring $\mathbf H(\ZZ_S)$. Since $x^{-1} = \frac{\bar x}{\Nrd(x)}$ for all $x \in \mathbf H(\QQ)^*$, we see that an element $x \in \mathbf H(\ZZ_S)$ is a unit if and only if its reduced norm $\Nrd(x)$ is a unit of the ring $\ZZ_S$. 

As in the previous section, the prime $2$ and the odd primes play a different role. 

We first consider $p=2$. Clearly, any $x \in L_{\{2\}}$ can be written as $2^n y$, where $n \in \ZZ$ and $y \in \mathbf H(\ZZ)$. By \cite[Lemma~2.5.5]{DSV}, any quaternion $y \in \mathbf H(\ZZ)$ can be written as a product $y = 2^m \pi \varepsilon$ for some $m \in \mathbf N$, some $\pi \in \{1, 1+i, 1+j, 1+k, (1+i)(1+j), (1+i)(1-k)\}$ and some $\varepsilon \in \mathbf H(\ZZ)$ whose reduced norm is odd. Since $\Nrd(x) \in \ZZ[\frac 1 2]^*$, we deduce that $\varepsilon$ belongs to $L_\varnothing =  \mathbf H(\ZZ)^* = \{\pm 1, \pm i, \pm i, \pm j, \pm k \}$. It follows that the group $L_{\{2\}}/\ZZ[\frac 1 2]^*$ is finite. 
On the other hand, for any odd prime $p$, the group  $L_{\{p\}}^*/\ZZ[\frac 1 p]^*$ is infinite and finitely generated,  see \cite[Theorem~2.5.13]{DSV}. 

As mentioned in the previous section, if the prime $p$ is odd, the Hamiltonian quaternion algebra $\mathbf H(\mathbf Q_p)$ over the $p$-adic numbers is isomorphic to the matrix algebra $M_2(\mathbf Q_p)$, which yields a natural injective homomorphism $\varphi_p \colon \mathbf  H(\mathbf Q)^*/ \mathbf Q^* \to \PGL_2(\mathbf Q_p)$. 
%Let $\Gamma_{\{p\}}$ denote the image of $L_{\{p\}}\mathbf Q^*/Q^*$ under $\varphi_p$. The intersection $\Gamma_{\{p\}} \cap \PSL_2(\mathbf Z_p) = \varphi_p(L_\varnothing)$ is finite, so $\Gamma_{\{p\}}$ is a discrete subgroup of $\PSL_2(\mathbf Q_p)$. Moreover we have 
%
Given a finite set of  primes $S$ containing at least one odd prime, let $\Gamma_S$ be the image of the product homomorphism 
$$\prod_{p \in S \setminus \{2\}} \varphi_p \colon L_S  /\ZZ_S^* \to \prod_{p \in S \setminus \{2\}} \PGL_2(\mathbf Q_p) = G_S.$$ 
By \cite[Chapter~IV, Theorem~1.1]{Vigneras}, the group $\Gamma_S$ is a discrete subgroup of $G_S$ and  the quotient space $G_S/\Gamma_S$ is compact. The discreteness is due to the fact   that $\Gamma_S \cap \prod_{p \in S \setminus \{2\}} \PGL_2(\mathbf Z_p) = \Gamma_{\{2\}} \cong L_{\{2\}}/\ZZ[\frac 1 2]^*$, which  is a finite group as mentioned above. Moreover,  for each $p \in S \setminus \{2\}$, the homomorphism $\varphi_p$ maps injectively $L_S  / \ZZ_S^*$ to a  Zariski dense subgroup of $\PGL_2(\mathbf Q_p)$, so that no finite index subgroup of $L_S /\ZZ_S^*$ splits as the direct product of two infinite subgroups. 

Those facts imply that for any non-empty finite set of primes $S$ containing at least two odd primes,   the hypotheses of Theorem~\ref{thm:MargulisNST} are satisfied, so that the group $\Gamma_S$ is hereditarily just-infinite. 

Since the group $\mathbf H(\mathbf Q)^*/ \mathbf Q^*$ is the directed union of the collection of subgroups $L_S/\ZZ_S^*$ indexed by the finite sets of primes $S$, it readily follows from Theorem~\ref{thm:MargulisNST} that every proper quotient of $\mathbf H(\mathbf Q)^*/ \mathbf Q^*$ is a \textbf{locally finite} group, i.e. a group in which every finitely generated subgroup is finite. Similarly, every proper quotient of the subgroup $\SL_1(\mathbf H(\mathbf Q))/\{\pm 1\}$ is locally finite. 
In fact, as observed by Margulis \cite{Margulis79}, using Strong Approximation one can  show via Theorem~\ref{thm:MargulisNST} that $\SL_1(\mathbf H(\mathbf Q))/\{\pm 1\}$ is hereditarily just-infinite (see \cite[p.~517]{PlatonovRapinchuk}). This is how the  proof of Theorem~\ref{thm:Margulis80}(i) is completed in the case of $K = \mathbf Q$.  

\begin{rmk}\label{rem:ArithmeticAmalgams}
If $S= \{p, \ell \}$ is a set consisting of two distinct odd primes, then the group $\Gamma_S$ is a cocompact lattice with dense projections in $\PGL_2(\mathbf Q_p)  \times \PGL_2(\mathbf Q_\ell)$. Using the actions of $\PGL_2(\mathbf Q_p) $ and $ \PGL_2(\mathbf Q_\ell)$ on their associated Bruhat--Tits trees, one can show that $\Gamma_S$ has a finite index subgroup which splits as an amalgamated free product $A *_C B$ where $A$ and $B$ are finitely generated free groups and $C$ is of finite index in $A$ and $B$. We refer to Proposition~\ref{prop:Rattaggi}  below for a concrete example with $S = \{3, 5\}$.  Since $\Gamma_S$ is hereditarily just-infinite, this shows that an amalgam as in Question~\ref{ques:Neumann} can be hereditarily just-infinite; in particular it may fail to be SQ-universal. 	
\end{rmk}
 
\subsection{Finite quotients of the  multiplicative group of a division algebra}\label{sec:RPP}

Kneser's conjecture has been successively generalized by V.~Platonov \cite{Platonov74} and G.~Margulis \cite{Margulis79}; its most general form, formulated in \cite[\S 2.4.8]{Margulis79}, is known as the \textbf{Margulis--Platonov conjecture}. That conjecture  describes the normal subgroup structure of the group of rational points of a simply connected simple algebraic group over a number field. Although the conjecture is still open in full generality, the special case of the reduced norm~$1$ group $\mathrm{SL}_1(D)$ of a division algebra $D$ over a number field was settled in a series of important papers by various authors (see \cite{SegevSeitz}, \cite{Rapinchuk06} and references therein). Roughly speaking, the normal subgroup structure of $\mathrm{SL}_1(D)$  is elucidated by a similar scheme as in the proof of Theorem~\ref{thm:Margulis80}: infinite quotients  and finite quotients are investigated separately, with completely different methods. While the treatment of infinite quotients is based on Margulis' Normal Subgroup Theorem and Strong Approximation as above,   the Classification of the Finite Simple Groups (CFSG) is used   in an essential way to investigate the finite quotients of $\mathrm{SL}_1(D)$ for all division algebras $D$ of degree~$\geq 3$ over number fields, see \cite{RapPot},  \cite{SegevSeitz} and \cite{Rapinchuk06}. We finish this section by mentioning  the following striking culmination of this direction of research, which is valid over an \textit{arbitrary} ground field. The proof relies on the CFSG via the consideration of   commuting graphs. 

\begin{thm}[{Rapinchuk--Segev--Seitz \cite{RPP}}]
	Let $D$ be  a  division algebra which is finite-dimensional over its center.   Every finite quotient of the multiplicative group $D^*$  is solvable.
\end{thm}

\section{Lattices in products of trees, after M.~Burger, S.~Mozes and D.~Wise}\label{sec:BMW}

\bigskip

\begin{flushright}
\begin{minipage}{0.6\linewidth}
\itshape
Let's start with the A-B-C of it\\
Getting right down to the X-Y-Z of it\\
Help me solve the mystery of it

\smallskip
\hfill \upshape G. De Paul, S. Cahn, \emph{Teach me tonight}, 1953
\end{minipage}
\end{flushright}

\bigskip

The goal of this section is to discuss a class of discrete groups acting properly and cocompactly on a product of regular locally finite trees. The study of those groups was pioneered by S.~Mozes~\cite{Mozes92}, Burger--Mozes \cite{BM_CRAS, BuMo1, BuMo2} and D.~Wise \cite{Wise_phd} in the mid 1990's, and provided notably an answer to P.~Neumann's Question~\ref{ques:Neumann}. 

\subsection{BMW-groups and BMW-complexes}

We recall that the \textbf{Cartesian product} of two graphs $(V_1, E_1)$ and $(V_2, E_2)$ is the graph whose vertex set is $V_1 \times V_2$ and whose edge set is defined as the collection of pairs $\big\{ \{v_1, v_2\}, \{w_1, w_2\}\big\}$ such that either $\{v_1, w_1\} \in E_1$ and $v_2 = w_2$, or  $\{v_2, w_2\} \in E_2$  and $v_1 = w_1$. 

A group $\Gamma$ is called a \textbf{BMW-group} if $\Gamma$ is capable of acting by automorphisms on  the Cartesian product of two trees, say $T$ and $T'$, so that every element of $\Gamma$ preserves the product decomposition   $T \times T'$ (i.e. no element of $\Gamma$ swaps the factors $T$ and $T'$), and that the action of $\Gamma$ on the vertex set of $T\times T'$ is free and transitive. Equivalently, $\Gamma$ is a BMW-group if it has   a finite generating set $\Sigma$ such that the Cayley graph of $(\Gamma, \Sigma)$, viewed as an undirected unlabeled graph, is isomorphic to the Cartesian product of two trees, say $T$ and $T'$, and if the image of the associated homormorphism $\Gamma \to \Aut(T \times T')$ is contained  in $\Aut(T) \times \Aut(T')$. Since $\Gamma$ acts vertex-transitively on its Cayley graphs, it is not difficult to see that $T$ and $T'$ must be regular  trees. The \textbf{degree} of a BMW-group   is the pair $(\deg T, \deg T')$, which depends a priori on the choice of the generating set $\Sigma$. 

Recall that the groups $\Gamma$ admitting a finite generating set $\Sigma$ such that the Cayley graph of $(\Gamma, \Sigma)$ is a tree, are the free products of finitely generated free groups and finitely generated free Coxeter groups. A \textbf{free Coxeter group} is a free product of cyclic groups of order~$2$. We emphasize that two non-isomorphic groups can have the same tree as a Cayley graph. This matter of fact is greatly amplified when passing from single trees to Cartesian products of two trees. Indeed, as we shall see, the various BMW-groups admitting a given product of trees as a Cayley graph can enjoy radically different algebraic properties. 

The product $C_2 \times C_2$ of two cyclic groups of order~$2$ is the only   BMW-group of degree $(1, 1)$. For all $m, n$, the direct product of the free Coxeter group of rank $m$ and the free Coxeter group of rank $n$ is a BMW-group of degree $(m, n)$. Similarly, the direct product of the free group of rank $m$ and the free group of rank $n$ is a BMW-group of degree $(2m, 2n)$. A BMW-group is called \textbf{reducible} if it has a finite index subgroup that splits as a direct product of two free groups of rank~$\geq 1$. The following example of a BMW-group was studied by D.~Wise in his thesis \cite{Wise_phd}, where it is notably proved to be irreducible. 

\begin{example}%[The Wise lattice]
 The group 
 $$\Gamma_{\mathrm{Wise}}= \langle a, b,  x, y, z \mid 
 a y a^{-1} x^{-1}, b y b^{-1} x^{-1}, a z b^{-1} z^{-1}, a x b^{-1} y^{-1}, b x a^{-1} z^{-1}, bz a^{-1} y^{-1} \rangle$$
 is an irreducible BMW-group of degree $(4, 6)$. It is called the \textbf{Wise lattice}. 
\end{example}

A key property of BMW-groups is that they can be identified by means of a presentation of a very specific form, defined as follows. 
A \textbf{BMW-presentation} is a group presentation of the form 
$$\Gamma = \langle A \cup X \mid R \rangle,$$
where $A$ and $X$ are disjoint  finite sets, and the set of relations $R$ satisfies the following two conditions:
\begin{enumerate}[(BMW1)]

	\item $R$ has a (possibly trivial) partition $R = R_2 \cup R_4$, such that  every $r \in R_2$ is of the form $r = t^2$ with $t \in A \cup X$, and every $r \in R_4$ is of the form $r = axa'x'$ with $a, a' \in   A \cup A^{-1}$ and $x, x' \in  X\cup X^{-1}$;
	
	\item For all $a \in   A \cup A^{-1}$ and $x \in  X\cup X^{-1}$, there exists a unique 
	$a' \in  A \cup A^{-1}$ and a unique $x' \in  X\cup X^{-1}$ such that $axa'x'$ or $a'x'ax$ or $a^{-1} x' a' x^{-1}$ or $a'x^{-1}a^{-1}x'$  belongs to $R_4$.
	
\end{enumerate}

To interpret correctly the uniqueness  conditions appearing in (BMW2), it is important to view the elements of $R_4$ as words in the group $ \langle A \cup X \mid R_2 \rangle$. Hence,  if $(a')^2 \in  R_2$, then $a'$ and $(a')^{-1}$ are the same element of $A \cup A^{-1}$, and similarly for $x'$.

The next result collects  basic properties of BMW-groups and  BMW-presentations.  The proof uses  basic results on  CAT($0$) groups as well as standard tools from Bass--Serre theory (some details may be found in \cite[Section~6.1]{BuMo2}, \cite[Chapter~I]{Rattaggi_phd} and \cite[\S3--4]{Radu_SimpleLatt}). 

\begin{prop}\label{prop:BMW-basic}
	Every BMW-group admits a BMW-presentation. 
	
	Conversely, let $\Gamma = \langle A \cup X \mid R\rangle$ be a BMW-presention, with $R = R_2 \cup R_4$ as above. Let   $A'  = \{a \in A\mid a^2 \in R_2\}$, $X'  =   \{x \in X\mid x^2 \in R_2\}$ and let $m = |A \setminus A'|$, $m' = |A'|$, $n = |X \setminus X'|$ and  $n' = |X'|$.
	\begin{enumerate}[(i)]
		\item The Cayley graph of $\Gamma$ with respect $A \cup X$ is isomorphic to the Cartesian product $T_A \times T_X$ of two regular trees of degree $M = 2m+m'$ and $N = 2n+n'$ respectively. In particular $\Gamma$ is a BMW-group of degree $(M, N) = (2m+m', 2n+n')$.

		\item We have $|R_4| \geq mn$. Moreover, if $R_2 = \varnothing$ and $|R_4| = mn$ then $\Gamma$ is torsion-free. 
		
		\item Every torsion-free BMW-group of degree $(2m, 2n)$ admits a BMW-presention with $|R_4|  = mn$. 
		
		\item The subgroup $\langle A  \rangle$ is the free product of a free group of rank~$m$ with a free Coxeter group of rank $m'$; it fixes  a vertex in $T_X$ and acts simply transitively on the vertices of $T_A$. 
		
		Similarly $\langle X \rangle$ is the free product of a free group of rank~$n$ with a free Coxeter group of rank $n'$; it fixes  a vertex in $T_A$ and acts simply transitively on the vertices of $T_X$. 
		
		\item $\Gamma$ has a torsion-free normal subgroup $\Gamma^+$ of index~$4$ with $\Gamma/\Gamma^+\cong C_2 \times C_2$. The group $\Gamma^+$ acts without edge-inversion on both $T_A$ and $T_X$. 
		
		\item If the $\Gamma^+$-action on $T_A$ is edge-transitive, then $\Gamma^+$ admits a decomposition as a free amalgamated product of the form $\Gamma^+\cong F_{N-1} *_{F_{MN-2M+1}} F_{N-1}$, where $F_d$ denotes the free group of rank~$d$. 
		
		Similarly, if the $\Gamma^+$-action on $T_X$ is edge-transitive, then $\Gamma^+$ admits a decomposition as a free amalgamated product of the form $\Gamma^+ \cong F_{M-1} *_{F_{MN-2N+1}} F_{M-1}$. 
		 
		\item Given a BMW-presentation $\Gamma' = \langle A' \cup X' \mid R'\rangle$ with $A \subset A'$, $X \subset X'$ and $R \subset R'$, the natural homomorphism $\Gamma \to \Gamma'$ induced by the inclusion of the generating set of $\Gamma$ is injective. Moreover the Cayley graph $\Cay(\Gamma, A \cup X)$ embeds as a $\Gamma$-invariant convex subgraph of $\Cay(\Gamma', A' \cup X')$.

	\end{enumerate}
\end{prop}	

%An example of a  BMW-presentation of a BMW-group with torsion is given in Proposition~\ref{prop:GammaRadu}  below.  In order to simplify the exposition, we will however focus mostly on BMW-groups that are torsion-free. 

The presentation $2$-complex of a torsion-free BMW-group $\Gamma$ of degree $(2m, 2n)$ with BMW-presentation $\Gamma = \langle A \cup X \mid R \rangle$ is called a \textbf{BMW-complex} of degree $(2m, 2n)$. It is a square complex $Y$ with a single vertex $v$, $m + n$ oriented edges labeled by the $m$ elements of $A$ and the $n$ elements of $X$, and $mn$ squares corresponding to the relations in $R$. The condition (BMW2)  corresponds to the geometric property that the link of $Y$ at $v$ is isomorphic to the complete bipartite graph $K_{2m, 2n}$. The universal cover of $Y$ is isomorphic to the product $T_{2m} \times T_{2n}$ of the regular trees of degrees $2m$ and $2n$, viewed as a square complex. The group $\Gamma$ is isomorphic to the fundamental group $\pi_1(Y)$; it acts on $T_{2m} \times T_{2n}$ by covering transformations, and the action preserves each of the two tree factors.

\subsection{Examples of BMW-groups of small degree}
We now describe some   BMW-groups of degree $(M, N)$ for the smallest values of $M$ and $N$. We start with the torsion-free case. 

The only torsion-free BMW-groups of degree $(2, 2)$ are the free abelian group  $\mathbf Z^2 = \langle a, x \mid axa^{-1}x^{-1} \rangle$ and the Klein Bottle group $\langle a, x \mid axa^{-1} x\rangle$. 

For every $n$, all torsion-free BMW-groups of degree $(2, 2n)$ are reducible. This follows from Theorem~\ref{thm:IrredInsep} below. 

The torsion-free BMW-groups of degree $(4,4)$ have been studied by Kimberley--Robertson \cite{KR} and  D.~Rattaggi \cite{Rattaggi_phd}. As explained in Section~7 from \cite{KR}, there are exactly $52$ homeomorphism types of BMW-complexes of degree $(4,4)$. It is important to underline that two non-homeomorphic complexes can have isomorphic fundamental groups. The number of isomorphism classes of  torsion-free BMW-groups of degree $(4,4)$ is not known, but according to loc.~cit. it belongs to the set $\{41, 42, 43\}$. 
By comparing and combining Table C.4 on p.~278 in Section C.5 of \cite{Rattaggi_phd} with Section~7 from \cite{KR} (keeping an eye on the structure of the abelianization), one can see that among the  $52$ BMW-complexes of degree $(4,4)$,  at least $50$ have a reducible fundamental group. As we shall see, the remaining two complexes happen to have an irreducible fundamental group. Those admit the following BMW-presentations.

\begin{example}%[The two irreducible torsion-free BMW-groups of degree $(4,4)$]
\label{ex:4,4}
The groups 
  $$\Gamma_{\mathrm{SV}}= \langle a, b, x, y \mid axay, ax^{-1} b x^{-1}, ay^{-1} b^{-1}y^{-1}, bxby^{-1}  \rangle $$
and
   $$\Gamma_{\mathrm{JW}}= \langle a, b, x, y \mid axay, ax^{-1} b y^{-1}, ay^{-1} b^{-1}x^{-1}, bxb^{-1}y^{-1} \rangle $$
are the only two irreducible torsion-free BMW-groups of  degree $(4,4)$, up to isomorphism. 
\end{example}

The irreducibility of $\Gamma_{\mathrm{SV}}$ is a consequence of  the main results in \cite{StixVdovina}. 

\begin{prop}[J.~Stix and A.~Vdovina \cite{StixVdovina}]\label{prop:StixVdovina}
	The BMW-group $\Gamma_{\mathrm{SV}}$ is irreducible. It embeds as a cocompact lattice with dense projections in $\PGL_2(\FF_3(\!(t)\!)) \times \PGL_2(\FF_3(\!(t)\!))$. In particular it is hereditarily just-infinite by the Margulis Normal Subgroup Theorem. 
\end{prop}	
	
The fact that $\Gamma_{\mathrm{JW}}$ is irreducible is established by Janzen--Wise in \cite{JW}. Although the presentations of $\Gamma_{\mathrm{SV}}$ and $\Gamma_{\mathrm{JW}}$ are rather similar, the groups are quite different. Indeed   $\Gamma_{\mathrm{SV}}$ is linear in characteristic~$3$,  while the group $\Gamma_{\mathrm{JW}}$ fails to be residually finite (this was observed independently in \cite[Theorem~15]{BondarenkoKivva} and \cite{CapraceWesolek}; see also Section~\ref{sec:ResidualFiniteness} below).

One can pursue the enumeration of BMW-complexes of larger degrees. The Wise lattice is an example of an irreducible torsion-free BMW-group of degree $(4,6)$. As we shall see in Section~\ref{sec:ResidualFiniteness} below, it is not residually finite. Another example of degree $(4, 6)$ is provided by the following result, due to D.~Rattaggi. 

\begin{prop}[{D.~Rattaggi \cite[Theorem~3.35 and Proposition~3.47]{Rattaggi_phd}}] \label{prop:Rattaggi}
The BMW-group 
$$\Gamma_{\mathrm{Ratt}} = \langle a, b, x, y, z \mid   axby, a y b x^{-1}, a z b^{-1} x, a z^{-1} a y^{-1}, ax^{-1} b^{-1} z, bzby^{-1} \rangle$$
is irreducible.  It embeds as a cocompact lattice with dense projections in $\PGL_2(\QQ_3) \times \PGL_2(\QQ_5)$. In particular it is hereditarily just-infinite by the Margulis Normal Subgroup Theorem. 
\end{prop}

The lattice $\Gamma_{\mathrm{Ratt}}$ happens to be a quaternionic arithmetic lattice as those discussed in Section~\ref{sec:Margulis_NST}. More precisely, consider the quaternion algebra $\mathbf H(\QQ)$ with standard basis $\{1, i, j, k\}$. One can compute that the assignments:
$$\begin{array}{rcl}
\varphi(a) & = & 1 - i - j, \\
\varphi(b) & = & 1-i + j ,\\ 
\varphi(x) & = & 1 + 2k, \\ 
\varphi(y) & = & 1- 2i,\\
\varphi(z) & = &1-2j
\end{array}$$
extend to a homomorphism $\varphi \colon \Gamma_{\mathrm{Ratt}} \to D^*/\QQ^*$, by checking that the defining relations of $\Gamma_{\mathrm{Ratt}}$ are satisfied. It turns out that $\varphi$ is injective, and maps $\Gamma_{\mathrm{Ratt}}$ to a finite index subgroup of the lattice $\mathbf H(\ZZ_{\{3, 5\}})^*/\ZZ_{\{3, 5\}}^*$ discussed in Section~\ref{sec:Margulis_NST}.

\medskip 
The following example of degree $(6, 6)$ was pointed out to me by I.~Bondarenko. 

\begin{prop}[I.~Bondarenko, D.~D'Angeli, E.~Rodaro] \label{prop:Lamplighter(6,6)}
The BMW-group 
\begin{align*}
\Gamma_{\mathrm{BDR}} = \langle a, b, c, x, y, z \mid  \ & axa^{-1}x^{-1}, ayb^{-1}z^{-1}, azc^{-1}y^{-1},cyc^{-1}x^{-1},\\
&  bya^{-1}y^{-1}, bxc^{-1}z^{-1}, bzb^{-1}x^{-1}, cxb^{-1}y^{-1}, cza^{-1}z^{-1}\rangle
\end{align*}
is an irreducible torsion-free BMW-group of degree $(6,6)$. The image of the respective projections of the free groups $\langle a, b, c \rangle$ and $\langle x, y, z\rangle$ to the automorphism groups of the tree factors $T_{\{x, y, z\}}$ and $T_{\{a, b, c\}}$ are both isomorphic to the lamplighter group $C_3 \wr \ZZ$. 
\end{prop}
\begin{proof}
The fact that $\Gamma_{\mathrm{BDR}}$ is a  torsion-free BMW-group of degree $(6,6)$ is clear from Proposition~\ref{prop:BMW-basic}. The statement on the projection of $\Gamma_{\mathrm{BDR}}$ to the automorphism groups of the tree factors follows from \cite[Corollary~2.14]{GlasnerMozes} and the main result in \cite{BDR}. The irreducibility follows from the statement on the projections together with Theorem~\ref{thm:IrredInsep} below. 
\end{proof}

We now present some examples of BMW-presentations involving generators of order~$2$. As remarked in Proposition~\ref{prop:BMW-basic}, if all generators of a BMW-group $\Gamma$ in a BMW-presentation are of infinite order, then the degree of $\Gamma$ consists of a pair of even integers. Allowing the generators to be torsion gives rise to BMW-groups in odd degree.
	
\begin{prop}[{N.~Rungtanapirom \cite[Theorem~A]{Rungtana}}]\label{prop:Rungtana}
	The BMW-group 
	\begin{align*}
	\Gamma_{\mathrm{Rung}} = \langle a, b, x, y \mid a^2, x^2,  axax, ayby, bxby^{-1}\rangle
		\end{align*}
  is irreducible of degree $(3, 3)$. It embeds as a cocompact lattice with dense projections in $\PGL_2(\FF_2(\!(t)\!)) \times \PGL_2(\FF_2(\!(t)\!))$. In particular it is hereditarily just-infinite by the Margulis Normal Subgroup Theorem. 
\end{prop}	

The following examples, respectively of degree $(3, 3)$, $(4, 5)$ and $(6,6)$, are due to N.~Radu \cite{Radu_SimpleLatt}. They all contain non-trivial torsion elements.
\begin{prop}[N.~Radu]\label{prop:GammaRadu}
The BMW-groups
 \begin{align*}
\Gamma_{3,3} = \langle a, b, c, x, y, z \mid \ &  a^2, b^2, c^2, x^2, y^2, z^2, \\
& axax, ayay, azbz, bxbx, bycy, cxcz  \rangle, 
\end{align*}
 \begin{align*}
\Gamma_{4,5} = \langle a, b, c, d, x, y, z, t, u \mid \ & a^2, b^2, c^2, d^2,x^2, y^2, z^2, t^2, u^2,\\
&  axax, ayay, azbz, bxbx, bycy, cxcz, \\
&  xtxt,  a t a t ,  a u d u ,  b t c u ,  d x d y ,  d z d t  \rangle
\end{align*}
and 
\begin{align*}
\Gamma_{6, 6} = \langle a, b, c, x, y, z \mid  \ & axay, ax^{-1} b y^{-1}, ay^{-1} b^{-1}x^{-1}, bxb^{-1}y^{-1},\\
& cx^{-1}c^{-1}x^{-1}, c^{-1} y c^{-1} y, cycz^{-1},\\
& az^{-1}a^{-1}z, bzcz, bz^{-1}bz^{-1}
\rangle
\end{align*}
are irreducible of degree $(3, 3)$, $(4, 5)$ and $(6, 6)$ respectively.  The subgroup $\langle a, b, c, x, y, z\rangle$ of $\Gamma_{4, 5}$ is isomorphic to $\Gamma_{3, 3}$; the subgroup $\langle a, b, x, y\rangle$ of $\Gamma_{6, 6}$ is isomorphic to the BMW-group $\Gamma_{\mathrm{JW}}$  from Example~\ref{ex:4,4}. 
\end{prop}
\begin{proof}
The irreducibility of $\Gamma_{3, 3}$ is established in \cite[Proposition~5.4]{Radu_SimpleLatt}; that assertion can actually be deduced from Corollary~\ref{cor:TrofimovWeiss} below. That $\Gamma_{3, 3}$ (resp. $\Gamma_{\mathrm{JW}}$) embeds naturally as a subgroup of $\Gamma_{4, 5}$ (resp. $\Gamma_{6, 6}$) follows from Proposition~\ref{prop:BMW-basic}(vii). 
The irreducibility of $\Gamma_{4, 5}$  and $\Gamma_{6,6}$ can then be deduced from the irreducibility of $\Gamma_{3, 3}$ and $\Gamma_{\mathrm{JW}}$ using Theorem~\ref{thm:IrredInsep} below. 
\end{proof}

The group $\Gamma_{4, 5}$ is denoted by $\Gamma_{4, 5;9}$ in \cite[\S 5.2]{Radu_SimpleLatt}, while $\Gamma_{6, 6}$ appears as  $\Gamma_{6, 6;2}$ in \cite[\S 5.1]{Radu_SimpleLatt}. 
We will come back to those lattices in Section~\ref{sec:SimpleBMW} below. Although the generators of $\Gamma_{6, 6}$ are of infinite order, the elements $bz^{-1}$ and $c^{-1} y \in \Gamma_{6, 6}$ have order~$2$, as follows clearly from the defining relations.

%The following example of degree $(6,6)$ that was found by N.~Radu and appears as the group denoted by $\Gamma_{6,6,2}$ in \cite{Radu_SimpleLatt}. That group contains non-trivial torsion elements; its  key algebraic property won't be revealed until Section~\ref{sec:SimpleBMW} below. 

It is important to remark that the number of homeomorphism types of BMW-complexes grows quickly with the degree. A lower bound on that number is given by an explicit formula in \cite[Formula (2.4) in \S 2.3]{StixVdovina}\footnote{The number appearing in the formula (2.4) in \S 2.3 of \cite{StixVdovina} indeed yields a lower bound on  the number of those homeomorphism types.}. With the help of a computer, N.~Radu \cite{Radu_phd} has enumerated all BMW-presentations of small degree. In particular, he has shown that there are $1001$ BMW-complexes of degree $(4,6)$. Among them, at least $890$ have a reducible fundamental group, while at least $16$ are irreducible. Moreover, there are  $32062$ BMW-complexes of degree $(6,6)$, among which at least $18426$ are reducible and at least $8227$ are irreducible, see \cite{Radu_phd}. In either case, the exact number of irreducible ones is unknown; neither is the number of those with a linear or residually finite fundamental group.  The difficulty is that there is no known necessary and sufficient  condition determining whether a BMW-group is irreducible (or linear, or residually finite, or just-infinite) that can be checked algorithmically on the BMW-presentation. Problems of that nature are recorded in \cite[Section~10]{Wise_CSC}. In the next sections, we shall discuss \emph{sufficient} (but not necessary!) conditions   that can be used to check some of those properties algorithmically.

\subsection{Inseparability and irreducibility}

A fundamental early discovery of D.~Wise \cite{Wise_phd} is that the irreducibility of a BMW-group $\Gamma = \langle A \cup X \mid R\rangle$ is related to the inseparability of the   subgroups $\langle A \rangle$ and $\langle X \rangle$ in $\Gamma$. This phenomenon was first highlighted by him in the case of the Wise lattice in \cite{Wise_phd} (see also \cite[Corollary~6.4]{Wise_CSC}). We shall present a general statement recently established in \cite[Corollary~32]{CKRW} and inspired by Wise's work \cite{Wise_CSC, Wise_fig8}. The statement requires the following terminology. 

A subgroup $H$ of a group $G$ is called \textbf{separable} if it is an intersection of finite index subgroups of $G$. Equivalently $H$ is separable if and only if for every $g \in G$, if $g \not \in H$ then there exists a finite quotient $\varphi \colon G \to Q$ with $\varphi(g) \not \in \varphi(H)$. The set of separable subgroups of $G$ is closed under intersections.  
The \textbf{profinite closure} of a subgroup $H$ in $G$, denoted by $\overline H$, is the smallest separable subgroup of $G$ containing $H$. It coincides with the closure of $H$ with respect to the profinite topology on $G$. 

A subgroup $H$ of $G$ is called \textbf{virtually normal} if it has a finite index subgroup which is normal in $G$. It is \textbf{weakly separable} if it is an intersection of virtually normal subgroups of $G$. Equivalently $H$ is weakly separable if and only if for every $g \in G$, if $g \not \in H$ then there exists a (possibly infinite)  quotient $\varphi \colon G \to Q$ such that $\varphi(H)$ is finite and $\varphi(g) \not \in \varphi(H)$. Clearly, every separable subgroup is weakly separable, but not conversely: Indeed, in an infinite simple group, any finite subgroup (including the trivial one) is weakly separable but not separable. 

\begin{thm}[{\cite[Corollary~32]{CKRW}}] \label{thm:IrredInsep}
	Let $T_1 , T_2$ be locally finite trees without vertices of degree~$1$, and let $\Gamma \leq \Aut(T_1) \times \Aut(T_2)$ be a discrete subgroup acting cocompactly on $T_1 \times T_2$. Then the following assertions are equivalent. 
\begin{enumerate}[(i)]
	\item There exists  $i \in \{1, 2\}$ such that the projection $\mathrm{pr}_i(\Gamma) \leq \Aut(T_i)$ is discrete. 
	
	\item There exists $i \in \{1, 2\}$ and a vertex or an edge $y \in VT_i \cup ET_i$ such that the stabilizer $\Gamma_{y}$ is a weakly separable subgroup of $\Gamma$.
	
	\item For all $i \in \{1, 2\}$ and all $y \in VT_i \cup ET_i$, the stabilizer $\Gamma_{y}$ is a separable subgroup of $\Gamma$.
	
	\item The groups 
	$$K_1 = \{g \in \Aut(T_1) \mid (g,1) \in \Gamma\}$$ 
	and 
	$$K_2 = \{g \in \Aut(T_2) \mid (1, g) \in \Gamma\}$$ 
	act cocompactly on $T_1$ and $T_2$ respectively, and the product $K_1 \times K_2$ is of finite index in $\Gamma$.
\end{enumerate}	
\end{thm}

If any of those conditions is satisfied, we say that $\Gamma$ is \textbf{reducible}. Otherwise it is called \textbf{irreducible}.  In the special case where $\Gamma$ is a BMW-group, this terminology coincides with the notion of (ir)reducibility introduced above. 

The equivalence between (i) and (iv) is due to M.~Burger and S.~Mozes  \cite[Proposition~1.2]{BuMo2} and follows from general results on lattices in locally compact groups.  The fact that the separability of the edge-stabilizers is equivalent to the reducibility of $\Gamma$ is due to D.~Wise, see \cite[Lemmas~5.7 and~16.2]{Wise_fig8}. In fact, D.~Wise's results allow one to derive more precise information in the case of BMW-groups. The following assertion will be relevant to our purposes (see \cite[Theorem~9]{BondarenkoKivva} for a related statement). We denote by $A^{*2}$ the set of all words of the form $ab$ with $a, b \in A \cup A^{-1}$. Thus in a BMW-group $\Gamma = \langle A \cup X \mid R\rangle $, the group $ {\langle A^{*2} \rangle}$ is the index~$2$ subgroup of   ${\langle A \rangle}$ consisting of the elements of even word length. 

\begin{prop}\label{prop:ProfClosureBMW}
	Let $\Gamma = \langle A \cup X \mid R\rangle $ be a BMW-presentation of an irreducible BMW-group. 
	Then the profinite closure $\overline{\langle A^{*2} \rangle}$ contains an element of the form $xy^{-1}$, for some distinct $x, y \in X \cup X^{-1}$.  
	
	Similarly the profinite closure $\overline{\langle X^{*2} \rangle}$ contains an element of the form $ab^{-1}$, with $a \neq b$ and $a, b \in A \cup A^{-1}$. 
\end{prop}
\begin{proof}
	The arguments in the proof of \cite[Corollary~6.4]{Wise_CSC} (or alternatively the combination of Lemmas~5.5, 5.7 and 16.2 in \cite{Wise_fig8}) imply that if $\Gamma$ is irreducible, then for any finite quotient $\varphi \colon \Gamma \to Q$, there exist distinct elements $x,  y \in X \cup X^{-1}$ such that $\varphi(xy^{-1}) \in \varphi(\langle A^{*2} \rangle)$.
	
	Let us now consider the set $P$ of all pairs of distinct elements $x,  y \in X \cup X^{-1}$ such that there is a finite group $Q_{(x, y)}$ and a homomorphism $\varphi_{(x, y)} \colon \Gamma \to Q_{(x, y)}$ with $\varphi_{(x, y)}(xy^{-1}) \not \in \varphi_{(x, y)}(\langle A^{*2} \rangle)$. Denoting by $Q$ the direct product of the groups $Q_{(x, y)}$, and by $\varphi$ the product homormorphism of the $\varphi_{(x, y)}$, taken over all $(x, y) \in P$, we obtain a homormorphism $\varphi \colon \Gamma \to Q$ of $\Gamma$ to a finite group, such that $\varphi(xy^{-1}) \not  \in \varphi(\langle A^{*2} \rangle)$ for all $(x, y) \in P$. In view of the preceding paragraph, there exists distinct  elements $x,  y \in X \cup X^{-1}$ such that $(x, y) \not \in P$.   The conclusion follows. 	A similar argument may be applied exchanging the roles of $A$ and $X$. 
\end{proof}	

We emphasize that $\langle A \rangle \cap \langle X \rangle = \{1\}$, see Proposition~\ref{prop:BMW-basic}, so that Proposition~\ref{prop:ProfClosureBMW} indeed witnesses the inseparability of $\langle A \rangle $ and $\langle X \rangle$ in $\Gamma$.

\subsection{Anti-tori and irreducibility}\label{sec:Anti-tori}

Another discovery of D.~Wise \cite[Section~II.4]{Wise_phd} is that the irreducibility of a BMW-group $\Gamma = \langle A \cup X \mid R\rangle$ can sometimes be established by highlighting what he called an \textit{anti-torus} in $\Gamma$. By definition, an \textbf{anti-torus} in a group $\Gamma$ of automorphisms of a product $T_1 \times T_2$ of two trees is a pair $g_1, g_2 \in \Gamma$ preserving the product decomposition and satisying the following conditions:
\begin{enumerate}
\item[(AT1)] For $i = 1, 2$, there is a point $p_i \in T_i$ and a geodesic line $\ell_i \subset T_i$ containing $p_i$ such that $\ell_1\times \{p_2\}$ is invariant under $g_1$ and $\{p_1\} \times \ell_2$ is invariant under $g_2$. 

\item[(AT2)] No non-zero powers of $g_1$ and $g_2$ commute. 
\end{enumerate}

The choice of terminology is motivated by considering the natural CAT($0$) realization of the product $T_1 \times T_2$. Given a anti-torus $\{g_1, g_2\}$ in $\Gamma$, the lines $\ell_1\times \{p_2\}$  and $\{p_1\} \times \ell_2$ span  a flat plane whose stabilizer  in $\Gamma$ does not act cocompactly. In other words that flat is not periodic.  

The following result is due to D.~Wise. 

\begin{prop} \label{prop:AntiTorus}
Let $T_1 , T_2$ be locally finite trees without vertices of degree~$1$, and let $\Gamma \leq \Aut(T_1) \times \Aut(T_2)$ be a discrete subgroup acting cocompactly on $T_1 \times T_2$. If $\Gamma$ contains an anti-torus, then it is irreducible. 
\end{prop}
\begin{proof}
See \cite[Proposition~9]{Rattaggi_GeomDed} for a proof in the context of torsion-free BMW-groups, and \cite[Lemma~18]{JW} for a proof in the more general context of torsion-free lattices in products of trees. Alternatively one may establish the claim without any requirement of torsion-freeness on $\Gamma$ as follows. 

Given an anti-torus 	$\{g_1, g_2\}$ in $\Gamma$ , let $p_i$ and $\ell_i$ be as in (AT1). Then $g_1$ fixes $p_2$ and $g_2$ fixes $p_1$. The condition (AT2) implies that the projections of $g_1$ on $\Aut(T_2)$ and of $g_2$ on $\Aut(T_1)$ cannot both have finite order. It follows that the projection of $\Gamma_{p_1}$ on $\Aut(T_1)$ or the projection of $\Gamma_{p_2}$ on $\Aut(T_2)$ has infinite image, and is thus non-discrete. Therefore $\Gamma$ is irreducible by Theorem~\ref{thm:IrredInsep}. 	
\end{proof}

I do not know whether the converse assertion holds. This amounts to asking the following: Given a cocompact lattice $\Gamma \leq \Aut(T_1) \times \Aut(T_2)$ and vertices $v_1 \in V(T_1)$ and $v_2 \in V(T_2)$, is it possible that the projection of the stabilizer $\Gamma_{v_1}$ to  $ \Aut(T_1)$ and the projection of  $\Gamma_{v_2}$ to  $\Aut(T_2)$ are both infinite torsion groups? 

\subsection{Local actions and irreducibility}\label{sec:LocalAction}

By Theorem~\ref{thm:IrredInsep}, in order to prove the irreducibility of a BMW-group, it suffices to show that its projection to the automorphism group of one of the tree factors of its Cayley graph is non-discrete. 
An idea developed by M.~Burger and S.~Mozes \cite{BuMo1, BuMo2} in order to check that condition  is to use some of the geometric aspects of finite group theory. Let us describe that in detail.

Let $\mathfrak g= (V, E)$ be a connected locally finite (undirected, unlabeled, simple) graph and $\Gamma \leq \Aut(\mathfrak g)$ be a group of automorphisms of $\mathfrak g$. Given a vertex $v \in V$, the \textbf{local action} of $\Gamma$ at $v$ is the finite permutation group induced by the action of the stabilizer $\Gamma_v$ on the sphere $S(v, 1)$ of radius~$1$ around $v$. Given an integer $r \geq 0$ and a vertex $v \in V$, we set 
$$\Gamma_v^{[r]} = \bigcap_{w \in V,\;  d(v, w) \leq r} \Gamma_w.$$
Thus $\Gamma_v^{[r]}$ is the pointwise stabilizer of the $r$-ball around $v$. 

The proof of the following fundamental result relies on the classification of the finite $2$-transitive groups, which relies in turn on the CFSG. 

\begin{thm}\label{thm:TrofimovWeiss}
Let $\Gamma \leq \Aut(\mathfrak g)$ be a vertex-transitive automorphism group of a connected locally finite graph  $\mathfrak g$. Let $\{v, w\}$ be an edge of $\mathfrak g$. Suppose that the local action at  $v$ is $2$-transitive. If the stabilizer $\Gamma_v$ is finite, then:
\begin{enumerate}[(i)]
\item {\upshape (Trofimov--Weiss \cite[Theorem~1.4]{TrofimovWeiss})} We have
$$\Gamma_v^{[5]} \cap \Gamma_v^{[5]} = \{1\}.$$ 
In particular  $\Gamma_v^{[6]}= \{1\}$. 

\item {\upshape (Trofimov--Weiss \cite[Theorem~1.3]{TrofimovWeiss})} If   $\Gamma_v^{[1]} \cap \Gamma_w^{[1]} \neq  \{1\}$, then the local action at $v$ contains a normal subgroup isomorphic to $\PSL_n(\FF_q)$ in its natural action on the points of the $n-1$-dimensional projective space over $\FF_q$. 

\item  {\upshape (R.~Weiss \cite[Theorem~1.1]{Weiss79})} If  the local action at $v$ contains a normal subgroup isomorphic to $\PSL_2(\FF_q)$ in its natural action on the points of the projective line   over $\FF_q$, then 
$$\Gamma_v^{[3]}  \cap \Gamma_w^{[3]}  = \{1\}.$$ 
In particular  $\Gamma_v^{[4]} = \{1\}$.
\end{enumerate}

\end{thm}

The hypothesis that the stabilizer $\Gamma_v$ be finite is equivalent to the condition that $\Gamma$ is a discrete subgroup of $\Aut(\mathfrak g)$ endowed with the topology of pointwise convergence on the vertex set $V(\mathfrak g)$, which is second countable, totally disconnected and locally compact. Notice  that for any vertex-transitive automorphism group $\Gamma$ of a connected locally finite graph $\mathfrak g$, and for any $n \geq 0$, we have 
$$\Gamma_v^{[n]} = \{1\} \hspace{1cm}\text{if and only if} \hspace{1cm} \Gamma_v^{[n]} \leq \Gamma_v^{[n+1]}.$$ 
Thus Theorem~\ref{thm:TrofimovWeiss} has the following discreteness criterion as an immediate consequence. 

\begin{cor}\label{cor:TrofimovWeiss}
Let $\Gamma \leq \Aut(\mathfrak g)$ be a vertex-transitive automorphism group of a connected locally finite graph  $\mathfrak g$. Suppose that the local action at  $v$ is $2$-transitive. If any of the following conditions holds, then $\Gamma$ is an indiscrete subgroup of  $\Aut(\mathfrak g)$:
\begin{enumerate}[(i)]
\item $\Gamma_v^{[6]}\not \leq \Gamma_v^{[7]} $. 

\item $\Gamma_v^{[2]}\not \leq \Gamma_v^{[3]} $ and the local action of $\Gamma$ at $v$ is not isomorphic to the action of a  linear  or semi-linear group of degree $n$ over $\FF_q$ on the points of the projective $n-1$-space over $\FF_q$. 

\item $\Gamma_v^{[4]} \not \leq \Gamma_v^{[5]}$ and the local action of $\Gamma$ at $v$ is   isomorphic to the action of a  linear  or semi-linear group of degree $2$ over $\FF_q$ on the points of the projective line over $\FF_q$. 
\end{enumerate}
\end{cor}

Given a BMW-presentation $\Gamma = \langle A \cup X \mid R\rangle $ and associated Cayley graph $T_A \times T_X$, the group $\Gamma$ acts vertex-transitively on $T_A \times T_X$. Moreover, that action preserves the product decomposition. By projecting, we obtain a vertex-transitive action of $\Gamma$ on $T_A$ and another one on $T_X$. In order to apply the criteria from Corollary~\ref{cor:TrofimovWeiss}, we need to compute the local action of $\Gamma$ at the base vertex of $T_A$ or $T_X$. This can be done effectively as follows. 

\begin{center}
	\begin{figure}[h]
		\includegraphics[width=8cm]{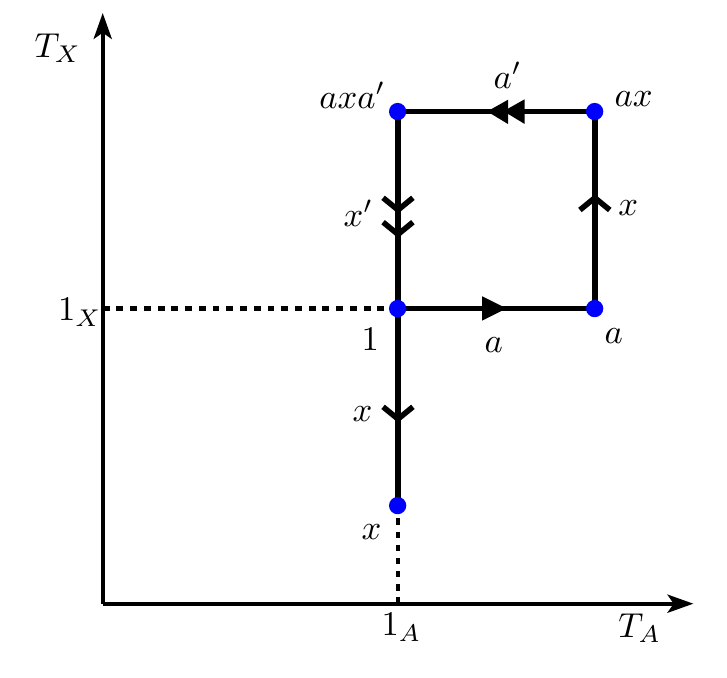}
			\caption{The local action of $a$ on $T_X$ at the base vertex $1_X$} \label{fig:LocalAction}
	\end{figure}
\end{center}

The directed edges of the Cayley graph $T_A \times T_X$ are labeled by the elements of $A \cup A^{-1} \cup X \cup X^{-1}$. Consider the base vertex $1$ and let $1_X$ be the projection of $1$ to the factor $T_X$ of the product graph $T_A \times T_X$,  (see Figure~\ref{fig:LocalAction}). Let $a \in A \cup A^{-1}$ and $x \in X \cup X^{-1}$.  The element $a \in A$ maps the vertex $1$ to $a$;  both of those vertices  project to $1_X$ in $T_X$. Indeed $\langle A \rangle$ fixes the vertex $1_X$. Moreover $a$ maps the directed edge $(1, x)$ to the directed edge $(a, ax)$, and both edges are labeled by $x$.  Now the presence of the relation $axa'x'$ in the set $R$ has the following geometric interpretation: the projection of the directed edge $(a, ax)$ to the $T_X$-fiber over $1_A$ in the product decomposition $T_A \times T_X$ coincides with the directed edge $(1, (x')^{-1})$. Thus, viewing $T_X$ as  a  directed graph whose edges are labeled by $X \cup X^{-1}$, we see that the local action of $a$ at $1_X$ maps the outgoing edge at $1_X$ labeled by $x$ to the ingoing edge at $1_X$ labeled by $x'$.

Proceeding in that way, one computes that for the BMW-group $\Gamma_{\mathrm{SV}}$ from Example~\ref{ex:4,4}, the local permutation induced by $a$ and $b$ around the base vertex $1_X$ of the tree $T_X$ are
$$a \colon (x, y^{-1}, y, x^{-1}) \hspace{1cm} \text{and} \hspace{1cm} b \colon (x, y, y^{-1}, x)$$
respectively. Thus the local action of $\Gamma_{\mathrm{SV}}$ at $1_X$ in $T_X$ is isomorphic to $\Sym(4) \cong \PGL_2(\FF_3)$. By similar arguments, one computes the local actions of the various examples introduced in the previous section. The information thus obtained is recorded in Table~\ref{tab:ExamplesLocalActions}.

\begin{center}
	\begin{table}[h]
		\def\arraystretch{1.5}
		$$\begin{array}{|c|c|c|c|c|}
		\hline
		\text{BMW-group} &\deg(T_A) & \text{Local action in } T_A & \deg(T_X) &\text{Local action in } T_X\\
		\hline
		\Gamma_{\mathrm{Rung}} & 3 &  \Sym(3) \cong \PGL_2(\FF_2) & 3 &  \Sym(3) \cong \PGL_2(\FF_2) \\
		\Gamma_{3, 3} & 3 &  \Sym(3) & 3 & C_2 \\
		\Gamma_{\mathrm{SV}}  & 4 & \Sym(4) \cong \PGL_2(\FF_3) & 4 & \Sym(4) \cong \PGL_2(\FF_3) \\
		\Gamma_{\mathrm{JW}} & 4 & \Alt(4) \cong \PSL_2(\FF_3) & 4 & D_8\\
		\Gamma_{4, 5} & 4 & \Sym(4) & 5 &  \Sym(5) \\
		\Gamma_{\mathrm{Wise}} & 4 & C_2 \times C_2  & 6 & \Sym(3) \times \Sym(3)\\
		\Gamma_{\mathrm{Ratt}} &  4 & \Sym(4) \cong \PGL_2(\FF_3) & 6 &  \PGL_2(\FF_5)   \\
		\Gamma_{\mathrm{BDR}}  & 6 & \Sym(3) \times C_3 & 6 & \Sym(3) \times C_3\\
		\Gamma_{6,6} & 6  & \Sym(6) & 6 & \Alt(6)\\
		\hline
		\end{array}$$
		
		\vspace{.5cm}
		
		\caption{Local actions for some small BMW-groups}\label{tab:ExamplesLocalActions}
	\end{table}
\end{center}

Notice that the local action of $\Gamma_{\mathrm{Wise}}$ on both tree factors is intransitive. Nevertheless $\Gamma_{\mathrm{Wise}}$  has non-discrete projections on both $\Aut(T_A)$ and $\Aut(T_X)$ by   \cite[Theorem~5.3]{Wise_CSC} and is thus irreducible by Theorem~\ref{thm:IrredInsep}. This shows that the irreducibility criterion for BMW-groups derived from Corollary~\ref{cor:TrofimovWeiss} and Theorem~\ref{thm:IrredInsep} provides a sufficient condition which is not necessary. However, that condition can be used to (re)check the irreducibility of the lattices $\Gamma_{3, 3}$, 	$\Gamma_{\mathrm{Rung}}$, $\Gamma_{\mathrm{SV}} $, $\Gamma_{\mathrm{JW}}$, $\Gamma_{4, 5}$,  $\Gamma_{\mathrm{Ratt}} $ and $\Gamma_{6,6}$.

\subsection{Residual finiteness}\label{sec:ResidualFiniteness}

In a Hausdorff topological group, the centralizer of any subset is closed. In particular, in a residually finite group, the centralizer of any subset is separable. This observation yields the following basic fact, where $G^{(\infty)}$ denotes the finite residual of $G$, as in the proof of Proposition~\ref{prop:hji}.%, where the \textbf{finite residual} of a group $G$, denoted by $G^{(\infty)}$, is the intersection of all finite index subgroups of $G$ (notice that $G^{(\infty)}$ coincides with the profinite closure $\overline{\{1\}}$ of the trivial subgroup). 

\begin{lem}%[{\cite[Lemma~16]{CKRW}}] 
\label{lem:FiniteResidual}
Let $G$ be a group. For any $H \leq G$ we have $[C_G(H), \overline H] \subset G^{(\infty)}$. 
\end{lem}

That observation was used by D.~Wise in \cite{Wise_phd, Wise_CSC} in order to construct BMW-groups that are not residually finite in the following way. 
If $G$ is a group and $H \leq G$ is an inseparable subgroup, then the \textbf{double of $G$ over $H$}, defined as the free amalgamated product $G *_H G$ of two copies of $G$, is not residually finite, since it admits an involutory automorphism swapping the two  factors, and whose centralizer coincides with $H$. 

Combining that idea with Proposition~\ref{prop:ProfClosureBMW}, we obtain the following. 

\begin{prop}[D.~Wise] \label{prop:Doubling}
	Let $\Gamma = \langle A \cup X \mid R \rangle$ be a BMW-presentation of an irreducible BMW-group of degree $(m, n)$. 
	Let $A \to \bar A : a \mapsto \bar a$ be a bijection between $A$ and a set $\bar A$, and set 
	$$\bar R_2 = \{  \bar a^2 \mid a \in A, \ a^2 \in R\},  $$ 
	$$\bar R_4 = \{  \bar a x \bar a' x' \mid a, a' \in A\cup A^{-1}, \ x, x' \in X \cup X^{-1}, \ axa'x' \in R\} $$
	and 
	$$\bar R = \bar R_2 \cup \bar R_4.$$
	Then $\Lambda = \langle A \cup \bar A \cup X \mid R \cup \bar R\rangle \cong \Gamma *_{\langle X \rangle} \Gamma$ is a BMW-presentation of an irreducible BMW-group of degree $(2m, n)$ which is not residually finite. More precisely, there exist $a \neq b \in A$ such that $ab^{-1}\bar b \bar a^{-1}$ lies in the intersection of all finite index subgroups of $\Lambda$. 
	\end{prop}

Applying that result to   $\Gamma_{\mathrm{Wise}}$, D.~Wise \cite{Wise_phd} obtained a non-residually finite BMW-group of degree $(8, 6)$, which was also the first example of a finitely presented non-residually finite small cancellation group and thereby answered negatively a question of P.~Schupp \cite{Schupp_survey} from 1973.

M.~Burger and S.~Mozes also constructed non-residually finite BMW-groups using another closely related method that they developed independently. The following result, which illustrates their method, was proved by them in the special case where the local action of $\Gamma$ on both trees $T_A$ and $T_X$ is quasi-primitive, see \cite[\S2.1, 2.2, 2.3 and 6.1]{BuMo2}. The proof of the present version will be discussed below. 

\begin{prop}\label{prop:FiberProduct}
 		Let $\Gamma = \langle A \cup X \mid R \rangle$ be a BMW-presentation of an irreducible BMW-group of degree $(M, N)=(2m+m', 2n+n')$, where the notation is as in Proposition~\ref{prop:BMW-basic}. 
 		Let 
 		\begin{align*}
 		\tilde R_2 = \bigg\{ (a_1, a_2)^2 \mid & a_1, a_2 \in A, \ a_1^2, a_2^2 \in R\bigg\} \cup
   \bigg\{ (x_1, x_2)^2 \mid x_1, x_2 \in X, \ x_1^2, x_2^2 \in R\bigg\}
 		\end{align*}
 		and
 		$$\tilde R_4=\bigg\{  (a_1, a_2)(x_1, x_2)(a'_1, a'_2)(x'_1, x'_2) \mid a_1x_1a'_1x'_1 \in R, \ a_2x_2 a'_2 x'_2  \in R\bigg\}.$$
 		Then 
 		$$\Gamma \boxtimes \Gamma = \langle (A \times A) \cup (X \times X) \mid \tilde R_2 \cup \tilde R_4\rangle$$
 		is an irreducible BMW-group of degree $\big(2m^2+ 4mm' + (m')^2, \ 2n^2 + 4nn'+ (n')^2\big)$ which is not residually finite. 
\end{prop}

The criterion developed by Burger--Mozes in order to prove their version of Proposition~\ref{prop:FiberProduct} ensures that, under suitable conditions on the local action, an irreducible lattice in $\Aut(T_1) \times \Aut(T_2)$ whose projections to at least one of the factors is not injective, cannot be residually finite, see \cite[Proposition~2.1]{BuMo2}. That criterion was subsequently generalized to lattices in products of CAT($0$) spaces \cite[Proposition~2.4]{CaMo_KM}, and then to irreducible lattices in products of locally compact groups in \cite[Corollary~33]{CKRW} and \cite[\S 5]{CapraceWesolek}. We record the following geometric version, where there is no condition on the local action. 

\begin{prop}[{Caprace--Monod \cite[Proposition~2.4]{CaMo_KM}}]\label{prop:NonRF}
 Let $n \geq 2$ and for each $i =1, \dots, n$, let $T_i$ be a locally finite tree with infinitely many ends and no vertex of degree~$1$. Let  also $\Gamma \leq \Aut(T_1) \times \dots \times \Aut(T_n)$ be a discrete subgroup acting cocompactly on $T_1 \times \dots \times T_n$. Assume that for each nonempty proper subset $I \subsetneq \{1, \dots, n\}$, the projection  $\Gamma \to \prod_{i \in I} \Aut(T_i)$ has non-discrete image. 
 
If there exists a nonempty proper subset $I \subsetneq \{1, \dots, n\}$ such that  the projection  $\Gamma \to \prod_{i \in I} \Aut(T_i)$ fails to be injective, then $\Gamma$ is not residually finite.
\end{prop}

This rather general criterion applies to numerous examples. First of all, Propositions~\ref{prop:Doubling} and~\ref{prop:FiberProduct} can both be derived  from it. Consider for example the set-up of Proposition~\ref{prop:Doubling}. Let $\pi \colon \Lambda =  \Gamma *_{\langle X \rangle}\Gamma  \to \Gamma$ be the natural homomorphism, which is injective on $\langle X\rangle$. Observe that $H = \langle X \rangle$ is a \textbf{commensurated subgroup} of   $\Lambda$, i.e. for every $g \in \Lambda$, we have $[H : H \cap gH g^{-1}] < \infty$. In particular, given $ g \in \Lambda$ and   $h \in H \cap g^{-1}H g$, we have $ghg^{-1}  = h'\in H$. If now $g \in  \Ker(\pi)$ and recalling  that $\pi$ is the identity on $H$, we see that $h' = \pi(h')= \pi(ghg^{-1}) = \pi(h) = h$. Thus any $g\in  \Ker(\pi)$ centralizes a finite index subgroup of $H = \langle X \rangle$. Therefore  $\Ker(\pi)$   acts trivially on the tree $T_X$, which is the Cayley graph of $\langle X \rangle$ by Proposition~\ref{prop:BMW-basic}. Thus $\Ker(\pi)$, which contains all elements of the form $a \bar a^{-1}$ with $a \in A$,  is contained in the kernel of the projection of $\Lambda$ to $\Aut(T_X)$. Proposition~\ref{prop:NonRF} ensures that $\Lambda$ is not residually finite. The proof of Proposition~\ref{prop:FiberProduct} follows by similar considerations. 

An alternative approach providing also explicit elements in  $ \Lambda^{(\infty)}$ is as follows. Using  Proposition~\ref{prop:ProfClosureBMW} and applying   Lemma~\ref{lem:FiniteResidual}   to the extension of $\Lambda$ by the automorphism of order~$2$ fixing $\langle X \rangle$ pointwise and mapping $a$ to $ \bar a$ for all $a \in A$, we see that there exist $a \neq b \in A$ such that $ab^{-1} \bar b \bar a^{-1} \in \Lambda^{(\infty)}$. 

Proposition~\ref{prop:NonRF} also applies to other situations. For instance, it implies that the group $\Gamma_{\mathrm{BDR}}$ is not residually finite: Indeed, the projection of the free group $\langle a, b, c\rangle$ on the automorphism group of the tree $T_{\{x, y, z\}}$ is not injective, since its image is isomorphic to a lamplighter group by Proposition~\ref{prop:Lamplighter(6,6)}. Similarly, a computation shows that in the group $\Gamma_{\mathrm{JW}}$ from Example~\ref{ex:4,4}, the elements $x^3$ and $y^3$ each centralize a subgroup of index~$4$ of the free group $\langle a, b\rangle$. It follows that $\langle x^3, y^3\rangle$ lies in the kernel of the action of $\Gamma_{\mathrm{JW}}$ on $T_A = T_{\{a, b\}}$. Hence $\Gamma_{\mathrm{JW}}$ is not residually finite. One can show in a similar way that the Wise lattice is not residually finite. The non-residual finiteness of $\Gamma_{\mathrm{JW}}$ and $\Gamma_{\mathrm{Wise}}$ was recently observed independently in \cite{BondarenkoKivva} and \cite{CapraceWesolek}. 

The  groups  $\Gamma_{\mathrm{BDR}}$, $\Gamma_{\mathrm{JW}}$ and $\Gamma_{\mathrm{Wise}}$ have small degree, and this was essential in the  verification  that they fulfill the non-injectivity hypothesis of Proposition~\ref{prop:NonRF}. The following result provides a general criterion in terms of local actions that allows one to check the hypothesis of Proposition~\ref{prop:NonRF} on BMW-groups of arbitrary degree. 

\begin{prop}[{Caprace--Wesolek}]\label{prop:NilpotentLocalAction}
Let $n \geq 2$ and for each $i =1, \dots, n$, let $T_i$ be a locally finite tree with infinitely many ends and no vertex of degree~$1$. Let  also $\Gamma \leq \Aut(T_1) \times \dots \times \Aut(T_n)$ be a discrete subgroup acting cocompactly on $T_1 \times \dots \times T_n$. Then the  following assertions hold for any   $i \in \{1, \dots, n\}$. 
 \begin{enumerate}[(i)]
 \item If the  $\Gamma$-action on $T_i$ is vertex-transitive, and the local action of $\Gamma$ in $T_i$ is nilpotent, then the projection  $\Gamma \to \prod_{j \neq i} \Aut(T_j)$ is not injective. 
 
 \item If in addition the projection  $\Gamma \to \prod_{j \neq i} \Aut(T_j)$  has a non-discrete image, then $\Gamma$ is not residually finite. 
 \end{enumerate}

\end{prop}
\begin{proof}
Assertion (i) follows from \cite[Corollary~5.3, Lemma~5.10 and Proposition~6.3]{CapraceWesolek}. Assertion (ii) follows from \cite[Corollary~6.4]{CapraceWesolek}. 
\end{proof}

The following consequence is immediate (see  Table~\ref{tab:ExamplesLocalActions}). 

\begin{cor}\label{cor:BMW:NilpotentLocal}
Let $\Gamma = \langle A \cup X \mid R \rangle$ be a BMW-presentation. If $\Gamma$ is irreducible and if the local action of $\Gamma$ on $T_A$ or on $T_X$ is nilpotent, then $\Gamma$ is not residually finite. In particular, the BMW-groups $\Gamma_{\mathrm{JW}}$, $\Gamma_{\mathrm{Wise}}$ and  $\Gamma_{3, 3}$ are not residually finite. 
\end{cor}

\begin{rmk}\label{rem:Gamma2}
An explicit non-trivial element of the finite residual of $\Gamma_{\mathrm{JW}}$ can be obtained as follows. We mentioned above that in the group $\Gamma_{\mathrm{JW}}$, the elements $x^3$ and $y^3$ each commute with a subgroup of index~$4$ of $\langle A\rangle$, where $A = \{a, b\}$. More precisely $x^3$ centralizes the stabilizer $\langle A \rangle_x \leq \langle A \rangle$ of the edge labelled by $x$ emanating from $1_X$ in the tree $T_X$, while  $y^3$ centralizes the stabilizer $\langle A \rangle_y \leq \langle A \rangle$ of the edge labelled by $y$.  Moreover the groups $\langle A \rangle_x$ and $\langle A \rangle_y$ are both contained in  $\langle A^{*2} \rangle$, which is a subgroup of index~$2$ in $\langle A \rangle$.  On the other hand, Proposition~\ref{prop:ProfClosureBMW} ensures that the profinite closure $\overline{\langle A^{*2} \rangle}$ contains a non-trivial element of $X^{*2}$. Analyzing the local action of $\Gamma_{\mathrm{JW}}$ on $T_X$, one checks that this element is $x^2$ or $y^2$. In view of the presentation of $\Gamma_{\mathrm{JW}}$, the assignments $a \mapsto a^{-1}, b \mapsto b^{-1}, x\mapsto y^{-1}$ and $y \mapsto x^{-1}$ extend to an automomorphism of $\Gamma_{\mathrm{JW}}$. It follows that  the profinite closure $\overline{\langle A^{*2} \rangle}$ contains both $x^2$ and $y^2$. Since $\langle A \rangle_x$ and $\langle A \rangle_y$ are both index~$2$ subgroups of $\langle A^{*2} \rangle$, we have $x^4, y^{4}  \in \overline{\langle A \rangle_x} \cap \overline{\langle A \rangle_y}$. By Lemma~\ref{lem:FiniteResidual}, we have $[x^3, \overline {\langle A \rangle_x}] \subset \Gamma_{\mathrm{JW}}^{(\infty)}$ and $[y^3, \overline {\langle A \rangle_y}] \subset \Gamma_{\mathrm{JW}}^{(\infty)}$. We deduce that the commutators $[x^3, y^4]$ and $[y^3, x^4]$ are both non-trivial elements of $\Gamma_{\mathrm{JW}}^{(\infty)}$.
\end{rmk}

\subsection{The Normal Subgroup Theorem, after U.~Bader and Y.~Shalom}

We have already seen examples of BMW-groups that are hereditarily just-infinite, see Propositions~\ref{prop:StixVdovina}, \ref{prop:Rattaggi} and~\ref{prop:Rungtana}. Those groups are irreducible  arithmetic lattices in products of two simple algebraic groups of rank~$1$ (in characteristic~$3$,~$0$ and~$2$ respectively), and the just-infinite property was established using the Margulis Normal Subgroup Theorem. A remarkable breakthrough due to M.~Burger and S.~Mozes \cite[Theorem~4.1]{BuMo2} was to extend the scope of the Margulis Normal Subgroup Theorem so that it applies to a much larger class of BMW-groups, including non-arithmetic (and even non-linear) ones. The Burger--Mozes Normal Subgroup Theorem applies to   irreducible cocompact lattices in products of certain virtually simple locally compact groups of automorphisms of trees. It was   generalized later by U.~Bader and Y.~Shalom~ \cite{BaSha}, who obtained a fundamental result whose   level of generality is absolutely stunning. In order to present it, we first recall that a locally compact group is called \textbf{just-non-compact} if it is not compact and if all its  proper Hausdorff quotients are compact.  It is \textbf{hereditarily just-non-compact} if every finite index open subgroup is just-non-compact. The main example of a hereditarily just-non-compact group to keep in mind is that of a \textbf{topologically simple} group, i.e. a non-trivial locally compact group whose only closed normal subgroups are the trivial ones. An example of a just-non-compact locally compact group which is not hereditarily so is provided by the wreath product $S \wr C_2$ of a topologically simple group $S$ with a cyclic group of order~$2$.  A general result on the structure of a compactly generated just-non-compact locally compact group  can be found in \cite[Theorem~E]{CaMo_cpt} (see also \cite[Proposition~2.7]{CapLB}).

 The following theorem  is a slight strengthening of the original formulation of Bader--Shalom in \cite{BaSha}, taking advantage of recent results from \cite{CapLB} concerning the case where the number of factors of the ambient product group is at least~$3$. 

\begin{thm}[U.~Bader and Y.~Shalom]\label{thm:BaderShalom}
Let $n \geq 2$ and for each $i =1 , \dots, n$, let $G_i$ be a non-discrete compactly generated locally compact group which is just-non-compact and contains no abelian normal subgroup other than the identity subgroup. 
Let $\Gamma < G_1 \times \dots \times G_n$   be a cocompact lattice such that  for all $j$, the projection  $\Gamma \to   G_j$ has dense image. If $n > 3$, assume in addition that for all $j$, the projection  $\Gamma \to \prod_{i \neq j} G_i$ has non-discrete image. Then $\Gamma$ is   just-infinite. 

If in addition $G_i$ is hereditarily just-non-compact for all $i$, then $\Gamma$ is hereditarily just-infinite. 
\end{thm} 
\begin{proof}
	Observe that whatever the value of $n \geq 2$ is,  the projection  $\Gamma \to \prod_{i \neq j} G_i$ has non-discrete image for all $j$. This is clear if $n=2$ or $n \geq 4$ in view of the hypotheses. In the case where $n=3$, the same assertion holds by   \cite[Corollary~I]{CapLB}. Therefore, in all cases, we may invoke \cite[Theorem~H]{CapLB}. This  ensures that   for all $j$, the projection $p_j \colon \Gamma \to   G_j$ is injective. 

Let now $N \neq \{1\}$ be a normal subgroup of $\Gamma$. Then for all $j$, the group $\overline{p_j(N)}$ is a non-trivial closed normal subgroup of $G_j$. It is thus cocompact in $G_j$ since the latter is just-non-compact. It then follows from \cite[Theorem~3.7(iv)]{BaSha} and \cite[Theorem~0.1]{Shalom00} that the quotient $\Gamma/N$ is finite. Hence $\Gamma$ is just-infinite. 

Assume now that  $G_i$ is hereditarily just-non-compact for all $i$, and let $\Lambda \leq \Gamma$ be a finite index subgroup of $\Gamma$. We may then replace $G_i$ by $\overline{p_i(\Lambda)}$ for all $i$ and apply the first part of the proof. This shows that $\Lambda$ is just-infinite. Hence $\Gamma$ is hereditarily just-infinite. 
\end{proof}

The condition excluding abelian normal subgroups in $G_i$ in Theorem~\ref{thm:BaderShalom} is necessary: it is easy to see that the isometry group $G = \Isom(\mathbf R)$ is just-non-compact, and that the discrete group $\Gamma = (\mathbf Z \times \mathbf Z) \rtimes  C_2$, where the generator of $C_2$ acts via $(a, b) \mapsto (-a, -b)$, is not just-non-compact, but embeds as a  cocompact lattice with dense projections in $G\times G$. 

Similarly, the condition of compact generation is also necessary. In order to see that, consider the group $G_1 =   \PSL_2(\overline{\FF_p}) \rtimes \Aut(\overline{\FF_p})$, where $\overline{\FF_p}$ denotes the algebraic closure of the finite field of order $p$. The group $G_1$ carries a second countable locally compact group topology which induces the discrete topology on the countable subgroup $\PSL_2(\overline{\FF_p}) $, and the Krull topology (which is compact) on the Galois group $\Aut(\overline{\FF_p})$. Since the discrete group $\PSL_2(\overline{\FF_p}) $ is locally finite, it follows that every compactly generated closed subgroup of $G_1$ is compact. Since $G_1$ is not compact, it follows that it cannot be compactly generated. On the other hand, it is not difficult to show that every non-trivial normal subgroup of $G_1$ contains $\PSL_2(\overline{\FF_p}) $, so that $G_1$ is just-non-compact (and hereditarily so). We also set $G_2 = \ZZ$ and $\Gamma = \PSL_2(\overline{\FF_p}) \rtimes_\alpha \ZZ$, where $\alpha$ is a generator of the pro-cyclic group  $\Aut(\overline{\FF_p})$. The group $\Gamma$ embeds as a cocompact lattice with dense projections in $G_1 \times G_2$. Both factors in that product are hereditarily just-non-compact,  but $\Gamma$  maps onto $\ZZ$ and thus fails to be just-infinite. 

Theorem~\ref{thm:BaderShalom} also holds for some non-uniform lattices under a technical hypothesis called \emph{integrability}, see  \cite{BaSha}.   This was exploited  in \cite{CaRe} to exhibit another family of finitely presented infinite simple groups coming from Kac--Moody theory; we will not pursue that direction here. 

The proof of Theorem~\ref{thm:BaderShalom} follows the scheme designed by Margulis in the proof of his own Normal Subgroup Theorem. The required conclusion is obtained by combining two independent results,   proved separately with completely different methods: the first ensures that every proper quotient of $\Gamma$ has Kazhdan's property (T), the second ensures that every proper quotient of $\Gamma$ is amenable (see \cite{BHV} for a detailed exposition of  those important notions). The conclusion follows since the only discrete groups that satisfy both property (T) and amenability are finite. In fact, some parts of Margulis' original proof were already formulated  for a rather general class of  locally compact groups without any hypothesis requiring that those are algebraic over local fields, see \cite[Theorem~1.3.2]{Margulis79}. The property (T) half of the scheme above was achieved at the greatest level of generality by Y.~Shalom \cite{Shalom00}, while the amenability half is the work of Bader--Shalom \cite{BaSha}.

\subsection{Alternating and fully symmetric local actions}% and a $10$-relator virtually simple BMW-group}

In order to apply the Normal Subgroup Theorem~\ref{thm:BaderShalom} to a BMW-group given by its BMW-presentation, one should check that the closure of its projections on the automorphism groups of both tree factors are just-non-compact. Once again, there is no general tool allowing one  to describe the closure of a non-discrete subgroup of the automorphism group of a tree or any other connected locally finite graph.  In addressing that issue, M.~Burger and S.~Mozes highlighted a very striking phenomenon, allowing them to control the closure of a non-discrete vertex-transitive automorphism group of a tree under the hypothesis that the local action has simple (or almost simple) point stabilizers. In order to give a precise statement, we need to recall the definition of the Burger--Mozes  {universal group} of automorphisms of the $d$-regular tree $T$ with local action prescribed by a permutation group $ F \leq \Sym(d)$, introduced in \cite[\S 3.2]{BuMo1}. Fix a map $i \colon E(T) \to \{1, \dots, d\}$ such that for every vertex $v \in V(T)$, the restriction $i |_{E(v)}$ of $i$ to the set of edges containing $v$ is a bijection. Given $g \in \Aut(T)$ and $v \in V(T)$, we set
$$\sigma(g, v) = i |_{E(gv)} \circ g \circ (i|_{E(v)})^{-1}.$$
Notice that $\sigma(g, v) \in \Sym(d)$ for all $g$ and $v$. 
Given any $F \leq \Sym(d)$, the  Burger--Mozes \textbf{universal group} of automorphisms of $T$ with local action prescribed by $F$ is defined as
$$U(F) = \{g \in \Aut(T) \mid \sigma(g, v) \in F \text{ for all } v \in V(T)\}.$$
One checks that, up to conjugation in $\Aut(T)$, it is independent of choice of the map $i$. If $d \geq 3$ and if the group $F$ is transitive and generated by its point stabilizers, then the index~$2$ subgroup $U(F)^+$ of $U(F)$ preserving the canonical bipartition of $T$ is a simple compactly generated locally compact group, see  \cite[Proposition~3.2.1]{BuMo1}. 

\begin{thm}[{Burger--Mozes \cite[Propositions~3.1.2, 3.3.1 and~3.3.2]{BuMo1}}] \label{thm:U(alt)}
Let $T$ be a $d$-regular tree with $d \geq 6$ and  $\Gamma \leq \Aut(T)$ be a non-discrete vertex-transitive subgroup.

\begin{enumerate}[(i)]
\item  If the local action of $\Gamma$ in $\Aut(T)$ contains the full alternating group $\Alt(d)$, then the closure $\overline \Gamma$ is hereditarily just-non-compact, without abelian normal subgroup other than $\{1\}$. 

\item If the local action of $\Gamma$ in $\Aut(T)$ coincides with the full alternating group $\Alt(d)$, then the closure $\overline \Gamma$ is conjugate to the universal group $U(\Alt(d))$ in $\Aut(T)$. In particular $\overline \Gamma$ has a simple subgroup of index~$2$. 

\end{enumerate}

\end{thm}

One should keep in mind that, in view of Corollary~\ref{cor:TrofimovWeiss}, the non-discreteness of $\Gamma$ can be checked in a ball of radius~$3$ under the assumption that the local action contains the alternating group of degree~$\geq 6$. The uniqueness (up to conjugacy) of the non-discrete vertex-transitive closed subgroup of $\Aut(T)$ established in Theorem~\ref{thm:U(alt)}(ii) is rather surprising. The complete classification of the non-discrete vertex-transitive closed subgroups of $\Aut(T)$ whose local action is the full symmetric group was recently achieved by N.~Radu. In particular, the following result of his is an important refinement of Theorem~\ref{thm:U(alt)}(i). 

\begin{thm}[{N.~Radu \cite[Theorem~B and Corollary~E]{Radu_InvMath}}] \label{thm:Radu}
Let $T$ be a $d$-regular tree with $d \geq 6$ and  $\Gamma \leq \Aut(T)$ be a non-discrete vertex-transitive subgroup.

If the local action of $\Gamma$ in $\Aut(T)$ contains the full alternating group $\Alt(d)$, then the closure $\overline \Gamma$ has a simple subgroup of index~$\leq 8$, and belongs to an explicit  infinite list of  examples. 
\end{thm}

We refer to \cite{Radu_InvMath} for a description of those examples, and for a more general classification result that does not require any hypothesis of vertex-transitivity. We emphasize the contrast between Theorem~\ref{thm:U(alt)}(ii) and Theorem~\ref{thm:Radu}: if the local action is $\Alt(d)$, then the non-discrete group $\overline{\Gamma}$ is uniquely determined, whereas if the local action is $\Sym(d)$, there are infinitely many pairwise non-conjugate possibilities for $\overline{\Gamma}$.

\subsection{Virtually simple BMW-groups  of small degree}\label{sec:SimpleBMW}

By way of illustration, let us consider the BMW-group $\Gamma_{6,6}$ introduced in Proposition~\ref{prop:GammaRadu}. It is irreducible of degree $(6,6)$, and in view of Table~\ref{tab:ExamplesLocalActions}, its local action on the two tree factors $T_A$ and $T_X$ is $\Sym(6)$ and $\Alt(6)$ respectively. In view of Theorems~\ref{thm:BaderShalom} and~\ref{thm:U(alt)}, it follows that $\Gamma_{6,6}$ is hereditarily just-infinite.  
From Theorem~\ref{thm:U(alt)} we deduce moreoveor that the closure of the projection of $\Gamma$ to $\Aut(T_X)$ is isomorphic to $U(\Alt(6))$. The closure of the projection of $\Gamma$ to $\Aut(T_A)$ has also been identified by N.~Radu \cite{Radu_phd}: it is isomorphic to the subgroup $G_{(i)}(\{1\}, \{1\}) \leq \Aut(T_A)$ in the notation of \cite[Definition~4.9]{Radu_InvMath}. It has a simple subgroup of index~$4$. 

Since $\Gamma_{6,6}$ also contains the non-residually finite group $\Gamma_{\mathrm{JW}}$ as a subgroup (see Proposition~\ref{prop:GammaRadu}), it follows from Proposition~\ref{prop:hji} that $\Gamma_{6,6}$ is virtually simple. In fact, the following more precise assertion holds. 

\begin{prop}[N.~Radu \cite{Radu_SimpleLatt}]\label{prop:Index4}
The finite residual of BMW-group $\Gamma_{6,6}$ coincides with its subgroup $\Gamma_{6,6}^+$ of index~$4$. In particular   $\Gamma_{6,6}^+$ is a finitely presented torsion-free simple group which splits as an amalgamated free product of the form $F_5 *_{F_{25}} F_5$. 

Similarly, the index~$4$ subgroup $\Gamma_{4, 5}^+ \leq \Gamma_{4, 5}$ is simple, and splits as an amalgamated free product of the form $F_3 *_{F_{11}} F_3$. 
\end{prop}
\begin{proof}
By Remark~\ref{rem:Gamma2}, we have  $[x^3, y^4] \in  \Gamma_{\mathrm{JW}}^{(\infty)}$. Moreover $\Gamma_{\mathrm{JW}}$  is a subgroup of $\Gamma_{6,6}$  by Proposition~\ref{prop:BMW-basic}, so that  the finite residual of $\Gamma_{6,6}$ also contains $[x^3, y^4]$. Since  $\Gamma_{6,6}$ is hereditarily just-infinite and non-residually finite, its finite residual is simple and coincides with its smallest non-trivial normal subgroup, see Proposition~\ref{prop:hji}. In particular it coincides with the normal closure of  $[x^3, y^4]$ in  $\Gamma_{6,6}$. Using a computer algebra software like GAP, it takes a couple of seconds to check that the quotient of the just-infinite finitely presented group $\Gamma_{6,6}$ by the normal closure of $[x^3, y^4]$ is of order~$4$. This completes the proof in the case of $\Gamma_{6, 6}$ in view of Proposition~\ref{prop:BMW-basic}.

The proof for   $\Gamma_{4, 5}$ follows a similar outline. However,  additional arguments are required to check that the hypotheses of Theorem~\ref{thm:BaderShalom} are satisfied: indeed,  the degree is too small for  Theorem~\ref{thm:U(alt)} to apply. We refer to  \cite[Theorem~5.5 and Corollary~5.6]{Radu_SimpleLatt} for the details. 
\end{proof}

The original paper of Burger--Mozes \cite{BuMo2} provided the first examples of virtually simple BMW-groups. The degrees of those examples are   rather large, and the finite index simple subgroup of the smallest example decomposes as an amalgamated free product of the form $F_{217} *_{F_{75601}} F_{217}$. The reason why the degree of those examples  is large is that Burger--Mozes relied on Proposition~\ref{prop:FiberProduct} applied to an arithmetic BMW-group $\Gamma$ to build a non-residually finite BMW-group. The smallest example produced in that way already has a rather large degree. The degree was then increased in order to build an example of a BMW-group containing that non-residually finite one, and which moreover has alternating local actions on both tree factors, and is thus also hereditarily just-infinite by Theorems~\ref{thm:BaderShalom} and~\ref{thm:U(alt)}. Much smaller examples were constructed by D.~Rattaggi \cite{Rattaggi_phd, Rattaggi_JGT} using the non-residual finiteness of the double of the Wise lattice. In that way, he obtained a   torsion-free simple group which decomposes as $F_7 *_{F_{73}} F_7$. Exploiting the fact that the Wise lattice itself is non-residually finite, an example decomposing as $F_7 *_{F_{49}} F_7$ is constructed in \cite{BondarenkoKivva}. We refer to \cite{Radu_phd} for a list of over a hundred other  examples of BMW-groups of degree~$(4, 5)$ and $(6,6)$ similar to $\Gamma_{4, 5}$ and  $\Gamma_{6,6}$,  that are virtually isomorphic to a simple group of the form  $F_3 *_{F_{11}} F_3$ or $F_5 *_{F_{25}} F_5$.   
By  \cite[ Corollary~IV]{Radu_SimpleLatt}, the group $\Gamma_{4, 5}^+$ admits the following presentation, which witnesses its amalgam decomposition as $F_3 *_{F_{11}} F_3$:
\begin{align*}
\Gamma_{4, 5}^+ \cong  \langle a, b, c, x, y, z \; | \;  & 
a = x, \\
& b^2 = yx^{-1} y,\\
& c^2 =z^2,\\
& c^{-1}ac = z^{-1}yz,\\
&c^{-1}bc = z^{-1}xz,\\
&b^{-1}ab = y^{-1}x^{-1}y,\\
&b^{-1}c^{-2}b = y^{-1}xz^{-2}y,\\
&b^{-1}c^{-1}b^{-1}acb = y^{-1}xz^{-1}x^{-1}yzx^{-1}y, \\
&b^{-1}c^{-1}abcb = y^{-1}xz^{-1}yz^{-1}xzx^{-1}y, \\
&b^{-1}c^{-1}b^2cb = y^{-1}xz^{-1}xz^{-1}xzx^{-1}y, \\
&b^{-1}c^{-1}b^{-1}cbcb = y^{-1}xz^{-1}x^{-1}z^{-1}xzx^{-1}y \rangle.
\end{align*}

\subsection{The hyperbolic manifold analogy}
\

\bigskip
\begin{flushright}
	\begin{minipage}{.8\linewidth}\parindent24pt
\itshape \small
\indent Nous voyons donc d\'ej\`a que les analystes ne sont pas de simples faiseurs de syllogismes \`a la fa\c con des scolastiques. Croira-t-on, d'autre part, qu'ils ont toujours march\'e pas \`a pas sans avoir la vision du but qu'ils voulaient atteindre? Il a bien fallu qu'ils devinassent le chemin qui y conduisait, et pour cela ils ont eu besoin d'un guide. 

\indent Ce guide, c'est d'abord l'analogie.\upshape

\smallskip

\hfill Henri Poincar\'e, \emph{La Valeur de la Science}, 1905
\end{minipage}

\end{flushright}
\bigskip

As we have seen, the proof of existence of virtually simple BMW-groups by Burger--Mozes elaborates on techniques initially developed by Margulis in his seminal study of irreducible lattices in semi-simple groups of  rank~$\geq 2$. In this speculative section, we discuss an analogy between BMW-complexes and real hyperbolic closed manifolds which suggests that BMW-groups are also strongly related to lattices in simple Lie  groups of rank~$1$. That analogy could serve as an invitation  for further research.

Every rank~$1$ simple Lie group is locally isomorphic to one of the following groups:
\begin{itemize}
\item The isometry group of the real hyperbolic $n$-space, denoted  $O(n, 1)$. 

\item The isometry group of the complex hyperbolic $n$-space, denoted  $U(n, 1)$. 

\item The isometry group of the quaternionic hyperbolic $n$-space, denoted  $\mathrm{Sp}(n, 1)$. For $n =1$, that space is isometric to the real hyperbolic $4$-space. 

\item The isometry group of the octonionic hyperbolic plane,  denoted  $F_4^{-20}$. 
\end{itemize}

Fundamental results of K.~Corlette \cite{Corlette} and Gromov--Schoen \cite{GromovSchoen} ensure that every lattice in $\mathrm{Sp}(n, 1)$ (with $n \geq 2$) and in $F_4^{-20}$ is arithmetic. We refer to \cite{Margulis_book} for the formal definition of arithmeticity. Let us merely mention Margulis' criterion \cite[Theorem~IX.1.10]{Margulis_book} according to which  a lattice in a simple Lie group is arithmetic if and only if its commensurator is discrete. The \textbf{commensurator} of a subgroup $\Gamma$ of a group $G$ is the set 
$$\Comm_G(\Gamma) = \{ g \in G \mid \  [\Gamma: \Gamma \cap g\Gamma g^{-1}] < \infty, \ [\Gamma: \Gamma \cap g^{-1}\Gamma g] < \infty\}, $$  
which is a subgroup of $G$. 

Lattices in $O(n,1)$ and $U(n, 1)$ can be non-arithmetic. While infinitely many examples of non-arithmetic lattices in  $O(n,1)$ are known for all $n \geq 2$, only finitely many non-arithmetic lattices in $U(n, 1)$ are known, see \cite{DPP} and references therein for the current state of the art. Finding infinite families of non-arithmetic lattices in $U(n,1)$ is a major challenge. The case of  $O(n,1)$ is much better understood. In particular, it is known that a huge majority of the lattices in $O(n, 1)$ are non-arithmetic. This is made precise by the following  result, which combines important contributions of various authors. 

\begin{thm} \label{thm:GelanderLevit}
Let $n \geq 4$ and for each $v > 0$, let $C^c_n(v)$ be the number of commensurability classes of real hyperbolic closed manifolds admitting a representative of volume $\leq  v$, which is finite by a classical result of H.~C.~Wang \cite{Wang}. Let also $C^{\mathrm{arith}}_n(v)$  be the number of commensurability classes of real hyperbolic closed manifolds admitting a representative of volume $\leq  v$ whose fundamental group is arithmetic. Then there exist positive constants $a,b, c, \varepsilon$ such that the following inequalities hold for all sufficiently large $v$:
\begin{enumerate}[(i)]
\item {\upshape (Burger--Gelander--Lubotzky--Mozes \cite{BGLM})} $C^c_n(v) \leq v^{bv}$. 

\item {\upshape (Gelander--Levit \cite{GelanderLevit})} $C^c_n(v) \geq v^{av}$.

\item {\upshape (M.~Belolipetsky \cite{Belol})} $C^{\mathrm{arith}}_n(v) \leq v^{c(\log v)^\varepsilon}$.
\end{enumerate}
 
\end{thm}

In particular, we see that the proportion of non-arithmetic  real hyperbolic closed manifolds becomes strikingly overwhelming  as the volume tends to infinity. It is conjectured that the  upper bounded on $C^{\mathrm{arith}}_n(v)$ given in Theorem~\ref{thm:GelanderLevit}(iii) can be improved to a  polynomial bound (i.e. $\varepsilon = 0$). This has in fact been proved by  M.~Belolipetsky for non-compact manifolds;  it is open in the compact case.

The enumeration of BMW-presentations of small degree in \cite{Rattaggi_phd} and \cite{Radu_phd} suggests a similar counting problem. The direct analogue would be to fix the degree $(m, n)$ and count the number of commensurability classes of cocompact lattices in $\Aut(T_m) \times \Aut(T_n)$ as a function of the covolume. However, formulating any conjecture about that number is premature, since there is currently no known evidence that the number of those commensurability classes actually grows at all for \emph{all} values of $m$ and $n$. Instead, we focus on BMW-groups of degree $(m, n)$ (so that the covolume is fixed, once the Haar measure on $\Aut(T_m) \times \Aut(T_n)$ is normalized so as to give measure~$1$ to the vertex stabilizers) and address the counting problem of  the number of their commensurability classes  as a function of the degree. A BMW-group is called \textbf{arithmetic} if the closure  of its projections to the automorphism groups of the tree factors of its Cayley graph are both rank~$1$ simple algebraic groups over local fields.

\begin{prob}\label{prob:BMW}
Let $\mathrm{BMW}(m, n)$ be the number of commensurability classes of  BMW-groups of degree $(m, n)$, and let moreover  $\mathrm{BMW}^{\mathrm{arith}}(m, n)$ be the number of those classes that have an arithmetic representative. Determine the asymptotic growth type of $\mathrm{BMW}(m, n)$ and $\mathrm{BMW}^{\mathrm{arith}}(m, n)$  as functions of $(m, n)$. In particular, determine whether there exist constants $a, b, c$ such that the following inequalities hold for all sufficiently large  $m, n \in \mathbf N$, where $w = m+n$: 
\begin{enumerate}[(i)]
\item $w^{aw} \leq \mathrm{BMW}(m, n) \leq w^{bw}$. 

\item $\mathrm{BMW}^{\mathrm{arith}}(m, n) \leq w^{c}$. 

\end{enumerate}

\end{prob}

We note that any two finitely generated free groups are commensurable, so that the \emph{reducible} BMW-complexes only contribute one commensurability class per degree.  We also note that the  Buhat--Tits tree of a rank~$1$ simple algebraic group over a local field is always semi-regular of bidegree $(q+1, q'+1)$, where $q$ and $q'$ are both powers of the same prime. Thus for most values of $(m, n)$, there is no  arithmetic BMW-group of degree $(m,n)$ whatsoever.  As pointed out to me by A.~Vdovina, the results by Stix--Vdovina \cite{StixVdovina} ensures that if $2n-1$ is a sufficiently large prime power, then the number of commensurability classes of arithmetic torsion-free BMW-groups of degree $(2n, 2n)$ is bounded \emph{below}  by a linear function of $n$. 

The difficulty in controlling the commensurability classes is that the currently known tools do not provide detailed enough information on the finite quotients of a BMW-group. We have pointed out the existence of BMW-groups of degree $(4, 5)$ and  $(6,6)$ that are virtually simple. The proof of non-residual finiteness relied on a very specific property, namely the existence of a BMW-subgroup that is not residually finite.  

The results of N.~Radu \cite{Radu_phd} show that there are $225145$ isomorphism classes of BMW-presentations of degree $(6,6)$ whose generators have infinite order. Among them, $23225$ yield a BMW-group whose  local action on both tree factors  is isomorphic to $\Alt(6)$ or $\Sym(6)$; these groups are all hereditarily just-infinite, see Proposition~\ref{prop:NonExistence666} below. I expect that they are \emph{all} virtually simple. However, this has only been checked for $96$ of them (and these $96$ groups have moreover been proved to be pairwise non-commensurable). Those are precisely the members of that list of $23225$ BMW-groups that contain an isomorphic copy of the non-residually finite group $\Gamma_{\mathrm{JW}}$. The very small size of the ratio $96/23225\simeq 0.4\%$ illustrates our lack  of understanding: approaching Problem~\ref{prob:BMW}  requires new and much finer criteria to check   non-residual finiteness. 

The following question, which overlaps several problems posed by D.~Wise in \cite[Section~10]{Wise_CSC}, provides further illustrations of the limitations of  the currently known tools. 

\begin{prob}\label{pb:BMW}
Is there an algorithm which determines whether a BMW-group given by a BMW-presentation is irreducible? Residually finite? Linear? Arithmetic? Just-infinite? Virtually simple? Has a discrete commensurator in the full automorphism group of the associated product of trees? Is there  an algorithm which determines whether two such groups are isomorphic? Commensurable? 
\end{prob}

A complete solution to Problem~\ref{pb:BMW} would of course be spectacular, but such a complete result is not necessary to tackle Problem~\ref{prob:BMW}. Indeed, it could be that the BMW-groups with a locally alternating or fully symmetric action on both tree factors already contribute enough commensurability classes to dominate the growth function from the conjecture. 

\subsection{Local actions of just-infinite groups acting on trees}

We have seen how the combination of Theorems~\ref{thm:BaderShalom} and~\ref{thm:U(alt)} (and its companion Theorem~\ref{thm:Radu}) can be used to build cocompact lattices $\Gamma \leq \Aut(T) \times \Aut(T')$ in the automorphism group of a product of two regular locally finite trees of degrees $d$ and $d'$ respectively, that are (hereditarily) just-infinite. A specific feature of those examples, coming from the hypotheses of Theorem~\ref{thm:U(alt)}, is that the local actions of $\Gamma$ on both $T$ and $T'$  contain $\Alt(d)$ and $\Alt(d')$. We have also seen that other just-infinite lattices in products of trees arise as arithmetic groups (see Propositions~\ref{prop:StixVdovina},~\ref{prop:Rattaggi} and~\ref{prop:Rungtana}); in those examples, the local actions on both tree factors are finite Lie type groups of rank~$1$ acting on projective lines over finite fields (see Table~\ref{tab:ExamplesLocalActions}). In \emph{all} known examples of just-infinite lattices in products of trees, the local action on each tree factor is $2$-transitive. This observation suggests the following. 

\begin{prob}\label{prob:LocalAction2trans}
	Let $T_1, \dots, T_n$ be locally finite trees all of whose vertices have degree~$\geq 3$,  and   $\Gamma \leq \Aut(T_1) \times \dots \times \Aut(T_n)$ be a discrete subgroup acting cocompactly on  $T_1  \times \dots \times T_n$. 
	
	Assume that $\Gamma$ is just-infinite. Must the local action of $\Gamma$ on $T_i$ be $2$-transitive for all $i =1, \dots, n$?
\end{prob}

The condition that $\Gamma$ be just-infinite implies that $n \geq 2$, since a discrete cocompact automorphism group of a single infinitely-ended tree is virtually a non-abelian free group. One could also ask the following more general question. 

\begin{prob}\label{prob:LocalAction:SQ}
	Let $T$ be a locally finite tree all of whose vertices have degree~$\geq 3$, and $\Gamma  \leq \Aut(T)$ be a (not necessarily discrete)  subgroup acting cocompactly. 
	
What are the possible local actions of $\Gamma$ at vertices of $T$ if $\Gamma$ is finitely generated and just-infinite? 
\end{prob}

As an illustration of this problem, let us mention that  if $\Gamma$ is finitely generated and its local action at every vertex of $T$ is nilpotent, then $\Gamma$ is virtually indicable (see \cite[Corollary~1.2]{CapraceWesolek}); in particular $\Gamma$ cannot be just-infinite if $T$ has infinitely many ends. Problem~\ref{prob:LocalAction:SQ} is closely related to the statement (2) given without proof in \cite[Section~9.15]{BassLubotzky} and attributed to E.~Rips, according to which a group $\Gamma$ of the form $A *_C B$ with $C \neq B$ and such that $|C\backslash A / C| \geq 3$ must be SQ-universal. In other words, that claim would imply that if $\Gamma$ is edge-transitive on $T$ and if every vertex of  $T$ has valency~$\geq 2$, then $\Gamma$ is SQ-universal as soon as its local action at one vertex of valency~$\geq 3$ fails to be $2$-transitive. However, that statement must be amended: indeed, the following result of A.~Le Boudec provides a counterexample. 

\begin{thm}[A. Le Boudec] \label{thm:LeBoudec}
There is a finitely generated infinite simple group $\Gamma$ acting edge-transitively on the regular tree $T$ of degree~$20$, and whose local action at every vertex is isomorphic to the permutational wreath product $\Alt(4) \wr \Alt(5)$. In particular the local action of $\Gamma$ at every vertex of  $T$ is not primitive.
\end{thm}
\begin{proof}
	We apply \cite[Theorem~1.3]{LeBoudec} to the group $F = C_2 \times C_2 \times  C_5$ of order~$20$, and the group $F' = \Alt(4) \wr \Alt(5)$, both viewed as transitive subgroups of $\Sym(20)$. Notice that the permutation group $F$ can naturally be viewed as a subgroup of $F'$. Since $F$ acts freely and $F'$ is generated by the derived subgroups of its point stabilizers, the existence of the required simple group $\Gamma \leq \Aut(T)$, arising as an index~$2$ subgroup of a vertex-transitive group denoted by $G(F, F')$, indeed follows from  \cite[Theorem~1.3]{LeBoudec}. 
\end{proof}

In particular, a finitely generated infinite simple group can act edge-transitively on a tree with an imprimitive local action at every vertex.  Note however that the group $\Gamma$ from Theorem~\ref{thm:LeBoudec} is   very different from the projection of a lattice as in Problem~\ref{prob:LocalAction2trans}, so that Theorem~\ref{thm:LeBoudec} should not be interpreted as evidence supporting a negative solution to Problem~\ref{prob:LocalAction2trans}. Indeed, in the set-up of Problem~\ref{prob:LocalAction2trans}, the group $\Gamma$ is finitely presented, and the vertex-stabilizers for the $\Gamma$-action on each tree factor $T_i$ are not torsion groups (they act properly and cocompactly on the product of the tree factors different from $T_i$). On the other hand, Le Boudec's group has locally finite vertex stabililizers  (see the discussion at the end of Section~3.1 in \cite{LeBoudec}); moreover it is  not finitely presented (see \cite[Proposition~5.4]{LeBoudec}). 

\subsection{Lattices in products of more than two trees}

We close this chapter with a discussion of another fascinating and natural problem. Products of two trees are part of the definition of a  BMW-group, but it is natural to consider also products of more than two factors. In view of the existence of virtually simple BMW-groups, the following problem is especially intriguing. 

\begin{prob}\label{pb:3factors}
Let $T_1, \dots, T_n$ be locally finite trees with infinitely many ends and   $\Gamma \leq \Aut(T_1) \times \dots \times \Aut(T_n)$ be a discrete subgroup acting cocompactly on  $T_1  \times \dots \times T_n$. 

Can $\Gamma$ be   simple if $n \geq 3$? 
\end{prob}

This problem may be viewed as a \textit{higher rank version} of P.~Neumann's Question~\ref{ques:Neumann}. 

The arithmetic constructions in Section~\ref{sec:Margulis_NST} provide examples of cocompact lattices  in products of an arbitrarily large number of factors that are hereditarily just-infinite. Explicit examples acting vertex-transitively on the associated product of trees may be found in \cite[Corollary~6.2]{ChinburgStover} and \cite{CFHKSSZZ}. All those groups are linear, hence residually finite. There is currently no known example of a lattice in a product of more than two \emph{non-linear} locally compact groups satisfying the hypotheses of  the Bader--Shalom Normal Subgroup Theorem~\ref{thm:BaderShalom}. A non-existence result for irreducible lattices in products of three or more factors has been established in the restricted class of certain locally compact Kac--Moody groups in \cite{CaMo_KM}. The techniques used in loc.~cit. are very specific to Kac--Moody theory, and yield no relevant information in the context of Problem~\ref{pb:3factors}. In the case of tree lattices, the following    result was   established by N.~Radu with the aid of a computer. 

\begin{prop}[{N.~Radu \cite[Theorem~VIII]{Radu_SimpleLatt}}]\label{prop:NonExistence666}
Let $T, T', T''$ be three copies of the regular tree of degree $6$. 
\begin{enumerate}[(i)]
\item There are $23225$ conjugacy classes of subgroups $\Gamma \leq \Aut(T) \times \Aut(T')$ acting simply transitively on the vertices of the product $T\times T'$, whose local action on both tree factors is $\Alt(6)$ or $\Sym(6)$, and such that $\Gamma_v$ and $\Gamma_{v'}$ are torsion-free for all $v \in V(T)$ and $v' \in V(T')$. All of them  are hereditarily just-infinite.  Among them, $2240$ are torsion-free. 

\item There is \textbf{no}  subgroup $\Gamma \leq \Aut(T) \times \Aut(T') \times \Aut(T'')$ acting simply transitively on the vertices of the product $T\times T' \times T''$,  such that the following  conditions hold, where $G$, $G'$ and $G''$ denote the closure of the respective projections of $\Gamma$ to $\Aut(T)$, $\Aut(T')$ and $\Aut(T'')$:
\begin{itemize}
	\item the groups $G$, $G'$ and $G''$ are non-discrete and their respective local actions at  every vertex of $T$, $T'$ and $T''$ is  $\Alt(6)$ or $\Sym(6)$;
%	\item the projections of $\Gamma$ to $G \times G'$ and to $G' \times G''$ are dense;
    \item $\Gamma_{v, v'}$ and $\Gamma_{v', v''}$ are torsion-free  for all $(v, v', v'') \in V(T \times T' \times T'')$.
\end{itemize}
\end{enumerate}

\end{prop}

Part (i) of Proposition~\ref{prop:NonExistence666} relies on a version of Mostow rigidity for irreducible cocompact lattices in products of trees with primitive local action, due to Burger--Mozes--Zimmer, see \cite[Theorem~1.4.1]{BMZ}. Part (ii) relies in an essential way on (i), see \cite[Theorem~VIII]{Radu_SimpleLatt}. In particular, it uses the Normal Subgroup Theorem in the case of two factors. Notice that the condition on the local action hypothetized in Proposition~\ref{prop:NonExistence666} ensures that $G$, $G'$ and $G''$ are subjected to Theorem~\ref{thm:Radu}. 
That condition is rather natural, especially if one expects a positive solution to Problem~\ref{prob:LocalAction2trans} (bearing in mind that `almost all' finite $2$-transitive groups are $\Alt(d)$ or $\Sym(d)$; see \cite[Corollary~B.2]{Radu_InvMath} for a precise statement clarifying the latter claim). 

The  contrast between the case of~$2$ and~$3$ factors in Proposition~\ref{prop:NonExistence666} is   striking. I interpret it as experimental evidence for a negative answer to Problem~\ref{pb:3factors}.

\section{Quotients of hyperbolic groups and asymptotic properties of finite simple groups}

In this final chapter, the difference between the construction of finite and infinite quotients of finitely presented groups is further illustrated by the  discussion of a major open problem in Geometric Group Theory, namely the residual finiteness of hyperbolic groups. 

\subsection{Examples of hyperbolic groups}\label{sec:Examples_hyp}

Hyperbolic groups form a class of groups introduced and developed by M.~Gromov in \cite{Gromov_hyp}. Their definition can be seen as an axiomatization of the fundamental groups of hyperbolic closed manifolds. Extensive treatments of the basic theory can be consulted in  \cite[Chapter~III.H]{BH},   \cite{CDP} or \cite{GhysHarpe}. Let us record a (non-exhaustive!)  list of  examples. 

\begin{itemize}
\item Finite groups, and more generally virtually cyclic groups. Those form the so-called \textbf{elementary hyperbolic groups}. 

\item Virtually free groups. Those include free amalgamated products of finite groups, e.g. the free product $C_a * C_b$ of cyclic groups of order $a$  and $b$. 

\item Fundamental groups of closed surfaces of genus~$g \geq 2$. 

\item Hyperbolic triangle groups. Those are groups of the form   $T(p, q, r) = \langle x, y \mid x^p, y^q, (xy)^r \rangle$ with $\frac 1 p + \frac 1 q + \frac 1 r < 1$. They are commensurable with surface groups. 

\item One-relator groups with torsion. 

\item Coxeter groups which do not contain $\mathbf Z \times \mathbf Z$. The latter condition can be characterized in terms of the Coxeter presentation, see \cite[\S 12.6]{Davis_book}. 

\item Fundamental groups of closed Riemannian manifolds of negative sectional curvature. 
\end{itemize}

\subsection{Finite and infinite quotients of hyperbolic groups}

Hyperbolic groups enjoy numerous remarkable algebraic properties. Let us collect a few of those.

\begin{thm}\label{thm:hyperbolic-basic}
Every  hyperbolic group  $G$ satisfies the following. 
\begin{enumerate}[(i)]
\item {\upshape (\cite[Corollary~2.2.A]{Gromov_hyp})} $G$ is finitely presented. 

\item {\upshape (\cite[Chapter~8, Theorem~37]{GhysHarpe})} Every subgroup of $G$ is either virtually cyclic, or contains a non-abelian free subgroup. 

\item {\upshape (T.~Delzant {\cite[Theorem~3.5]{Delzant96}}; A. Olshanskii \cite{Olsh_SQ})} If $G$ is non-elementary, then $G$ is SQ-universal. 

\end{enumerate}
\end{thm}

The SQ-universality of non-elementary hyperbolic groups is a vast generalization of the classical theorem of Higman--Neumann--Neumann~\cite{HNN} according to which the free group $F_2$ is SQ-universal. It was announced by M.~Gromov~\cite[\S5.6.E]{Gromov_hyp} (without proof). It underlines the huge supply of \emph{infinite} proper quotients that every non-elementary hyperbolic group has. Since the early days of the theory, the question of existence of \emph{finite} proper quotients  arose naturally. The following question is a major open problem in Geometric Group Theory. 

\begin{prob}\label{prob:hyp_RF}
Are all hyperbolic groups residually finite?
\end{prob}
 
In his foundational paper on hyperbolic groups, M.~Gromov suggested that the answer should be negative by writing the following.

\begin{rmk}[{M.~Gromov \cite[\S 5.3.B]{Gromov_hyp}}]
\textit{Probably, ``generic'' word hyperbolic groups admit no sequences of subgroups of finite index with trivial intersection.}
\end{rmk}

Gromov's remark has sometimes been  interpreted as a conjecture predicting the existence of a non-elementary hyperbolic group whose only finite quotient is the trivial one (see for example Olshanskii's comment following Theorem~2 in \cite{Olsh_2000}).  
While the latter conjecture clearly implies a negative answer to Problem~\ref{prob:hyp_RF}, it turns out that, conversely, the existence of a non-residually finite hyperbolic group would imply that the conjecture is true. This was observed independently by Kapovich--Wise \cite{KaWi} and A.~Olshanskii \cite{Olsh_2000}. 

Problem~\ref{prob:hyp_RF} has  motivated a tremendous amount of research; it goes beyond the scope of this article to survey it all. We will only emphasize specific results that we find most relevant to the general theme of these notes. The following statement illustrates the strength and scope that a positive solution to Problem~ \ref{prob:hyp_RF} would have. 

\begin{thm}\label{thm:if_hyp_RF}
If all hyperbolic groups are residually finite, then the following assertions hold. 
\begin{enumerate}[(i)]
\item {\upshape (Kapovich--Wise \cite[Theorem~5.1]{KaWi})} Every hyperbolic group is virtually torsion-free. 

\item {\upshape (Agol--Groves--Manning \cite{AGM})} Every quasi-convex subgroup of a every hyperbolic group is separable. 

\item {\upshape (A.~Lubotzky \cite[Remark~4.2]{Lubot})} Cocompact lattices in $\mathrm{Sp}(n, 1)$ (with $n \geq 2$) and $F_4^{-20}$ do not satisfy the Congruence Subgroup Property. 

\end{enumerate}
\end{thm}

A \textbf{quasi-convex subgroup}  of a hyperbolic group $G$ is a finitely generated subgroup $H$ such that the inclusion of $H$ in $G$ is a quasi-isometric embedding. In particular, quasi-convex subgroups are themselves hyperbolic. 

One may interpret Theorem~\ref{thm:if_hyp_RF}(ii) as follows: if all hyperbolic groups are residually finite, then all hyperbolic groups have a tremendous amount of finite quotients. This assertion can be made more precise using the following. 

\begin{prop}\label{prop:QCERF}
Let $G$ be a group and $H$ be a subgroup. If every finite index subgroup of $H$ is separable in $G$, then every homomorphism  $\varphi  \colon H \to Q$ to a finite group extends to a homomorphism $\tilde \varphi \colon \tilde G \to Q$ defined on a finite index subgroup $\tilde G$ of $G$ containing $H$. 

In particular, if $G$ is a non-elementary hyperbolic group all of whose quasi-convex subgroups are separable, then every finite group is a quotient of a finite index subgroup of $G$. 
\end{prop}
\begin{proof}
The first assertion follows from the proof of \cite[Theorem~4.0.7]{LongReid_2005}. The second follows from the first together with the fact that every non-elementary hyperbolic group admits quasi-convex subgroups that are free of arbitrarily large rank. 
\end{proof}

One way to establish Assertion~(iii) in Theorem~\ref{thm:if_hyp_RF} is to deduce it from Assertion~(ii) and Proposition~\ref{prop:QCERF}. Indeed, the property that all finite groups appear as virtual quotients is incompatible with the congruence subgroup property. 

We conclude this section by mentioning that a far-reaching generalization of Theorem~\ref{thm:hyperbolic-basic}(iii) on the SQ-universality of non-elementary hyperbolic groups was recently established by Dahmani--Guirardel--Osin \cite{DGO}, who indeed  showed that all the so-called \textbf{acylindrically hyperbolic groups} are SQ-universal. Acylindrically hyperbolic groups form an immensely vast class of groups, formalized by D.~Osin \cite{Osin_acyl}. They include all non-elementary hyperbolic groups, as well as numerous other examples of a very different nature, see \cite[\S 8]{Osin_acyl}. For example, it is proved in \cite[Corollary~4.26]{MinasyanOsin} that the Higman group from Theorem~\ref{thm:Higman} is acylindrically hyperbolic. Another illustration of this concept is provided by the following result, giving a rather general perspective on the fact that the Margulis' Normal Subgroup Theorem fails in the rank~$1$ case.  

\begin{thm}\label{thm:LatticeInHyp}
Let $G$ be a  locally compact group which is Gromov hyperbolic with respect to the word metric associated with a compact generating set. Assume that $G$ does not contain a discrete cocompact cyclic subgroup.  Then every lattice $\Gamma \leq G$ is acylindrically hyperbolic, hence SQ-universal. 
\end{thm}
\begin{proof}
By \cite[Proposition~2.1]{CCMT}, the group $G$ has a continuous proper cocompact action on a Gromov hyperbolic proper geodesic metric space $X$. Any discrete subgroup of $G$ thus acts properly on $X$. Now, if a discrete subgroup  $\Gamma \leq G$ does not contain any loxodromic elements, then by \cite[Proposition~5.5]{CaFu} either $\Gamma$ has a bounded orbit on $X$, or $\Gamma$ has a unique fixed point in the Gromov boundary $\partial X$. If $\Gamma$ is a lattice, then the former condition would imply that $X$ carries a $G$-invariant probability measure, hence that $G$ has a bounded orbit (since a sufficiently large ball in $X$ would have measure~$>\frac 1 2$), while the latter would imply that the Gromov boundary $\partial X$ carries a $G$-invariant probability measure, which is incompatible with the contracting dynamics of the $G$-action by the hypothesis that $G$ is non-elementary. Thus every lattice in $G$ acts properly on $X$ and contains a loxodromic element. The required conclusions now follow from \cite[Theorems~1.2 and~8.1]{Osin_acyl}. 
\end{proof}

\subsection{Hyperbolic quotients of hyperbolic groups, after A.~Olshanskii}

Our next goal is to emphasize the relation between Problem~\ref{prob:hyp_RF} and the asymptotic properties of finite simple groups. A relation of that kind was first highlighted by Ivanov--Olshanskii \cite[Problem~2]{IvanovOlshanskii}. 

Let us first recall from \cite[Corollary~8.2.C]{Gromov_hyp} that every infinite order element $g$ of a hyperbolic group $G$ is contained in a unique maximal elementary quasi-convex subgroup $E_G(g)$, called the \textbf{elementary closure} of $g$. Given a subgroup $H \leq G$, we define $E_G(H)$ as the intersection of $E_G(h)$ where $h$ runs over the set of all infinite order elements of $H$. By \cite[Proposition~1]{Olsh_residualing}, if $H$ is not an elementary quasi-convex subgroup of $G$, then $E_G(H)$ coincides with the largest finite subgroup of $G$ normalized by $H$. 

\begin{thm}[{A.~Olshanskii \cite[Theorem~2]{Olsh_residualing} and \cite[Lemma~5.1]{Olsh_2000}}] \label{thm:Olshanskii}
Let $G$ be a non-elementary hyperbolic group and $H_1, \dots, H_k$ be non-elementary subgroups with $E_G(G) = E_G(H_1) = \dots = E_G(H_k) = \{1\}$. Then there is a homomorphism $\varphi \colon G \to \tilde G$ onto a non-elementary hyperbolic group $\tilde G$ such that $\varphi(H_i) = \tilde G$ for all $i$. 
\end{thm}

Given two  non-elementary hyperbolic groups $G_1, G_2$,  let $H_i = G_i/E_{G_i}(G_i)$ and form the free product $G = H_1 * H_2$, which is hyperbolic. Applying Theorem~\ref{thm:Olshanskii} to $G$, we obtain the following. 
 
\begin{cor}[A.~Olshanskii] \label{cor:Olsh}
Any two non-elementary hyperbolic groups have a common non-elementary hyperbolic quotient. 
\end{cor}

A first connection with finite simple groups arises in the following. 

\begin{cor}\label{cor:Olsh2}
If all hyperbolic groups are residually finite, then for every non-elementary hyperbolic group $G$ and every $n \geq 5$, the group $G$ has a finite simple quotient containing a copy of $\Alt(n)$. 
\end{cor}
\begin{proof}
We use Corollary~\ref{cor:Olsh} to construct a common hyperbolic quotient $Q$ of $G$ and the virtually free group $\Alt(n) * \Alt(n)$, which is perfect. If $Q$ has a non-trivial finite quotient, then a smallest such quotient is a finite simple quotient of $G$ which contains $\Alt(n)$ as a subgroup. 
\end{proof}

One may now contemplate again the list of examples of hyperbolic groups mentioned above and wonder whether it contains candidates of groups that do not map onto non-abelian finite simple groups of arbitrarily large rank. We emphasize that the question of determining the finite simple quotients of the virtually free group $C_a * C_b$ or of the hyperbolic triangle group $T(p, q, r)$ are important topics of current investigations, see \cite{King}, \cite{Liebeck_survey} and references therein. 

Combining Theorem~\ref{thm:if_hyp_RF} and Proposition~\ref{prop:QCERF}, we see that if all hyperbolic groups were residually finite, then every finite group would be a quotient of a finite index subgroup of every hyperbolic group. The following immediate consequence of Corollary~\ref{cor:Olsh2} strengthens that fact. 

\begin{cor}\label{cor:Olsh3}
If all hyperbolic groups are residually finite, then every finite group embeds in some finite quotient of every non-elementary hyperbolic group.
\end{cor}

That property can be viewed as a finite counterpart of the SQ-universality of hyperbolic groups. 

Let us finally record a last important consequence of Theorem~\ref{thm:Olshanskii}.

\begin{cor}\label{cor:Olsh4}
If all hyperbolic groups are residually finite, then for every non-elementary hyperbolic group $G$ with $E_G(G)= \{1\}$, every finite subset $M \subset G$ and every $n \geq 5$, there is a homomorphism $\rho \colon G \to Q$ of $G$ onto a finite simple group $S$ containing a subgroup isomorphic to $\Alt(n)$, and such that the restriction of $\rho$ to $M$ is injective. 
\end{cor}
\begin{proof}
For every pair $x \neq y \in M$, let $H_{x, y}$ denote the normal closure of $xy^{-1}$ in $G$. Since $E_G(G)= \{1\}$, it follows that $H_{x, y}$ is a non-elementary normal subgroup of $G$, which moreover satisfies  $E_G(H_{x, y}) = \{1\}$. By Theorem~\ref{thm:Olshanskii}, the hyperbolic group $G * \Alt(n) * \Alt(n)$ admits a non-elementary hyperbolic quotient $\varphi \colon G * \Alt(n) * \Alt(n) \to Q$ such that $\varphi(H_{x, y}) = Q$ for all $x \neq y \in M$ and moreover $\varphi(\Alt(n) * \Alt(n)) = Q$. 

Assume that $Q$ is residually finite, and let $\psi \colon Q \to S$ be a smallest non-trivial finite quotient of $Q$. Since $Q$ is generated by two copies of $\Alt(n)$, it is perfect and so $S$ is a finite simple group containing $\Alt(n)$. Let now $x \neq y \in M$.  Since $H_{x, y}$ is the normal closure of $xy^{-1}$ in $G$ and $\varphi$ is surjective, it follows that $\varphi(H_{x, y}) = Q$ is the normal closure of $\varphi(xy^{-1})$ in $Q$. In particular $\rho(H_{x, y}) = S$ is the normal closure of $\rho(xy^{-1})$ in $S$. Therefore $\rho(x) \neq \rho(y)$. 
\end{proof}

We recall that if $P$ is a group property (e.g. being finite, or nilpotent, or solvable), a group $G$ is called \textbf{residually $P$} if every non-trivial element of $G$ remains non-trivial in some quotient of $G$ satisfying $P$. Moreover $G$ is called  \textbf{fully residually $P$} if for every finite subset $M$ of  $G$, there is a quotient map $\rho \colon G \to Q$ onto some group satisfying $P$, which is moreover injective on $M$.

With that terminology at hand, we deduce from Corollary~\ref{cor:Olsh4} that if  all hyperbolic groups are residually finite, then   every non-elementary hyperbolic group $G$ with $E_G(G)= \{1\}$  is \textbf{fully residually finite simple}. The following neat observation was pointed out to me by P.~Neumann. 

\begin{prop}[P.~Neumann]\label{prop:Neumann}
Let $G$ be a group, all of whose non-trivial normal subgroups have a trivial centralizer. For any group property $P$, the group $G$ is residually $P$ if and only if  $G$ is fully residually $P$. 	

In particular, a non-elementary hyperbolic group $G$ with $E_G(G)=\{1\}$ or, more generally, an acylindrically hyperbolic group with trivial amenable radical,  is residually $P$ if and only if it is fully residually $P$. 
\end{prop}
\begin{proof}
Clearly, if $G$ is fully residually $P$ then it is residually $P$. In order to prove the converse, it suffices to show that for any non-empty finite subset $M \subset G$, there is a  quotient map $\rho \colon G \to Q$ onto some group satisfying $P$, such that $\rho(x) \neq 1$ for all $x \in M$ with $x \neq 1$. We do this by induction on $|M|$, the base case $|M|=1$ being clear by the hypothesis that $G$ is residually $P$. 

Assume now that $|M|>1$. If $M$ contains only one non-trivial element, then the result is clear since $G$ is residually $P$. We may therefore assume that $M$ contains two non-trivial elements, say $x$ and $y$. Let $N$ be the normal closure of $y$ in $G$. Thus $N$ is the subgroup of $G$ generated by the conjugacy class of $y$. Since $N$ has a trivial centralizer by hypothesis, we have $x \not \in \mathrm C_G(N)$, so that there exists $g \in G$ with $[x, gyg^{-1}] \neq 1$. Let $M' = M \setminus \{x, y\} \cup \{[x, gyg^{-1}]\}$. By induction, there is a  quotient map $\rho \colon G \to Q$ onto some group satisfying $P$, such that $\rho(z) \neq 1$ for all $z \in M'$ with $z \neq 1$. In particular $\rho(x) \neq 1 \neq \rho(y)$. Thus $\rho(z) \neq 1$ for all $z \in M$ with $z \neq 1$, as required.

If $G$ is non-elementary hyperbolic, then $E_G(G)$ coincides with the amenable radical of $G$. Assume that  $G$ is acylindrically hyperbolic with trivial amenable radical. Let $N$ be a non-trivial normal subgroup. Then   $ N \mathrm C_G(N)$ is acylinrically hyperbolic by \cite[Corollary~1.5]{Osin_acyl}, so by (the proof of) \cite[Lemma~7.3]{Osin_acyl}, either $N$ or $\mathrm C_G(N)$ is contained in the amenable radical of $G$. Since the latter is trivial while $N$ is non-trivial, we conclude that $\mathrm C_G(N)$ is trivial. 
\end{proof}

\subsection{The space of marked groups}

The  remark preceding Proposition~\ref{prop:Neumann} can be nicely interpreted in the framework of the space of marked groups. This is a compact Hausdorff topological space that was alluded to by M.~Gromov in \cite{Gromov_poly}, and formally defined by R.~Grigorchuk \cite{Gri}. In this section, we follow the presentation of Champetier--Guirardel \cite[\S 2]{ChamGuir}.

A \textbf{($d$-generated) marked group} is a pair $(G, S)$ consisting of a group $G$ and a $d$-tuple $(s_1, \dots, s_d)$ which generates $G$. Let $F_d$ be  the free group of rank $d$ and fix a generating $d$-tuple $(a_1, \dots, a_d)$. Every $d$-generated marked group $(G, S)$ gives rise to a unique surjective homomorphism $F_d \to G$ with $a_i \mapsto s_i$. Thus the set of isomorphism classes of $d$-generated marked groups can be identified with the set of normal subgroups of $F_d$. That set carries a natural  totally disconnected compact Hausdorff topology, namely the \textbf{Chabauty topology}. The \textbf{space  of $d$-generated marked groups} is the compact space consisting of the isomorphism classes of $d$-generated marked groups. An alternative way to see it is by associating  to a marked group $(G,S)$ its Cayley graph $\Cay(G,S)$, viewed as a directed edge-labeled graph with labels in $S$. Any isomorphism class of $d$-generated marked groups admits a representative of the form $(H, S)$, so that the Cayley graph of $H$ with respect to $S$ has the same degree and the same set of labels as the Cayley graph of $(G,S)$. A basic neighbourhood  of the isomorphism class of $(G,S)$ in the space of marked groups consists of the classes admitting a respresentative $(H,S)$ such that the $n$-ball around the identity in $\Cay(G,S)$ and $\Cay(H,S)$ are isomorphic as directed, edge-labeled graphs.  It is customary to identify a marked group $(G,S)$ with its isomorphism class; that abuse of language and notation should not cause any confusion.  

Here are a few examples of  converging sequences in the space of marked groups. 

\begin{itemize}
\item The sequence $(\mathbf Z/n\mathbf Z, \{1 + n \mathbf Z\})$ converges to  $(\mathbf Z, \{1\})$ as $n$ tends to infinity in the space of cyclic marked groups. 

\item The sequence $(\PSL_n(\FF_p), E_p)$ converges to $(\PSL_n(\ZZ), E)$ as the prime $p$ tends to infinity, where $E_p$ and $E$ are the images of the sets of elementary matrices. 

\item If $G = \langle S \mid r_1, r_2, \dots \rangle$ is an infinite presentation of a group $G$, then $(G,S)$ is a limit of the sequence $(G_n, S)$, where $G_n = \langle S \mid r_1,  \dots, r_n \rangle$.
\end{itemize}

The following basic fact is well known and easy to see. 

\begin{lem}\label{lem:FP}
Let  $G$ be a finitely presented group. For every marking $(G,S)$, the set 
$$\{ (G/N,SN/N) \mid  N \lhd G\}$$ 
of marked quotients of $G$ is a neighbourhood of $(G,S)$ in the space of marked groups. 
\end{lem}

Since finitely generated nilpotent groups are finitely presented, it follows that a non-trivial nilpotent group cannot be a limit of perfect groups in the space of marked groups. 

It is a natural problem to study the limits of  non-abelian finite simple groups. We emphasize the following question. 

\begin{prob}\label{prob:limits}
\begin{enumerate}[(i)]
\item Find an algebraic characterization of the  limits of non-abelian finite simple groups in the space of marked groups.  

\item 
Can an infinite group of finite exponent be a limit of non-abelian finite simple groups? 

\end{enumerate}
\end{prob}
 
 It should be emphasized that being a limit of 	a sequence of non-abelian finite simple groups is an algebraic property which is independent of any choice of marking. Indeed, this follows from a general observation due to Cornulier--Guyot--Pitsch \cite[Lemma~1]{CGP}, so that  Problem~\ref{prob:limits}(i)   is well-posed. 
 
A  more precise version of Problem~\ref{prob:limits}(ii) is proposed by Ivanov--Olshanskii in \cite[Problem~2]{IvanovOlshanskii}. 

In an earlier version of these notes, I also asked whether a metabelian group can be a limit of non-abelian finite simple groups, or whether a limit of non-abelian finite simple groups  can satisfy a law. The latter two questions were positively answered  by Y.~Cornulier, who pointed out the following. 

\begin{prop}[Y.~Cornulier]\label{prop:Cor}
For any prime $p$, the wreath product $C_p \wr \mathbf Z$ is a limit of alternating groups of prime degrees in the space of marked groups. In particular $\mathbf Z \wr \mathbf Z$  is also a limit of alternating groups. 
\end{prop}
\begin{proof}
Let us start with the case where $p$ is odd. Let $q $ be another prime with $q > 2p$.  Let $a =(1, \dots, p)$ and $b = (1, \dots, q)$ be cyclic permutations in $\Sym(q)$. Notice that $G= \langle a, b\rangle$ is contained in $\Alt(q)$. We claim that $G= \Alt(q)$. 

First observe that $a$ and $b^{p-1} a b^{-p+1}$ are two $p$-cycles whose support overlap in the singleton $\{p\}$. Therefore they generate a group preserving the set  $\{1, \dots, 2p-1\}$  and whose action on that set is $2$-transitive by \cite[Lemma 3.1]{BCGM}. A classical result of C.~Jordan (see  \cite[Theorem~13.9]{Wielandt}) ensures that a primitive group of degree $k$ containing a prime cycle of order $p \leq k-3$ contains $\Alt(k)$. Using that fact (and a direct computation in case $p=3$), we see that $\langle a, b^{p-1} a b^{-p+1} \rangle \cong \Alt(2p-1)$. It is easy to see that the group generated by the $\langle b \rangle$-conjugates of $\Alt(2p-1)$ is $2$-transitive on $\{1, \dots, q\}$. Using again Jordan's result, it follows that this group is the full $\Alt(q)$. This proves the claim. 

Consider now the generating pair $S = (a, t)$ for $G$, where $t = b^p$. For $n < q/p$, the ball of radius $n$ around the identity in $\Cay(G, S)$ is isomorphic to the ball of radius $n$ in the Cayley graph of  $C_p \wr \mathbf Z$  with respect to the natural generating pair. The desired assertion follows by letting the prime $q$ tend to infinity. 

For $p=2$, one defines $a =(1,2)(3,4)$ and uses similar considerations. 

Finally, the result for $\mathbf Z \wr \mathbf Z$ follows since the latter group is a limit of $C_p \wr \mathbf Z$ for $p$ tending to infinity. 
\end{proof}

Notice that Proposition~\ref{prop:Cor} together with Lemma~\ref{lem:FP} yields a proof of the well-known fact that the lamplighter group $C_p \wr \mathbf Z$ is not finitely presented.

Problem~\ref{prob:limits} is quite natural in its own right. Moreover, a negative answer to Problem~\ref{prob:limits}(ii)  would yield a negative answer to Problem~\ref{prob:hyp_RF}. In order to see that, let us recall another well known and easy fact. 

\begin{lem}\label{lem:FullyResidually}
Let $G$ be a group and $\mathcal P$ be a collection of groups. Assume that $G$ is fully residually in $\mathcal P$. 

Then for any $d$-generator marking $(G, S)$, there is a sequence of $d$-generated marked groups $(H_n, S_n)$ with $H_n \in \mathcal P$ which converges to $(G,S)$ in the space of marked groups. 
\end{lem}

In view of Corollary~\ref{cor:Olsh4}, we obtain:

\begin{cor}\label{cor:Olsh5}
If all hyperbolic groups are residually finite, then for every non-elementary hyperbolic group $G$ with $E_G(G) = \{1\}$, every $d$-generator marking $(G,S)$  is a limit a non-abelian finite simple groups (containing arbitrarily alternating groups) in the space of $d$-generated marked groups. 
\end{cor}

A crucial  point to underline is that, in the space of marked groups, limits of hyperbolic groups can be quite wild: they can be infinite groups of finite exponent by \cite{IvanovOlshanskii}. Thus, if all hyperbolic groups were residually finite, then Burnside groups would arise as limits of non-abelian finite simple groups. 

\subsection{Examples of fully residually finite simple groups}

Currently,  rather little is known on limits of non-abelian finite simple groups in the space of marked groups. We devote  this subsection to a list of a few families of groups that are fully residually finite simple. Those are thus limits of finite simple groups with respect to any choice of marking in view of Lemma~\ref{lem:FullyResidually}.

\begin{enumerate}[(1)]
\item The free group $F_r$ of rank $r \geq 2$ is fully residually finite simple.  This can be established as follows. First observe that $\PSL_2(\ZZ)$ has a finite index subgroup isomorphic to the free group $F_2$ (this can be seen geometrically, or algebraically using that $\PSL_2(\ZZ)$ is isomorphic to $C_2 * C_3$). In particular, for every $r \geq 2$, the free group $F_r$ is a finite index subgroup of $\PSL_2(\ZZ)$. On the other hand, for any finite set $M \subset \PSL_2(\ZZ)$ and  any sufficiently large prime $p$, the congruence quotient map $\PSL_2(\ZZ) \to \PSL_2(\FF_p)$ is injective on $M$. Since $\PSL_2(\FF_p)$ is generated by elementary matrices, which have order $p$, it cannot act non-trivially on a set with less than $p$ elements. Hence $\PSL_2(\FF_p)$  does not have any proper subgroup of index~$< p$. It follows that the restriction of the quotient map $\PSL_2(\ZZ) \to \PSL_2(\FF_p)$ to a fixed finite index subgroup of $\PSL_2(\ZZ)$ is surjective for all but finitely many primes $p$. It follows that $F_r$ is fully residually finite simple. 

\item The free group $F_r$ of rank $r \geq 2$ is also fully residually in the class $\{\Alt(n) \mid n \geq 5\}$ of alternating groups. Indeed, it follows from \cite{KatzMagnus} that $F_r$ is residually alternating, and the conclusion then follows from Proposition~\ref{prop:Neumann}.   See also \cite{DPSS} and \cite{Wilton_AltQuotients} for various extensions of that result.

\item In view of (1), any group that is fully residually free must also be fully residually finite simple. Fully residually free groups coincide with the so-called \textbf{limit groups} in the sense of Z.~Sela, see \cite{Sela}. That class includes the fundamental groups of closed surfaces of genus~$\geq 2$. 

\item Let $G = \ast_{i \in I} G_i$ be a free product of at least two non-trivial residually finite groups. It is proved in \cite{TamburiniWilson} that $G$ is residually in  $\{\Alt(n) \mid n \geq 5\}$ unless $G$ is the infinite dihedral group. Hence $G$ is \emph{fully} residually in $\{\Alt(n) \mid n \geq 5\}$ by Proposition~\ref{prop:Neumann}.  Various extensions of that result were established in \cite{LiebeckShalev1, LiebeckShalev2}. 

\item The  fundamental group $G$ of a closed $3$-manifold is fully residually finite simple if it is infinite and contained in $\SL_2(\overline {\mathbf Q})$. This  follows from \cite{LongReid} and Proposition~\ref{prop:Neumann}. 

\item For $n \geq 3$, the group $\Out(F_n)$ is residually in $\{\Alt(n) \mid n \geq 5\}$ by \cite{Gilman}, hence fully residually alternating by Proposition~\ref{prop:Neumann} and \cite[\S 8(b)]{Osin_acyl}. 
\end{enumerate}

\subsection{Virtual specialties, after I.~Agol, F.~Haglund and D.~Wise}

The concept of \emph{virtually special} groups was introduced by Haglund--Wise \cite{HaglundWise}. We do not reproduce the definition here; the reader may consult the excellent surveys \cite{Bestvina}, \cite{AFW}, \cite{Sageev_CCC} and \cite{Wilton_CCC}, that cover the basic theory of CAT($0$) cube complexes and virtually special groups, and its relevance to the theory of $3$-manifolds.  The following result is one of the most spectacular breakthroughs in Geometric Group Theory over the past decade. 

\begin{thm}\label{thm:AgolWise}
Let $G$ be a non-elementary hyperbolic group. 
\begin{enumerate}[(i)]
\item {\upshape (I.~Agol \cite{Agol})} If $G$ is capable of acting properly and cocompactly on a {\upshape CAT($0$)} cube complex, then $G$ is virtually special. 

\item {\upshape (Haglund--Wise  \cite{HaglundWise})} If $G$ is virtually special, then:
\begin{itemize}
\item  $G$ is linear over $\ZZ$, hence residually finite;
\item $G$ has a finite index subgroup that maps onto a non-abelian free group;
\item Every quasi-convex subgroup of $G$ is separable. 
\end{itemize}

\end{enumerate}

\end{thm}

Even without bearing in mind what  the formal definition of virtually special groups is, the key feature of Theorem~\ref{thm:AgolWise} should be transparent:  it provides a purely geometric condition on a group $G$ ensuring that $G$ enjoys very strong algebraic properties, namely the existence of a proper cocompact action on a {\upshape CAT($0$)} cube complex.  It is also important to underline that many hyperbolic groups satisfy that geometric condition: this is notably the case of all examples mentioned in Section~\ref{sec:Examples_hyp}, except for the fundamental groups of certain negatively curved closed manifolds. In fact, at the time of this writing, the only known obstruction for an infinite hyperbolic group to act properly and cocompactly on a CAT($0$) cube complex is provided by Kazhdan's property (T) (see \cite{NibloReeves}).

The strikingly broad scope of Theorem~\ref{thm:AgolWise} sheds new light on Problem~\ref{prob:hyp_RF}. Since the publication of the original paper of Haglund--Wise \cite{HaglundWise} where virtually special groups were introduced, numerous constructions and results valid in the category of hyperbolic groups were also established in the virtually special framework: Rips' construction \cite[Theorem~10.1]{HaglundWise}, the Combination Theorem \cite[Theorem~1.2]{HaglundWise_combi}, the Dehn Filling Theorem \cite[Theorem A]{Wise_hierarchy}, \cite{AGM_Dehn}. The following question is thus very natural. 

\begin{prob}\label{prob:SpecialOlshanskii}
\begin{enumerate}[(i)]
\item Is every non-elementary 	virtually special hyperbolic group $G$ with $E_G(G)=\{1\}$ fully residually finite simple? Fully residually alternating? 
	
\item Is it true that any two non-elementary virtually special hyperbolic groups have a common non-elementary virtually special quotient?

\item More generally, is Olshanskii's Theorem~\ref{thm:Olshanskii} valid in the virtually special framework, i.e. can one find a  quotient group $\tilde G$ with the extra property of being virtually special if the initial group $G$ is assumed virtually special? 

\end{enumerate}

\end{prob}

That Problem~\ref{prob:SpecialOlshanskii}(i)  holds in the special case of free amalgamated products of finite groups is closely related to a conjecture of D\v zambi\'c--Jones \cite[p.~206]{DzambicJones}.

A positive solution to these questions would imply the validity of Corollaries~\ref{cor:Olsh}, \ref{cor:Olsh2}, \ref{cor:Olsh3}, \ref{cor:Olsh4} and \ref{cor:Olsh5} in the context of virtually special groups. In particular, if Problem~\ref{prob:SpecialOlshanskii}(iii) had a positive solution, then (i) and (ii) would also do.  Such a result would be highly interesting  from the point of view of hyperbolic groups, as well as from that of asymptotic properties of the finite simple groups.

Coming back to Problem~\ref{prob:hyp_RF}, the previous discussion suggests that the most promising candidates of non-residually finite hyperbolic groups are to be found among the infinite hyperbolic groups with Kazhdan's property (T), since those  are not virtually special. Here is an explicit presentation of such a group, where $[x, y] = xyx^{-1}y^{-1}$. 

\begin{example}\label{ex:EscherGroup}
The group 
\begin{align*}
E = \langle x, y, z, t, r \; | \; & x^7, y^7, [x, y]z^{-1}, [x, z], [y, z], \\
& t^2, r^{73}, trtr, \\
&	[x^{2}yz^{-1}, t], [xyz^3, tr], [x^{3}yz^{2}, tr^{17}], \\
& [x, tr^{-34}], [y, tr^{-32}],  [z, tr^{-29}], \\
&[x^{-2}yz, tr^{-25}], [x^{-1}yz^{-3}, tr^{-19}], [x^{-3}yz^{-2}, tr^{-11}] \rangle,
	\end{align*}
	is an infinite Gromov hyperbolic group enjoying Kazhdan's property (T), see \cite[Theorem~1]{PEC_HypT}.  
\end{example}

The group $E$ has a retraction $E \to \langle x, y\rangle$ that is trivial on $\langle t, r\rangle$, and a retraction $E \to \langle t, r\rangle$ that is trivial on $\langle x, y\rangle$. Their product is a quotient homomorphism of $E$ onto the direct product of the Heisenberg group over $\FF_7$ with the dihedral group of order~$146$; the kernel of that map is torsion-free (see \cite[Theorem~1]{PEC_HypT}). I do not know any larger finite quotient of $E$. In particular, I do not know any non-abelian finite simple quotient of $E$. In view of Corollary~\ref{cor:Olsh2}, a negative solution to the following problem would imply the existence of a non-residually finite hyperbolic group. 

\begin{prob}\label{prob:FiniteSimpleQuotientE}
Does the group $E$ admit finite simple quotients containing arbitrarily large alternating groups? Is $E$ residually finite simple? 
\end{prob}

\section*{Acknowledgements}
I  thank the Isaac Newton Institute for Mathematical Sciences, Cambridge for support and hospitality during the programme \textit{Non-positive curvature group actions and cohomology} where part of the work on this paper was accomplished. The paper is based on mini-courses I gave at the university of Lille in March, at the Winter School on \emph{Arithmetic Groups} in Ein Gedi in March, at the Newton Institute in April, and at the Groups St  Andrews conference in Birmingham in August 2017. I thank the organizers of those events for their hospitality, and the participants for their feedback from which I benefited greatly. I am especially grateful to V.~Guirardel, P.~Neumann, C.~Praeger, N.~Radu for inspiring conversations, and to Y.~Cornulier, J.-P. Tignol and an anonymous referee for their comments and corrections on an earlier version of these notes. Finally, it is a pleasure to acknowledge an intellectual debt to Marc Burger, Shahar Mozes and Dani Wise: my  interest in the topic of these notes grew out of my fascination for their seminal contributions.

%=========================== The bibliography===================================

\newcommand{\etalchar}[1]{$^{#1}$}
\providecommand{\bysame}{\leavevmode\hbox to3em{\hrulefill}\thinspace}
\providecommand{\MR}{\relax\ifhmode\unskip\space\fi MR }
\providecommand{\MRhref}[2]{%
  \href{http://www.ams.org/mathscinet-getitem?mr=#1}{#2}
}
\providecommand{\href}[2]{#2}


\begin{thebibliography}{KMW84b}

\bibitem[AFW15]{AFW}
Matthias Aschenbrenner, Stefan Friedl, and Henry Wilton, \emph{3-manifold
  groups}, EMS Series of Lectures in Mathematics, European Mathematical Society
  (EMS), Z\"urich, 2015. \MR{3444187}

\bibitem[AGM09]{AGM}
Ian Agol, Daniel Groves, and Jason~Fox Manning, \emph{Residual finiteness,
  {QCERF} and fillings of hyperbolic groups}, Geom. Topol. \textbf{13} (2009),
  no.~2, 1043--1073. \MR{2470970}

\bibitem[AGM16]{AGM_Dehn}
\bysame, \emph{An alternate proof of {W}ise's malnormal special quotient
  theorem}, Forum Math. Pi \textbf{4} (2016), e1, 54~pp. \MR{3456181}

\bibitem[Ago13]{Agol}
Ian Agol, \emph{The virtual {H}aken conjecture}, Doc. Math. \textbf{18} (2013),
  1045--1087, With an appendix by Agol, Daniel Groves, and Jason Manning.
  \MR{3104553}

\bibitem[Bau69]{Baumslag}
Gilbert Baumslag, \emph{A non-cyclic one-relator group all of whose finite
  quotients are cyclic}, J. Austral. Math. Soc. \textbf{10} (1969), 497--498.
  \MR{0254127}

\bibitem[BCGM12]{BCGM}
Uri Bader, Pierre-Emmanuel Caprace, Tsachik Gelander, and Shahar Mozes,
  \emph{Simple groups without lattices}, Bull. Lond. Math. Soc. \textbf{44}
  (2012), no.~1, 55--67. \MR{2881324}

\bibitem[BCL16]{BCL}
Uri Bader, Pierre-Emmanuel Caprace, and Jean L\'ecureux, \emph{On the linearity
  of lattices in affine buildings and ergodicity of the singular cartan flow},
  Preprint, arXiv:1608.06265, 2016.

\bibitem[BdlHV08]{BHV}
Bachir Bekka, Pierre de~la Harpe, and Alain Valette, \emph{Kazhdan's property
  ({T})}, New Mathematical Monographs, vol.~11, Cambridge University Press,
  Cambridge, 2008. \MR{2415834}

\bibitem[BDR16]{BDR}
Ievgen Bondarenko, Daniele D'Angeli, and Emanuele Rodaro, \emph{The lamplighter
  group {$\mathbf{Z}_3\wr \mathbf{Z}$} generated by a bireversible automaton},
  Comm. Algebra \textbf{44} (2016), no.~12, 5257--5268. \MR{3520274}

\bibitem[Bel07]{Belol}
Mikhail Belolipetsky, \emph{Counting maximal arithmetic subgroups}, Duke Math.
  J. \textbf{140} (2007), no.~1, 1--33, With an appendix by Jordan Ellenberg
  and Akshay Venkatesh. \MR{2355066}

\bibitem[Bes14]{Bestvina}
Mladen Bestvina, \emph{Geometric group theory and 3-manifolds hand in hand: the
  fulfillment of {T}hurston's vision}, Bull. Amer. Math. Soc. (N.S.)
  \textbf{51} (2014), no.~1, 53--70. \MR{3119822}

\bibitem[BGLM02]{BGLM}
M.~Burger, T.~Gelander, A.~Lubotzky, and S.~Mozes, \emph{Counting hyperbolic
  manifolds}, Geom. Funct. Anal. \textbf{12} (2002), no.~6, 1161--1173.
  \MR{1952926}

\bibitem[BH99]{BH}
Martin~R. Bridson and Andr\'e Haefliger, \emph{Metric spaces of non-positive
  curvature}, Grundlehren der Mathematischen Wissenschaften [Fundamental
  Principles of Mathematical Sciences], vol. 319, Springer-Verlag, Berlin,
  1999. \MR{1744486}

\bibitem[Bha94]{Bhatt}
Meenaxi Bhattacharjee, \emph{Constructing finitely presented infinite nearly
  simple groups}, Comm. Algebra \textbf{22} (1994), no.~11, 4561--4589.
  \MR{1284345}

\bibitem[BK]{BondarenkoKivva}
Ievgen Bondarenko and Bohdan Kivva, \emph{Automaton groups and complete square
  complexes}, Preprint arXiv:1707.00215.

\bibitem[BL01]{BassLubotzky}
Hyman Bass and Alexander Lubotzky, \emph{Tree lattices}, Progress in
  Mathematics, vol. 176, Birkh\"auser Boston, Inc., Boston, MA, 2001, With
  appendices by Bass, L. Carbone, Lubotzky, G. Rosenberg and J. Tits.
  \MR{1794898}

\bibitem[BM97]{BM_CRAS}
Marc Burger and Shahar Mozes, \emph{Finitely presented simple groups and
  products of trees}, C. R. Acad. Sci. Paris S\'er. I Math. \textbf{324}
  (1997), no.~7, 747--752. \MR{1446574}

\bibitem[BM00a]{BuMo1}
\bysame, \emph{Groups acting on trees: from local to global structure}, Inst.
  Hautes \'Etudes Sci. Publ. Math. (2000), no.~92, 113--150 (2001).
  \MR{1839488}

\bibitem[BM00b]{BuMo2}
\bysame, \emph{Lattices in product of trees}, Inst. Hautes \'Etudes Sci. Publ.
  Math. (2000), no.~92, 151--194 (2001). \MR{1839489}

\bibitem[BMZ09]{BMZ}
Marc Burger, Shahar Mozes, and Robert~J. Zimmer, \emph{Linear representations
  and arithmeticity of lattices in products of trees}, Essays in geometric
  group theory, Ramanujan Math. Soc. Lect. Notes Ser., vol.~9, Ramanujan Math.
  Soc., Mysore, 2009, pp.~1--25. \MR{2605353}

\bibitem[BQ]{BleakQuick}
Collin Bleak and Martyn Quick, \emph{The infinite simple group {V} of {R}ichard
  {J}. {T}hompson: presentations by permutations}, Preprint, arXiv:1511.02123.

\bibitem[Bro92]{Brown}
Kenneth~S. Brown, \emph{The geometry of finitely presented infinite simple
  groups}, Algorithms and classification in combinatorial group theory
  ({B}erkeley, {CA}, 1989), Math. Sci. Res. Inst. Publ., vol.~23, Springer, New
  York, 1992, pp.~121--136. \MR{1230631}

\bibitem[BS06]{BaSha}
Uri Bader and Yehuda Shalom, \emph{Factor and normal subgroup theorems for
  lattices in products of groups}, Invent. Math. \textbf{163} (2006), no.~2,
  415--454. \MR{2207022}

\bibitem[Cam53]{Camm}
Ruth Camm, \emph{Simple free products}, J. London Math. Soc. \textbf{28}
  (1953), 66--76. \MR{0052420}

\bibitem[Cap17]{PEC_HypT}
Pierre-Emmanuel Caprace, \emph{A sixteen-relator presentation of an infinite
  hyperbolic {K}azhdan group}, Preprint arXiv:1708.09772, 2017.

\bibitem[CCMT15]{CCMT}
Pierre-Emmanuel Caprace, Yves Cornulier, Nicolas Monod, and Romain Tessera,
  \emph{Amenable hyperbolic groups}, J. Eur. Math. Soc. (JEMS) \textbf{17}
  (2015), no.~11, 2903--2947. \MR{3420526}

\bibitem[CDP90]{CDP}
M.~Coornaert, T.~Delzant, and A.~Papadopoulos, \emph{G\'eom\'etrie et th\'eorie
  des groupes}, Lecture Notes in Mathematics, vol. 1441, Springer-Verlag,
  Berlin, 1990, Les groupes hyperboliques de Gromov. [Gromov hyperbolic
  groups], With an English summary. \MR{1075994}

\bibitem[CF10]{CaFu}
Pierre-Emmanuel Caprace and Koji Fujiwara, \emph{Rank-one isometries of
  buildings and quasi-morphisms of {K}ac-{M}oody groups}, Geom. Funct. Anal.
  \textbf{19} (2010), no.~5, 1296--1319. \MR{2585575}

\bibitem[CFH{\etalchar{+}}15]{CFHKSSZZ}
Ted Chinburg, Holley Friedlander, Sean Howe, Michiel Kosters, Bhairav Singh,
  Matthew Stover, Ying Zhang, and Paul Ziegler, \emph{Presentations for
  quaternionic {$S$}-unit groups}, Exp. Math. \textbf{24} (2015), no.~2,
  175--182. \MR{3350524}

\bibitem[CFP96]{CFP}
J.~W. Cannon, W.~J. Floyd, and W.~R. Parry, \emph{Introductory notes on
  {R}ichard {T}hompson's groups}, Enseign. Math. (2) \textbf{42} (1996),
  no.~3-4, 215--256. \MR{1426438}

\bibitem[CG05]{ChamGuir}
Christophe Champetier and Vincent Guirardel, \emph{Limit groups as limits of
  free groups}, Israel J. Math. \textbf{146} (2005), 1--75. \MR{2151593}

\bibitem[CKRW]{CKRW}
Pierre-Emmanuel Caprace, Peter~H. Kropholler, Colin~D. Reid, and Phillip
  Wesolek, \emph{On the residual and profinite closures of commensurated
  subgroups}, Preprint arXiv:1706.06853.

\bibitem[CLB]{CapLB}
Pierre-Emmanuel Caprace and Adrien Le~Boudec, \emph{Bounding the covolume of
  lattices in products}, Preprint arXiv:1805.04469.

\bibitem[CM11]{CaMo_cpt}
Pierre-Emmanuel Caprace and Nicolas Monod, \emph{Decomposing locally compact
  groups into simple pieces}, Math. Proc. Cambridge Philos. Soc. \textbf{150}
  (2011), no.~1, 97--128. \MR{2739075}

\bibitem[CM12]{CaMo_KM}
\bysame, \emph{A lattice in more than two {K}ac-{M}oody groups is arithmetic},
  Israel J. Math. \textbf{190} (2012), 413--444. \MR{2956249}

\bibitem[Cor92]{Corlette}
Kevin Corlette, \emph{Archimedean superrigidity and hyperbolic geometry}, Ann.
  of Math. (2) \textbf{135} (1992), no.~1, 165--182. \MR{1147961}

\bibitem[CR09]{CaRe}
Pierre-Emmanuel Caprace and Bertrand R\'emy, \emph{Simplicity and superrigidity
  of twin building lattices}, Invent. Math. \textbf{176} (2009), no.~1,
  169--221. \MR{2485882}

\bibitem[CS14]{ChinburgStover}
Ted Chinburg and Matthew Stover, \emph{Small generators for {$S$}-unit groups
  of division algebras}, New York J. Math. \textbf{20} (2014), 1175--1202.
  \MR{3291615}

\bibitem[CW]{CapraceWesolek}
Pierre-Emmanuel Caprace and Phillip Wesolek, \emph{Indicability, residual
  finiteness, and simple subquotients of groups acting on trees}, Preprint
  arXiv:1708.04590.

\bibitem[Dav15]{Davis_book}
Michael~W. Davis, \emph{The geometry and topology of {C}oxeter groups},
  Introduction to modern mathematics, Adv. Lect. Math. (ALM), vol.~33, Int.
  Press, Somerville, MA, 2015, pp.~129--142. \MR{3445448}

\bibitem[dCGP07]{CGP}
Yves de~Cornulier, Luc Guyot, and Wolfgang Pitsch, \emph{On the isolated points
  in the space of groups}, J. Algebra \textbf{307} (2007), no.~1, 254--277.
  \MR{2278053}

\bibitem[Del96]{Delzant96}
Thomas Delzant, \emph{Sous-groupes distingu\'es et quotients des groupes
  hyperboliques}, Duke Math. J. \textbf{83} (1996), no.~3, 661--682.
  \MR{1390660}

\bibitem[DGO17]{DGO}
F.~Dahmani, V.~Guirardel, and D.~Osin, \emph{Hyperbolically embedded subgroups
  and rotating families in groups acting on hyperbolic spaces}, Mem. Amer.
  Math. Soc. \textbf{245} (2017), no.~1156, v+152. \MR{3589159}

\bibitem[DPP16]{DPP}
Martin Deraux, John~R. Parker, and Julien Paupert, \emph{New non-arithmetic
  complex hyperbolic lattices}, Invent. Math. \textbf{203} (2016), no.~3,
  681--771. \MR{3461365}

\bibitem[DPSS03]{DPSS}
John~D. Dixon, L\'aszl\'o Pyber, \'Akos Seress, and Aner Shalev, \emph{Residual
  properties of free groups and probabilistic methods}, J. Reine Angew. Math.
  \textbf{556} (2003), 159--172. \MR{1971144}

\bibitem[DSV03]{DSV}
Giuliana Davidoff, Peter Sarnak, and Alain Valette, \emph{Elementary number
  theory, group theory, and {R}amanujan graphs}, London Mathematical Society
  Student Texts, vol.~55, Cambridge University Press, Cambridge, 2003.
  \MR{1989434}

\bibitem[DzJ13]{DzambicJones}
Amir D\v~zambi\'c and Gareth~A. Jones, \emph{{$p$}-adic {H}urwitz groups}, J.
  Algebra \textbf{379} (2013), 179--207. \MR{3019251}

\bibitem[GdlH90]{GhysHarpe}
\'E. Ghys and P.~de~la Harpe (eds.), \emph{Sur les groupes hyperboliques
  d'apr\`es {M}ikhael {G}romov}, Progress in Mathematics, vol.~83, Birkh\"auser
  Boston, Inc., Boston, MA, 1990, Papers from the Swiss Seminar on Hyperbolic
  Groups held in Bern, 1988. \MR{1086648}

\bibitem[Gil77]{Gilman}
Robert Gilman, \emph{Finite quotients of the automorphism group of a free
  group}, Canad. J. Math. \textbf{29} (1977), no.~3, 541--551. \MR{0435226}

\bibitem[GL14]{GelanderLevit}
Tsachik Gelander and Arie Levit, \emph{Counting commensurability classes of
  hyperbolic manifolds}, Geom. Funct. Anal. \textbf{24} (2014), no.~5,
  1431--1447. \MR{3261631}

\bibitem[GM05]{GlasnerMozes}
Yair Glasner and Shahar Mozes, \emph{Automata and square complexes}, Geom.
  Dedicata \textbf{111} (2005), 43--64. \MR{2155175}

\bibitem[Gri84]{Gri}
R.~I. Grigorchuk, \emph{Degrees of growth of finitely generated groups and the
  theory of invariant means}, Izv. Akad. Nauk SSSR Ser. Mat. \textbf{48}
  (1984), no.~5, 939--985. \MR{764305}

\bibitem[Gri00]{Grigorchuk_branch}
\bysame, \emph{Just infinite branch groups}, New horizons in pro-{$p$} groups,
  Progr. Math., vol. 184, Birkh\"auser Boston, Boston, MA, 2000, pp.~121--179.
  \MR{1765119}

\bibitem[Gro81a]{Gromov81}
M.~Gromov, \emph{Hyperbolic manifolds, groups and actions}, Riemann surfaces
  and related topics: {P}roceedings of the 1978 {S}tony {B}rook {C}onference
  ({S}tate {U}niv. {N}ew {Y}ork, {S}tony {B}rook, {N}.{Y}., 1978), Ann. of
  Math. Stud., vol.~97, Princeton Univ. Press, Princeton, N.J., 1981,
  pp.~183--213. \MR{624814}

\bibitem[Gro81b]{Gromov_poly}
Mikhael Gromov, \emph{Groups of polynomial growth and expanding maps}, Inst.
  Hautes \'Etudes Sci. Publ. Math. (1981), no.~53, 53--73. \MR{623534}

\bibitem[Gro87]{Gromov_hyp}
M.~Gromov, \emph{Hyperbolic groups}, Essays in group theory, Math. Sci. Res.
  Inst. Publ., vol.~8, Springer, New York, 1987, pp.~75--263. \MR{919829}

\bibitem[GS92]{GromovSchoen}
Mikhail Gromov and Richard Schoen, \emph{Harmonic maps into singular spaces and
  {$p$}-adic superrigidity for lattices in groups of rank one}, Inst. Hautes
  \'Etudes Sci. Publ. Math. (1992), no.~76, 165--246. \MR{1215595}

\bibitem[Hig51]{Higman}
Graham Higman, \emph{A finitely generated infinite simple group}, J. London
  Math. Soc. \textbf{26} (1951), 61--64. \MR{0038348}

\bibitem[Hig74]{Higman_Thompson}
\bysame, \emph{Finitely presented infinite simple groups}, Department of Pure
  Mathematics, Department of Mathematics, I.A.S. Australian National
  University, Canberra, 1974, Notes on Pure Mathematics, No. 8 (1974).
  \MR{0376874}

\bibitem[HNN49]{HNN}
Graham Higman, B.~H. Neumann, and Hanna Neumann, \emph{Embedding theorems for
  groups}, J. London Math. Soc. \textbf{24} (1949), 247--254. \MR{0032641}

\bibitem[HW98]{HsuWise}
Tim Hsu and Daniel~T. Wise, \emph{Embedding theorems for non-positively curved
  polygons of finite groups}, J. Pure Appl. Algebra \textbf{123} (1998),
  no.~1-3, 201--221. \MR{1492901}

\bibitem[HW08]{HaglundWise}
Fr\'ed\'eric Haglund and Daniel~T. Wise, \emph{Special cube complexes}, Geom.
  Funct. Anal. \textbf{17} (2008), no.~5, 1551--1620. \MR{2377497}

\bibitem[HW12]{HaglundWise_combi}
\bysame, \emph{A combination theorem for special cube complexes}, Ann. of Math.
  (2) \textbf{176} (2012), no.~3, 1427--1482. \MR{2979855}

\bibitem[IO96]{IvanovOlshanskii}
S.~V. Ivanov and A.~Yu. Ol'shanski\u\i, \emph{Hyperbolic groups and their
  quotients of bounded exponents}, Trans. Amer. Math. Soc. \textbf{348} (1996),
  no.~6, 2091--2138. \MR{1327257}

\bibitem[Jac89]{JacobsonII}
Nathan Jacobson, \emph{Basic algebra. {II}}, second ed., W. H. Freeman and
  Company, New York, 1989. \MR{1009787}

\bibitem[JW09]{JW}
David Janzen and Daniel~T. Wise, \emph{A smallest irreducible lattice in the
  product of trees}, Algebr. Geom. Topol. \textbf{9} (2009), no.~4, 2191--2201.
  \MR{2558308}

\bibitem[Kin17]{King}
Carlisle S.~H. King, \emph{Generation of finite simple groups by an involution
  and an element of prime order}, J. Algebra \textbf{478} (2017), 153--173.
  \MR{3621666}

\bibitem[KM68]{KatzMagnus}
Robert~A. Katz and Wilhelm Magnus, \emph{Residual properties of free groups},
  Comm. Pure Appl. Math. \textbf{22} (1968), 1--13. \MR{0233873}

\bibitem[KMW84a]{KMW84}
Peter K\"ohler, Thomas Meixner, and Michael Wester, \emph{The affine building
  of type {$\tilde A_{2}$}\ over a local field of characteristic two}, Arch.
  Math. (Basel) \textbf{42} (1984), no.~5, 400--407. \MR{756691}

\bibitem[KMW84b]{KMW_triangle}
\bysame, \emph{Triangle groups}, Comm. Algebra \textbf{12} (1984), no.~13-14,
  1595--1625. \MR{743306}

\bibitem[KMW85]{KMW85}
\bysame, \emph{The {$2$}-adic affine building of type {$\tilde A_2$} and its
  finite projections}, J. Combin. Theory Ser. A \textbf{38} (1985), no.~2,
  203--209. \MR{784716}

\bibitem[Kne56]{Kneser56}
Martin Kneser, \emph{Orthogonale {G}ruppen \"uber algebraischen
  {Z}ahlk\"orpern}, J. Reine Angew. Math. \textbf{196} (1956), 213--220.
  \MR{0080101}

\bibitem[KR02]{KR}
Jason~S. Kimberley and Guyan Robertson, \emph{Groups acting on products of
  trees, tiling systems and analytic {$K$}-theory}, New York J. Math.
  \textbf{8} (2002), 111--131. \MR{1923572}

\bibitem[Kur44]{Kurosh}
A.~G. Kuro\v{s}, \emph{Teoriya {G}rupp}, OGIZ, Moscow-Leningrad, 1944.
  \MR{0022843}

\bibitem[KW00]{KaWi}
Ilya Kapovich and Daniel~T. Wise, \emph{The equivalence of some residual
  properties of word-hyperbolic groups}, J. Algebra \textbf{223} (2000), no.~2,
  562--583. \MR{1735163}

\bibitem[LB16]{LeBoudec}
Adrien Le~Boudec, \emph{Groups acting on trees with almost prescribed local
  action}, Comment. Math. Helv. \textbf{91} (2016), no.~2, 253--293.
  \MR{3493371}

\bibitem[Lie13]{Liebeck_survey}
Martin~W. Liebeck, \emph{Probabilistic and asymptotic aspects of finite simple
  groups}, Probabilistic group theory, combinatorics, and computing, Lecture
  Notes in Math., vol. 2070, Springer, London, 2013, pp.~1--34. \MR{3026185}

\bibitem[LR98]{LongReid}
D.~D. Long and A.~W. Reid, \emph{Simple quotients of hyperbolic {$3$}-manifold
  groups}, Proc. Amer. Math. Soc. \textbf{126} (1998), no.~3, 877--880.
  \MR{1459136}

\bibitem[LR05]{LongReid_2005}
Darren Long and Alan~W. Reid, \emph{Surface subgroups and subgroup separability
  in 3-manifold topology}, Publica\c{c}\~oes Matem\'aticas do IMPA. [IMPA
  Mathematical Publications], Instituto Nacional de Matem\'atica Pura e
  Aplicada (IMPA), Rio de Janeiro, 2005, 25${^{{}}{\rm{o}}}$ Col\'oquio
  Brasileiro de Matem\'atica. [25th Brazilian Mathematics Colloquium].
  \MR{2164951}

\bibitem[LS03a]{LiebeckShalev2}
Martin~W. Liebeck and Aner Shalev, \emph{Residual properties of free products
  of finite groups}, J. Algebra \textbf{268} (2003), no.~1, 286--289.
  \MR{2005288}

\bibitem[LS03b]{LiebeckShalev1}
\bysame, \emph{Residual properties of the modular group and other free
  products}, J. Algebra \textbf{268} (2003), no.~1, 264--285. \MR{2005287}

\bibitem[Lub05]{Lubot}
Alexander Lubotzky, \emph{Some more non-arithmetic rigid groups}, Geometry,
  spectral theory, groups, and dynamics, Contemp. Math., vol. 387, Amer. Math.
  Soc., Providence, RI, 2005, pp.~237--244. \MR{2180210}

\bibitem[Mar79]{Margulis79}
G.~A. Margulis, \emph{Finiteness of quotient groups of discrete subgroups},
  Funktsional. Anal. i Prilozhen. \textbf{13} (1979), no.~3, 28--39.
  \MR{545365}

\bibitem[Mar80]{Margulis80}
\bysame, \emph{Multiplicative groups of a quaternion algebra over a global
  field}, Dokl. Akad. Nauk SSSR \textbf{252} (1980), no.~3, 542--546.
  \MR{577836}

\bibitem[Mar91]{Margulis_book}
\bysame, \emph{Discrete subgroups of semisimple {L}ie groups}, Ergebnisse der
  Mathematik und ihrer Grenzgebiete (3) [Results in Mathematics and Related
  Areas (3)], vol.~17, Springer-Verlag, Berlin, 1991. \MR{1090825}

\bibitem[Mar17]{Martin_Higman}
Alexandre Martin, \emph{On the cubical geometry of {H}igman's group}, Duke
  Math. J. \textbf{166} (2017), no.~4, 707--738. \MR{3619304}

\bibitem[MO15]{MinasyanOsin}
Ashot Minasyan and Denis Osin, \emph{Acylindrical hyperbolicity of groups
  acting on trees}, Math. Ann. \textbf{362} (2015), no.~3-4, 1055--1105.
  \MR{3368093}

\bibitem[Moz92]{Mozes92}
Shahar Mozes, \emph{A zero entropy, mixing of all orders tiling system},
  Symbolic dynamics and its applications ({N}ew {H}aven, {CT}, 1991), Contemp.
  Math., vol. 135, Amer. Math. Soc., Providence, RI, 1992, pp.~319--325.
  \MR{1185097}

\bibitem[{Neu}37]{Neumann37}
B.~H. {Neumann}, \emph{{Some remarks on infinite groups.}}, {J. Lond. Math.
  Soc.} \textbf{12} (1937), 120--127.

\bibitem[Neu73]{Neumann73}
Peter~M. Neumann, \emph{The {$SQ$}-universality of some finitely presented
  groups}, J. Austral. Math. Soc. \textbf{16} (1973), 1--6, Collection of
  articles dedicated to the memory of Hanna Neumann, I. \MR{0333017}

\bibitem[NN53]{NN53}
B.~H. Neumann and Hanna Neumann, \emph{A contribution to the embedding theory
  of group amalgams}, Proc. London Math. Soc. (3) \textbf{3} (1953), 243--256.
  \MR{0057880}

\bibitem[NR97]{NibloReeves}
Graham Niblo and Lawrence Reeves, \emph{Groups acting on {${\rm CAT}(0)$} cube
  complexes}, Geom. Topol. \textbf{1} (1997), 1--7. \MR{1432323}

\bibitem[Ol'93]{Olsh_residualing}
A.~Yu. Ol'shanski\u\i, \emph{On residualing homomorphisms and {$G$}-subgroups
  of hyperbolic groups}, Internat. J. Algebra Comput. \textbf{3} (1993), no.~4,
  365--409. \MR{1250244}

\bibitem[Ol'95]{Olsh_SQ}
\bysame, \emph{{${\rm SQ}$}-universality of hyperbolic groups}, Mat. Sb.
  \textbf{186} (1995), no.~8, 119--132. \MR{1357360}

\bibitem[Ol'00]{Olsh_2000}
\bysame, \emph{On the {B}ass-{L}ubotzky question about quotients of hyperbolic
  groups}, J. Algebra \textbf{226} (2000), no.~2, 807--817. \MR{1752761}

\bibitem[Osi16]{Osin_acyl}
D.~Osin, \emph{Acylindrically hyperbolic groups}, Trans. Amer. Math. Soc.
  \textbf{368} (2016), no.~2, 851--888. \MR{3430352}

\bibitem[Pla75]{Platonov74}
V.~P. Platonov, \emph{Arithmetic and structural problems in linear algebraic
  groups}, pp.~471--478, Anad. Math. Ongress, Montreal, Que., 1975.
  \MR{0466334}

\bibitem[PR94]{PlatonovRapinchuk}
Vladimir Platonov and Andrei Rapinchuk, \emph{Algebraic groups and number
  theory}, Pure and Applied Mathematics, vol. 139, Academic Press, Inc.,
  Boston, MA, 1994, Translated from the 1991 Russian original by Rachel Rowen.
  \MR{1278263}

\bibitem[Rad]{Radu_SimpleLatt}
Nicolas Radu, \emph{New simple lattices in products of trees and their
  projections}, Preprint arXiv:1712.01091.

\bibitem[Rad17a]{Radu_InvMath}
\bysame, \emph{A classification theorem for boundary 2-transitive automorphism
  groups of trees}, Invent. Math. \textbf{209} (2017), no.~1, 1--60.
  \MR{3660305}

\bibitem[Rad17b]{Radu_phd}
\bysame, \emph{Ph{D} thesis, {U}niversit\'e catholique de {L}ouvain}, In
  preparation, 2017.

\bibitem[Rag07]{Raghu}
M.~S. Raghunathan, \emph{Discrete subgroups of {L}ie groups}, Math. Student
  (2007), no.~Special Centenary Volume, 59--70 (2008). \MR{2527560}

\bibitem[Rap06]{Rapinchuk06}
Andrei~S. Rapinchuk, \emph{The {M}argulis-{P}latonov conjecture for {${\rm
  SL}_{1,D}$} and 2-generation of finite simple groups}, Math. Z. \textbf{252}
  (2006), no.~2, 295--313. \MR{2207799}

\bibitem[Rat04]{Rattaggi_phd}
Diego Rattaggi, \emph{Computations in groups acting on a product of trees:
  {N}ormal subgroup structures and quaternion lattices}, ProQuest LLC, Ann
  Arbor, MI, 2004, Thesis (Dr.sc.math.)--Eidgenoessische Technische Hochschule
  Zuerich (Switzerland). \MR{2715704}

\bibitem[Rat05]{Rattaggi_GeomDed}
\bysame, \emph{Anti-tori in square complex groups}, Geom. Dedicata \textbf{114}
  (2005), 189--207. \MR{2174099}

\bibitem[Rat07]{Rattaggi_JGT}
\bysame, \emph{A finitely presented torsion-free simple group}, J. Group Theory
  \textbf{10} (2007), no.~3, 363--371. \MR{2320973}

\bibitem[Ron84]{Ronan}
M.~A. Ronan, \emph{Triangle geometries}, J. Combin. Theory Ser. A \textbf{37}
  (1984), no.~3, 294--319. \MR{769219}

\bibitem[RP96]{RapPot}
Andrei Rapinchuk and Alexander Potapchik, \emph{Normal subgroups of {${\rm
  SL}_{1,D}$} and the classification of finite simple groups}, Proc. Indian
  Acad. Sci. Math. Sci. \textbf{106} (1996), no.~4, 329--368. \MR{1425612}

\bibitem[RSS02]{RPP}
Andrei~S. Rapinchuk, Yoav Segev, and Gary~M. Seitz, \emph{Finite quotients of
  the multiplicative group of a finite dimensional division algebra are
  solvable}, J. Amer. Math. Soc. \textbf{15} (2002), no.~4, 929--978.
  \MR{1915823}

\bibitem[Run]{Rungtana}
Nithi Rungtanapirom, \emph{Quaternionic arithmetic lattices of rank 2 and a
  fake quadric in characteristic 2}, Preprint arXiv:1707.09925.

\bibitem[Sag14]{Sageev_CCC}
Michah Sageev, \emph{{$\rm CAT(0)$} cube complexes and groups}, Geometric group
  theory, IAS/Park City Math. Ser., vol.~21, Amer. Math. Soc., Providence, RI,
  2014, pp.~7--54. \MR{3329724}

\bibitem[Sch71]{Schupp71}
Paul~E. Schupp, \emph{Small cancellation theory over free products with
  amalgamation}, Math. Ann. \textbf{193} (1971), 255--264. \MR{0291298}

\bibitem[Sch73]{Schupp_survey}
\bysame, \emph{A survey of small cancellation theory}, 569--589. Studies in
  Logic and the Foundations of Math., Vol. 71. \MR{0412289}

\bibitem[Sel01]{Sela}
Zlil Sela, \emph{Diophantine geometry over groups. {I}. {M}akanin-{R}azborov
  diagrams}, Publ. Math. Inst. Hautes \'Etudes Sci. (2001), no.~93, 31--105.
  \MR{1863735}

\bibitem[Sha00]{Shalom00}
Yehuda Shalom, \emph{Rigidity of commensurators and irreducible lattices},
  Invent. Math. \textbf{141} (2000), no.~1, 1--54. \MR{1767270}

\bibitem[SS02]{SegevSeitz}
Yoav Segev and Gary~M. Seitz, \emph{Anisotropic groups of type {$A_n$} and the
  commuting graph of finite simple groups}, Pacific J. Math. \textbf{202}
  (2002), no.~1, 125--225. \MR{1883974}

\bibitem[SV17]{StixVdovina}
Jakob Stix and Alina Vdovina, \emph{Simply transitive quaternionic lattices of
  rank 2 over {$\Bbb F_q(t)$} and a non-classical fake quadric}, Math. Proc.
  Cambridge Philos. Soc. \textbf{163} (2017), no.~3, 453--498. \MR{3708519}

\bibitem[Tit86]{Tits85}
Jacques Tits, \emph{Buildings and group amalgamations}, Proceedings of
  groups---{S}t.\ {A}ndrews 1985, London Math. Soc. Lecture Note Ser., vol.
  121, Cambridge Univ. Press, Cambridge, 1986, pp.~110--127. \MR{896503}

\bibitem[Tit13]{Tits_cours}
\bysame, \emph{R\'esum\'es des cours au {C}oll\`ege de {F}rance 1973--2000},
  Documents Math\'ematiques (Paris) [Mathematical Documents (Paris)], vol.~12,
  Soci\'et\'e Math\'ematique de France, Paris, 2013. \MR{3235648}

\bibitem[TW84]{TamburiniWilson}
Chiara Tamburini and John~S. Wilson, \emph{A residual property of certain free
  products}, Math. Z. \textbf{186} (1984), no.~4, 525--530. \MR{744963}

\bibitem[TW95]{TrofimovWeiss}
V.~I. Trofimov and R.~M. Weiss, \emph{Graphs with a locally linear group of
  automorphisms}, Math. Proc. Cambridge Philos. Soc. \textbf{118} (1995),
  no.~2, 191--206. \MR{1341785}

\bibitem[Vig80]{Vigneras}
Marie-France Vign\'eras, \emph{Arithm\'etique des alg\`ebres de quaternions},
  Lecture Notes in Mathematics, vol. 800, Springer, Berlin, 1980. \MR{580949}

\bibitem[Wan72]{Wang}
Hsien~Chung Wang, \emph{Topics on totally discontinuous groups}, 459--487. Pure
  and Appl. Math., Vol. 8. \MR{0414787}

\bibitem[Wei79]{Weiss79}
Richard Weiss, \emph{Groups with a {$(B,\,N)$}-pair and locally transitive
  graphs}, Nagoya Math. J. \textbf{74} (1979), 1--21. \MR{535958}

\bibitem[Wie64]{Wielandt}
Helmut Wielandt, \emph{Finite permutation groups}, Translated from the German
  by R. Bercov, Academic Press, New York-London, 1964. \MR{0183775}

\bibitem[Wil71]{Wilson71}
John~S. Wilson, \emph{Groups with every proper quotient finite}, Proc.
  Cambridge Philos. Soc. \textbf{69} (1971), 373--391. \MR{0274575}

\bibitem[Wil00]{Wilson_JustInfinite}
\bysame, \emph{On just infinite abstract and profinite groups}, New horizons in
  pro-{$p$} groups, Progr. Math., vol. 184, Birkh\"auser Boston, Boston, MA,
  2000, pp.~181--203. \MR{1765120}

\bibitem[Wil11]{Wilton_CCC}
Henry Wilton, \emph{Non-positively curved cube complexes}, Course notes,
  available as AMS Open Math Notes:201704.110697, 2011.

\bibitem[Wil12]{Wilton_AltQuotients}
\bysame, \emph{Alternating quotients of free groups}, Enseign. Math. (2)
  \textbf{58} (2012), no.~1-2, 49--60. \MR{2985009}

\bibitem[Wis96]{Wise_phd}
Daniel~T. Wise, \emph{Non-positively curved squared complexes: {A}periodic
  tilings and non-residually finite groups}, ProQuest LLC, Ann Arbor, MI, 1996,
  Thesis (Ph.D.)--Princeton University. \MR{2694733}

\bibitem[Wis06]{Wise_fig8}
\bysame, \emph{Subgroup separability of the figure 8 knot group}, Topology
  \textbf{45} (2006), no.~3, 421--463. \MR{2218750}

\bibitem[Wis07]{Wise_CSC}
\bysame, \emph{Complete square complexes}, Comment. Math. Helv. \textbf{82}
  (2007), no.~4, 683--724. \MR{2341837}

\bibitem[Wis12]{Wise_hierarchy}
\bysame, \emph{The structure of groups with a quasi-convex hierarchy}, Preprint
  available at
  \texttt{https://www.math.u-psud.fr/\~{}haglund/Hierarchy29Feb2012.pdf}, 2012.

\end{thebibliography}
\end{document}